\documentclass[11pt, a4paper]{amsart}
\usepackage{amsmath,amssymb,amstext,amsgen,amsbsy,amsopn,amsfonts,bbm,graphicx,amsthm,hyperref,cleveref,mathrsfs,mathtools,todonotes}
\usepackage{enumitem}

\usepackage[left=3cm, right=3cm, top=3.5cm, bottom=3cm,footskip=1cm,headsep=1.5cm]{geometry}
\usepackage[all]{xy}
\usepackage{tikz-cd}
\usepackage[OT2,T1]{fontenc}
\usepackage{float}
\usepackage[alphabetic]{amsrefs}

\crefname{equation}{}{}
\numberwithin{equation}{section}

\crefname{section}{\S\!\!}{\S\!\!}
\Crefname{section}{\S\!\!}{\S\!\!}

\newtheorem{theorem}[equation]{Theorem}
\newtheorem{lemma}[equation]{Lemma}
\newtheorem{cor}[equation]{Corollary}
\Crefname{cor}{Corollary}{Corollaries}

\newtheorem{proposition}[equation]{Proposition}

\theoremstyle{remark}
\newtheorem{remark}[equation]{Remark}
\theoremstyle{definition}
\newtheorem{example}[equation]{Example}
\theoremstyle{definition}

\theoremstyle{definition}
\newtheorem{defi}[equation]{Definition}
\Crefname{defi}{Definition}{Definitions}
\theoremstyle{definition}
\newtheorem{notation}[equation]{Notation}
\theoremstyle{definition}

\theoremstyle{definition}
\newtheorem{assumption}[equation]{Assumption}

\DeclareSymbolFont{cyrletters}{OT2}{wncyr}{m}{n}
\DeclareMathSymbol{\Sha}{\mathalpha}{cyrletters}{"58}
\DeclareMathSymbol{\sha}{\mathalpha}{cyrletters}{"78}
\def\K{\ensuremath K}

\def\F{\ensuremath\mathbb{F}}

\newcommand{\art}[2]{\left(\frac{#1}{#2}\right)}

\def\QQ{\ensuremath\mathbb{Q}}
\def\ZZ{\ensuremath\mathbb{Z}}
\def\FFF{\ensuremath\mathscr{F}}
\def\CCC{\ensuremath\mathscr{C}}

\def\SSS{\ensuremath\mathscr{S}}
\def\gal{\ensuremath\operatorname{Gal}}

\newcommand{\sel}[1]{\operatorname{Sel}_{#1}}
\newcommand{\gp}[1]{\langle{#1}\rangle}
\newcommand{\set}[1]{\left\{{#1}\right\}}

\address{Max-Planck-Institut für Mathematik, Vivatsgasse 7, 53111 Bonn,
Germany}
\email{ajmorgan44@gmail.com}

\address{School of Mathematics and Statistics, University of Glasgow, University Place, Glasgow, G12 8QQ.}
\email{r.paterson.2@research.gla.ac.uk}
\subjclass[2010]{11G05 (11G10 11N45 14H52)}

\begin{document}

\title{On $2$-Selmer groups of twists after quadratic extension}

\author{Adam Morgan and Ross Paterson}

\hypersetup{
pdftitle={On 2-Selmer groups of twists after quadratic extension},
pdfauthor={Adam Morgan and Ross Paterson},
pdfsubject={11G05 (11G10 11N45 14H52)},
pdfkeywords={elliptic curves, abelian varieties, arithmetic statistics, Selmer groups, quadratic twists}
}
\makeatother

\begin{abstract}
Let $E/\mathbb{Q}$ be an elliptic curve with full rational $2$-torsion. As $d$ varies over squarefree integers, we study the behaviour of the quadratic twists $E_d$ over a fixed quadratic extension $K/\mathbb{Q}$.  
We prove that for 100\% of twists the dimension of the $2$-Selmer group over $K$ is given by an explicit local formula, and use this to show that this dimension follows an Erd\H{o}s--Kac type distribution.
This is in stark contrast to the distribution of the dimension of the corresponding $2$-Selmer groups over $\QQ$, and this discrepancy allows us to determine the distribution of the $2$-torsion in the Shafarevich--Tate groups of the $E_d$ over $K$ also.

As a consequence of our methods we prove that, for $100\%$ of twists $d$, the action of $\gal(K/\QQ)$ on the $2$-Selmer group of $E_d$ over $K$ is trivial, and the Mordell--Weil group $E_d(K)$ splits integrally as a direct sum of its invariants and anti-invariants.
On the other hand, we give examples of thin families of quadratic twists in which a positive proportion of the $2$-Selmer groups over $K$ have non-trivial $\gal(K/\QQ)$-action, illustrating that these previous results are genuinely statistical  phenomena.
\end{abstract}

\maketitle

\tableofcontents

\section{Introduction}
\label{sec:introduction}

Let $E/\mathbb{Q}$ be an elliptic curve with $E[2]\subseteq E(\mathbb{Q})$, and consider the family of quadratic twists of $E$ over $\mathbb{Q}$$\colon$
\[\{E_d ~~\colon ~~ d\in \mathbb{Z}~~\textup{squarefree}\}.\]
Let $K/\mathbb{Q}$ be a quadratic extension. In this paper, as $d$ varies we study the $2$-Selmer groups $\textup{Sel}^2(E_d/K)$ of  $E_d$ over $K$. 

\subsection{Erd\H{o}s--Kac  for $2$-Selmer}

Our first result, strongly reminiscent of the Erd\H{o}s--Kac theorem \cite{MR2374}, shows that the distribution of the quantity
\[\frac{\dim \textup{Sel}^2(E_d/K)-\log \log|d|}{\sqrt{2\log \log|d|}}\]
is standard normal. That is, we have the following:

\begin{theorem}[\Cref{cor:Erdos Kac For Selmer}] \label{intro: EK thm}
	For every $z\in\mathbb{R}$ we have 
	\[\lim_{X\to\infty}~\frac{
	\#\set{\vert d\vert\leq X \textnormal{ squarefree }: \ \frac{\dim \textup{Sel}^2(E_d/K)-\log\log|d|}{\sqrt{2\log\log|d|}}\leq z}
	}{
	\#\set{\vert d\vert\leq X \textnormal{ squarefree }}
	}=\frac{1}{\sqrt{2\pi}}\int_{-\infty}^ze^{-t^2/2}dt.\]
\end{theorem}

One immediate consequence  is that, for any fixed real number $z$, the proportion of $|d|\leq X$ for which  $\dim \textup{Sel}^2(E_d/K)$ is smaller than $z$ tends to $0$ as $X$ tends to infinity. We present this as \Cref{is_usually_large}.

By contrast, if we assume that $E$ has no cyclic $4$-isogeny defined over $\mathbb{Q}$, it is a result of Kane \cite{kane2013ranks}, building on work of Heath-Brown \cite{MR1292115} and Swinnerton-Dyer \cite{swinnerton2008effect}, that for any fixed interger $n\geq 2$  a positive proportion of twists $d$ have $\dim \textup{Sel}^2(E_d/\mathbb{Q})$ equal to $n$. Thus \Cref{intro: EK thm} shows that the groups $\textup{Sel}^2(E_d/K)$ exhibit significantly different behaviour to the corresponding groups over $\mathbb{Q}$. As a consequence of this discrepancy, we are able to show that, at least when $E$ has no cyclic $4$-isogeny defined over $\mathbb{Q}$, \Cref{intro: EK thm} remains true when $\dim \textup{Sel}^2(E_d/K)$ is replaced by  $\dim \Sha(E_d/K)[2]$ in the statement. We present this alternative perspective as \Cref{sha corollary}.

The distribution in \Cref{intro: EK thm} does show up in a slightly different setting over $\mathbb{Q}$ however, and is the same form as that appearing in work of  Klagsbrun--Lemke Oliver \cite{MR3430377} and Xiong--Zaharescu  \cite{xiong2008distribution}  concerning the distribution  of $2$-isogeny Selmer groups  in quadratic twist families of certain elliptic curves. We discuss a precise analogy explaining this similarity in \Cref{weil res intro} below.

\subsection{Structural results for $100\%$ of twists}

The  growth of the $2$-Selmer group  when passing from $\mathbb{Q}$ to $K$ apparent in \Cref{intro: EK thm} can be explained by work of Kramer  \cite{MR597871}. Write $G=\textup{Gal}(K/\mathbb{Q})$ for the Galois group of $K$ over $\mathbb{Q}$. For any elliptic curve $E'/\mathbb{Q}$, the $2$-Selmer group of $E'$ over $K$ is naturally a $G$-module. Roughly speaking, the work of Kramer identifies a quotient of the invariant subgroup $\textup{Sel}^2(E'/K)^G$ whose dimension is controlled by purely local invariants (this is implicit in \cite{MR597871}*{Theorem 1}, see also \Cref{wiles greenberg application,prop:ResCoresExactSequence} below). This is analogus to the situation for class groups of quadratic fields where the dimension of the $2$-torsion of the (narrow) class group admits an explicit description via genus theory.

In order to prove \Cref{intro: EK thm} we study, as $d$ varies,  the discrepancy between the `systematic' part of the $2$-Selmer group    $\textup{Sel}^2(E_d/K)$ alluded to above, and the full $2$-Selmer group. Ultimately, \Cref{intro: EK thm} is a consequence of the following result, giving a precise description of the full $2$-Selmer group for $100\%$ of twists.  

\begin{notation} \label{def:local_norm_map}
For each place $v$ of $\mathbb{Q}$, and  any place $w$ of $K$ extending $v$, define the local norm map
\[N_{K_w/\mathbb{Q}_v}\colon E(K_w) \longrightarrow E(\mathbb{Q}_v)\]
by the formula
\[N_{K_w/\mathbb{Q}_v}(P)=\sum_{\sigma \in\textup{Gal}(K_w/\mathbb{Q}_v)}\sigma(P).\]
\end{notation}

\begin{theorem}[\Cref{cor holds a lot}] \label{into norm cokernel theorem}
For $100\%$ of squarefree $d$ ordered by absolute value, the $\textup{Gal}(K/\mathbb{Q})$-action on $\textup{Sel}^2(E_d/K)$ is trivial, and 
 we have
\begin{equation} \label{explicit selmer formula}
\dim \textup{Sel}^2(E_d/K) =-2+ \sum_{v~\textup{place of }\mathbb{Q}}\dim E_d(\mathbb{Q}_v)/N_{K_w/\mathbb{Q}_v}E_d(K_w).
\end{equation}
Here $N_{K_w/\mathbb{Q}_v}$ is defined as in \Cref{def:local_norm_map} above. Since  $ E_d(\mathbb{Q}_v)/N_{K_w/\mathbb{Q}_v}E_d(K_w)$ is trivial if $v$ splits in $K$, the right hand side of \eqref{explicit selmer formula} does not depend on the choice of $w\mid v$.
\end{theorem}

 \begin{remark} 
In \Cref{the function g section} we study the behaviour of the right hand side of \Cref{explicit selmer formula}. Even when $E$ does not have all its $2$-torsion defined over $\mathbb{Q}$, we are still able to use this to gain partial control of the Selmer groups $\textup{Sel}^2(E_d/K)$ as $d$ varies. In particular, provided that $\mathbb{Q}(E[2])\cap K=\mathbb{Q}$ we show in \Cref{is_usually_large} that for any fixed real number $z$, the dimension of $\textup{Sel}^2(E_d/K)$ exceeds $z$ for $100\%$ of twists $d$. For the remainder of the introduction however, we continue to assume $E[2]\subseteq E(\mathbb{Q})$.
\end{remark}

\begin{remark}
The statement of \Cref{into norm cokernel theorem} is very reminiscent of a recent result of Fouvry--Koymans--Pagano \cite{FKP20}. There it is shown that, for $100\%$ of odd positive squarefree integers $n$, the class group of the Dirichlet biquadratic field $\mathbb{Q}(\sqrt{n},i)$ has $4$-rank equal to $\omega_3(n)-1$, where $\omega_3(n)$ is the number of primes dividing $n$ that are congruent to $3$ modulo $4$ (i.e. are inert in $\mathbb{Q}(i)$). The similarity with the statement of \Cref{into norm cokernel theorem} is made apparent by \Cref{Genus Theory as 2-Torsion Frobenii}. This can be viewed as an instance of the known analogy between $4$-ranks of class groups and   $2$-Selmer ranks of elliptic curves  apparent in the works of Heath--Brown \cites{MR1193603,MR1292115} and Fouvry--Kl\"uners  \cite{MR2276261}, and extended to higher $2$-power ranks in the recent work of Smith  \cite{S2017}.
\end{remark}

It is natural to ask if the description in \Cref{into norm cokernel theorem} simply holds for \textit{all} $d$. This is, however, not the case. We discuss examples where the Galois action is nontrivial in \Cref{intro: twist examples} below. 

Since the group $E_d(K)/2E_d(K)$ sits inside $\textup{Sel}^2(E_d/K)$, we can deduce  some consequences for Mordell--Weil groups from the above results. Specifically, for  a squarefree integer $d$, write $\Lambda(E_d/\mathbb{Q})$
 for the finite rank free $\mathbb{Z}$-module given by the quotient of $E_d(\mathbb{Q})$ by its torsion subgroup:
 \[\Lambda(E_d/\mathbb{Q})=E_d(\mathbb{Q})/E_d(\mathbb{Q})_\textup{tors}.\]
We view this as a $G$-module with trivial action. Write $\Lambda(E_d/\mathbb{Q})(-1)$ for the $G$-module with underlying abelian group $\Lambda(E_d/\mathbb{Q})$ on which the generator of $G$ acts as multiplication by $-1$. Further, write $K=\mathbb{Q}(\sqrt{\theta})$.  We  have the following result, giving a complete description of the $G$-module structure of $E_d(K)$ for $100\%$ of $d$.  
 
\begin{theorem}[\Cref{prop:MWGroupsStatistical}] \label{model_weil_split_intro}
For $100\%$ of squarefree $d$ ordered by absolute value, we have an isomorphism of $\ZZ[G]$-modules
	\begin{equation} 
	E_d(K)\cong (\mathbb{Z}/2\mathbb{Z})^2~\oplus ~\Lambda(E_d/\mathbb{Q}) ~\oplus ~\Lambda(E_{d\theta}/\mathbb{Q})(-1) ,
	\end{equation}
	where here $(\mathbb{Z}/2\mathbb{Z})^2$ carries trivial $G$-action. 
\end{theorem}
 
  \subsection{Twists of the Weil restriction of scalars} \label{weil res intro}
 
 Write $A=\textup{Res}_{K/\mathbb{Q}}E$ for the Weil restriction of scalars of $E$ from $K$ to $\mathbb{Q}$. This is a principally polarised abelian surface over $\mathbb{Q}$. For each squarefree integer $d$ we have (see \Cref{weil_res_quad_twist})
 \[\textup{Sel}^2(E_d/K)\cong \textup{Sel}^2(A_d/\mathbb{Q}).\]
 
 In particular, we can view \Cref{intro: EK thm} as giving the distribution of $2$-Selmer groups in the quadratic twist family over $\mathbb{Q}$ of the abelian surface $A$. We state this formally as \Cref{main result for Weil restrictions}. 
 
 It is also possible to use this perspective to draw parallels between our work and existing work in the literature. Specifically, we show in \Cref{weil restriction subsection} that for each $d$, the twist $A_d$ admits an isogeny $\phi_d:A_d\rightarrow E_d\times E_{d\theta}$ whose kernel is a subgroup of $A_d[2]$.   The order of the Selmer group $\textup{Sel}^{\phi_d}(A_d/\mathbb{Q})$ associated to $\phi_d$ is then, up to a quantity bounded independent of $d$, a lower bound for the size of the Selmer group $\textup{Sel}^2(A_d/\mathbb{Q})$. In turn, writing $\widehat{\phi_d}$ for the dual isogeny, a lower bound for the size of $\textup{Sel}^{\phi_d}(A_d/\mathbb{Q})$ is given by the \textit{Tamagawa ratio}
 \[\mathcal{T}(\phi_d)=\frac{|\textup{Sel}^{\phi_d}(A_d/\mathbb{Q})|}{|\textup{Sel}^{\widehat{\phi_d}}(E_d\times E_{d\theta}/\mathbb{Q})|}.\]
For any isogeny between abelian varieties, the Tamagawa ratio is known to admit a local formula, and in our case  this is essentially given by the right hand side of   \Cref{explicit selmer formula} (see \Cref{weil restriction subsection} for details). Consequently, one explanation for the unbounded growth of $\dim \textup{Sel}^2(E_d/K)$ seen in \Cref{intro: EK thm} is that the Tamagawa ratios $\mathcal{T}(\phi_d)$ tend to grow with $d$. Similarly, growth of the relevant Tamagawa ratios is  the phenomenon underlying the behaviour of $2$-isogeny Selmer groups of quadratic twist families of certain elliptic curves seen in work of   Klagsbrun--Lemke Oliver \cite{MR3430377} and Xiong--Zaharescu  \cite{xiong2008distribution}. Thus the behaviour we uncover can be viewed as an extension of those works to a special class of abelian surfaces.

\subsection{Prime twists of the congruent number curve} \label{intro: twist examples}
 
 As a complement to our main results, we provide examples of thin subfamilies of quadratic twists in which significantly different behaviour occurs to that exhibited by the full family.   Specifically,   take $E$ to be the congruent number curve:
 \[E:y^2=x^3-x.\]
Further, take  $K=\QQ(\sqrt{\theta})$ to be an imaginary quadratic extension of class number 1 in which 2 is inert. Thus $\theta\in\set{-3, -11, -19, -43, -67, -163}$. For a prime $p$, define  non-negative integers $e_1(E_p/K)$ and $e_2(E_p/K)$ such that we have an $\mathbb{F}_2[G]$-module isomorphism
\[\textup{Sel}^2(E_p/K)\cong \F_2^{  e_1(E_p/K)}\oplus\F_2[G]^{  e_2(E_p/K)}.\]

\begin{theorem}[\Cref{thm:density_for_CNC}] \label{main intro theorem congruent}
	The natural density of primes $p$ for which $e_1(E_p/K)=e_1$ and $e_2(E_p/K)=e_2$ is as follows:
	\[\lim_{X\to\infty}\frac{\#\set{p\leq X\textnormal{ prime}~:~e_1(E_p/K)=e_1\textnormal{ and }e_2(E_p/K)=e_2}}{\#\set{p\leq X\textnormal{ prime}}}=
	\begin{cases}
		9/16&\textnormal{if }(e_1,e_2)=(4,0),
		\\1/16&\textnormal{if }(e_1,e_2)=(2,2),
		\\4/16&\textnormal{if }(e_1,e_2)=(2,1),
		\\2/16&\textnormal{if }(e_1,e_2)=(2,0).
	\end{cases}\]
	In particular, the proportion of prime twists for which the $G$-action on $\textup{Sel}^2(E_p/K)$ is non-trivial is equal to $5/16$.
\end{theorem}

\subsection{Overview of the proofs}

The proofs of the results outlined above requires a combination of algebraic and analytic methods. Where possible we have tried to decouple these, so that the algebraic results stand alone. 

The algebraic work is largely carried out in \Cref{sec:seloverquad}  and \Cref{case of full 2-tors section}, and is based on  work of Kramer \cite{MR597871}. For a squarefree integer $d$, a key role in our results is played by the group $\textup{Sel}_{\CCC_d}(\mathbb{Q},E_d[2])$ of \Cref{the selmer structure f}. This is a subgroup of the $2$-Selmer group $\textup{Sel}^2(E_d/\mathbb{Q})$ of $E_d$ over $\mathbb{Q}$. Our key algebraic result is \Cref{what happens when selc vanishes} which shows that the Selmer group $\textup{Sel}^2(E_d/K)$ admits the explicit description of \Cref{into norm cokernel theorem} as soon as this auxiliary Selmer group $\textup{Sel}_{\CCC_d}(\mathbb{Q},E_d[2])$ vanishes. 

The main statistical theorems of the paper then depend  on proving \Cref{main statistical theorem written in sec 5}, which shows that $\textup{Sel}_{\CCC_d}(\mathbb{Q},E_d[2])$ is trivial for $100\%$ of $d$. To do this we draw on analytic techniques developed by Heath-Brown, and used to determine the distribution of the $2$-Selmer groups of quadratic twists of the congruent number curve \cites{MR1193603,MR1292115}. That work takes as a point of departure the explicit description of $2$-Selmer groups of elliptic curves with full $2$-torsion provided by $2$-descent. In \Cref{explicit selmer conditions 1} we similarly give an explicit description of $\textup{Sel}_{\CCC_d}(\mathbb{Q},E_d[2])$ as a subgroup of $(\mathbb{Q}^{\times}/\mathbb{Q}^{\times 2})^2$.
 
 In fact, for the analytic part of the argument we have opted to replace $\textup{Sel}_{\CCC_d}(\mathbb{Q},E_d[2])$ with a certain   subgroup $S_d$ of $\mathbb{Q}^{\times}/\mathbb{Q}^{\times 2}$ (see \Cref{definition of Sd}) whose vanishing implies the vanishing of $\textup{Sel}_{\CCC_d}(\mathbb{Q},E_d[2])$, but which admits a simpler explicit description. In \Cref{jacobi sums section} we give a formula for the order of $S_d$ as a sum of Jacobi symbols in a form which can be treated by the analytic tools of Heath-Brown mentioned above. An alternative method at this point might be to draw on the alternative approaches of Kane \cite{kane2013ranks} or Smith \cite{S2017}.

It is worth remarking that the passage from $\textup{Sel}_{\CCC_d}(\mathbb{Q},E_d[2])$ to $S_d$ is somewhat wasteful. By following the work of Heath-Brown \cite{MR1193603} more closely, one can similarly describe the order of $\textup{Sel}_{\CCC_d}(\mathbb{Q},E_d[2])$ as a sum of Jacobi symbols. This would likely lead to significant improvements to the error bounds in \Cref{main statistical theorem written in sec 5}. We have opted not to do this in favour of working with the simpler and more explicit sums arising from $S_d$. In this respect, the resulting analysis is  much closer to that carried out by Fouvry--Kl\"uners in \cite{MR2276261} to determine the distribution of $4$-ranks of class groups of quadratic fields. 

 \subsection{Layout of the paper}
 
 In \S2 we introduce some notation that will be in use throughout. 
 
In \S3 we review some basic properties of Selmer structures and their associated Selmer groups which we use in later sections.

In \S4 we study algebraically the behaviour of $2$-Selmer groups of elliptic curves in quadratic extensions, building on work of Kramer \cite{MR597871}. Along the way we give two reinterpretations of Kramer's work, one in the language of Selmer structures, and another in terms of the Weil restriction of scalars. 

In \S5 we study the analytic properties of the function $g(d)$ of \Cref{the function g of d} (essentially the right hand side of \Cref{explicit selmer formula}) which gives a lower bound for $\dim \textup{Sel}^2(E_d/K)$. In particular, we show in \Cref{thm:genus theory is normally distributed} that $g(d)$ follows an Erd\H{o}s--Kac type distribution.

In \S6 we state our main technical result, \Cref{main statistical theorem written in sec 5}, on the vanishing of the auxiliary Selmer group $\textup{Sel}_{\CCC_d}(\mathbb{Q},E_d[2])$ for $100\%$ of twists $d$. From this we deduce \Cref{intro: EK thm,into norm cokernel theorem,model_weil_split_intro}, along with related results. 
 
The proof of \Cref{main statistical theorem written in sec 5} is carried out across \S7 and \S8. In \S7 we give some algebraic preliminaries. In \S8 we build on this by describing the order of $\textup{Sel}_{\CCC_d}(\mathbb{Q},E_d[2])$  as a sum of Jacobi symbols, before following closely the strategy of \cite{MR2276261}*{\S5} to study the  behaviour of these sums as $d$ varies. 

In \S9 we prove \Cref{main intro theorem congruent} concerning the behaviour over certain quadratic extensions of the $2$-Selmer groups  of prime twists of the congruent number curve.

 \subsection{Acknowledgements}

  We would like to thank Alex Bartel for suggesting we look at the behaviour of $2$-Selmer groups in quadratic extensions from a statistical point of view, and for countless helpful comments. We would also like to thank Peter Koymans,  Carlo Pagano and Efthymios Sofos for helpful discussions, and the anonymous referee for carefully reading   the paper and providing several helpful corrections and suggestions.
    
 Throughout this work, AM was supported by the Max-Planck-Institut für Mathematik in Bonn, and RP was supported by a PhD scholarship from the Carnegie Trust for the Universities of Scotland.
  
 \section{Notation and conventions}

In this section we detail some notation and conventions which will be used throughout the paper. 

\subsection{Arithmetic functions}
  Given a positive integer $n$ we write $\omega(n)$ for the number of distinct prime factors of $n$.  We denote by $\mu$ the M\"obius function, and for coprime integers $m$ and $n$ with $n$ odd and positive, we write 
$\left(\frac{m}{n}\right)$
for  the corresponding Jacobi symbol.  

\subsection{Galois cohomology}
For a field $F$ of characteristic $0$ we write $\bar{F}$ for a (fixed once and for all) algebraic closure of $F$, and denote its absolute Galois group by  $G_F=\textup{Gal}(\bar{F}/F)$. For a positive integer $n$ we write $\boldsymbol \mu_n$ for the $G_F$-module of $n$-th roots of unity in $\bar{F}$, and write $\boldsymbol \mu=\cup_{n\geq 1}\boldsymbol \mu_n$.

 By a $G_F$-module $M$ we mean a discrete module $M$ on which $G_F$ acts continuously. For $i\geq 0$ we write $H^i(F,M)$ as a shorthand for the continuous cohomology groups $H^i(G_F,M)$. We define the \textit{dual} of $M$ to be 
\[M^*:=\textup{Hom}(M,\boldsymbol \mu).\]
  This is a $G_F$-module with action given by setting, for $\sigma \in G_F$ and $\phi\in M^*$, 
\[ {}^\sigma\phi(m)=\sigma \phi(\sigma^{-1}m).\]

For $i\geq 0$, if $L/F$ is a finite extension we denote the corresponding restriction and corestriction maps by
\[\textup{res}_{L/F}\colon H^i(F,M) \to H^i(L,M)\]
and 
\[\textup{cor}_{L/F}\colon H^i(L,M) \to H^i(F,M)\]
respectively. When $L/F$ is Galois and the action of $G_F$ on $M$ factors through $\textup{Gal}(L/F)$, we write $H^i(L/F,M)$ as a shorthand for the cohomology group $H^i(\textup{Gal}(L/F),M)$. 

\subsection{Number fields and completions}
For a number field $F$ and a place $v$  of $F$, we write $F_v$ for the completion of $F$ at $v$. 
 We implicitly fix embeddings $\bar{F}\hookrightarrow \bar{F}_v$ for each place $v$ and in this way view $G_{F_v}$ as a subgroup of $G_F$ for each $v$. In this way, for $M$ a $G_F$-module $M$, we obtain restriction maps on cohomology $\textup{res}_v:H^i(F,M)\rightarrow H^i(F_v,M)$ for each $v$. When $v$ is non-archimedean we denote by $F_v^{\textup{nr}}$ the maximal unramified extension of $F_v$, and write 
 \[H^1_{\textup{nr}}(F_v,M):=\ker\left(H^1(F_v,M)\stackrel{\textup{res}}{\longrightarrow}H^1(F_v^{\textup{nr}},M)\right)\]
 for the subgroup of unramified classes in $H^1(F_v,M)$.
 
 \subsection{The Kummer image for abelian varieties} \label{subsec:Kummer}
 Still taking $F$ to be a number field, for an abelian variety $A$ over $F$, and for a place $v$ of $F$, we denote by $\SSS(A/F_v)$ the image of the coboundary map
\begin{equation} \label{local_kummer_map}
\delta_v:A(F_v)/2A(F_v)\hookrightarrow H^1(F_v,A[2])
\end{equation}
arising from the short exact sequence of $G_{F_v}$-modules
\begin{equation} \label{elliptic curve kummer sequence}
0\longrightarrow A[2]\longrightarrow A(\bar{F}_v)\stackrel{2}{\longrightarrow}A(\bar{F}_v)\longrightarrow 0.
\end{equation}
  
  \subsection{Quadratic twists} \label{subsec:quadratic twists}
For a field $F$ of characteristic $0$, and for an element $d$ of $F^\times/F^{\times 2}$, we write $\chi_d$ for the associated quadratic character. Thus $\chi_d$ is the function from $G_F$ to $\{\pm 1\}$ defined by, for $\sigma \in G_F$,  the formula
\[\chi_d(\sigma)=\sigma(\sqrt{d})/\sqrt{d}.\] 
Given an abelian variety $A$ over $F$ we write $A_d$ for the quadratic twist of $A$ by $d$. That is, $A_d$ is an abelian variety over $F$, equipped with an $\bar{F}$-isomorphism 
\begin{equation}\label{eq:QuadraticTwistMap}\psi_d\colon A\stackrel{\sim}{\longrightarrow} A_d\end{equation}
such that for all $\sigma$ in $G_F$, $\psi_d^{-1}\circ{}^\sigma\psi_d $ is multiplication by $\chi_d(\sigma)$ on $A$. Here ${}^\sigma\psi_d$ is the unique isomorphism $A\rightarrow A_d$ sending any $P\in A(\bar{F})$ to $\sigma \psi_d( \sigma^{-1}P)$. Note in particular that $\psi_d$ is defined over $F(\sqrt{d})$, so that $A$ is isomorphic to $A_d$ over $F(\sqrt{d})$. 

\section{Selmer structures}  

In this section we review the properties of Selmer structures and their associated Selmer groups which will be used later. For details see e.g.  \cites{MR2031496,MR1638477} and the references therein. 

Throughout this section we take $F$ to be a number field. We take $M$ to be a finite $G_F$-module annihilated by $2$, so that $M$ is  a finite dimensional  $\mathbb{F}_2$-vector space. All dimensions will be taken over $\mathbb{F}_2$.

\subsection{Local duality} \label{local pairings}

For each place $v$ of $F$, we have the \textit{local Tate pairing}  
\[\left \langle~,~\right \rangle_v:H^1(F_v,M)\times H^1(F_v,M^*)\longrightarrow H^2(F_v,\boldsymbol \mu_2)=\textup{Br}(F_v)[2]\hookrightarrow \mathbb{Q}/\mathbb{Z}\] given by the composition of cup-product and the local invariant map.

\begin{theorem}[Tate local duality]
For each place $v$ of $F$ the pairing $\left \langle~,~\right \rangle_v$ is non-degenerate. Moreover, for each non-archimedean place $v\nmid 2$ such that the inertia group $I_{F_v}$ acts trivially on $M$, $H^1_{\textup{nr}}(F_v,M)$ and $H^1_{\textup{nr}}(F_v,M^*)$ are orthogonal complements under this pairing.
\end{theorem}

\begin{proof}
See \cite{MR2392026}*{Corollary 7.2.6} for non-archimedean $v$ and op. cit. Theorem 7.2.17 for archimedean $v$. The claim about the unramified subspaces is op. cit. Theorem 7.2.15.
\end{proof}

\begin{example} \label{hilbet symbol example}
Take $M=\boldsymbol \mu_2$, which is self-dual. For each place $v$ of $F$, Kummer theory gives a canonical isomorphism $H^1(F_v,\boldsymbol \mu_2)\cong F_v^\times/F_v^{\times 2}$ (and we have the corresponding isomorphism globally also). For any non-archimedean place $v\nmid 2$ of $F$  we have 
\[H^1_{\textup{nr}}(F_v,\boldsymbol \mu_2)= \mathcal{O}_{F_v}^\times/\mathcal{O}_{F_v}^{\times 2}\subseteq  F_v^\times/F_v^{\times 2}.\]
The local Tate pairing 
\[F_v^{\times}/F_v^{\times 2}\times F_v^{\times}/F_v^{\times 2}\longrightarrow \mathbb{Q}/\mathbb{Z}\]
is the  Hilbert symbol $(x,y)\mapsto (x,y)_v\in \{\pm 1\}\cong \frac{1}{2}\mathbb{Z}/\mathbb{Z}$. 
\end{example}

\subsection{Selmer structures}

\begin{defi}
A \textit{Selmer structure} $\mathcal{L}=\{\mathcal{L}_v\}_{v}$ for $M$ is a collection of subspaces 
\[\mathcal{L}_v\subseteq H^1(F_v,M)\]
 for each place $v$ of $F$, such that $\mathcal{L}_v=H^1_\textup{nr}(F_v,M)$ for all but finitely many places. The associated \textit{Selmer group} $\textup{Sel}_\mathcal{L}(F,M)$ is defined by the exactness of 
\[0\longrightarrow \textup{Sel}_\mathcal{L}(F,M)\longrightarrow H^1(F,M)\longrightarrow \prod_{v\textup{ place of }F}H^1(F_v,M)/\mathcal{L}_v.\]
It is a finite dimensional $\mathbb{F}_2$-vector space. 

For each place $v$ we write $\mathcal{L}_v^*$ for the orthogonal complement of $\mathcal{L}_v$ under the local Tate pairing, so that $\mathcal{L}_v^*$ is a subspace of $H^1(F_v,M^*)$. We   define the \textit{dual Selmer structure} $\mathcal{L}^*$ for $M^*$  by taking $\mathcal{L}^*=\{\mathcal{L}_v^*\}_v$. We refer to $\textup{Sel}_{\mathcal{L}^*}(F,M)$ as the \textit{dual Selmer group}. 
\end{defi}
 
\subsection{The  Greenberg--Wiles formula}

The following theorem  due to Greenberg and Wiles describes the difference in dimension between a Selmer group and its dual. 

\begin{theorem} \label{wiles greenberg formula}
Let $\mathcal{L}=\{\mathcal{L}_v\}_v$ be a Selmer structure for $M$. Then we have
\begin{eqnarray*}
 \dim \textup{Sel}_{\mathcal{L}}(F,M)-\dim \textup{Sel}_{\mathcal{L}^*}(F,M^*)\phantom{huge amounts of spaaaaaaaaaaaace}\\
\phantom{spaaaace}=~\dim M^{G_F}-\dim (M^*)^{G_F} +\sum_{v~\textup{place of }F}(\dim\mathcal{L}_v- \dim M^{G_{F_v}}).
\end{eqnarray*}
\end{theorem}

 \begin{proof}
This is \cite{wiles1995modular}*{Proposition 1.6}. See also  \cite{MR1638477}*{Theorem 2}.
 \end{proof}
 
 \begin{example}[The $2$-Selmer group of an elliptic curve] \label{selmer is max iso example}
Let $E/F$ be an elliptic curve.  For  each place $v$  of $F$ we have the Kummer image 
\[\SSS(E/F_v)\subseteq H^1(F_v,E[2])\]
defined in \Cref{subsec:Kummer}.   
The collection $\SSS=\{\SSS(E/F_v)\}_v$ defines a Selmer structure for $E[2]$. This is a consequence of the fact that,   for  a non-archimedean place $v\nmid 2$ of $F$ at which $E$ has good reduction, we have \[\SSS(E/F_v)=H^1_{\textup{nr}}(F_v,E[2]).\] 
Using the Weil pairing $(~,~)_{e_2}:E[2]\times E[2]\rightarrow \boldsymbol \mu_2$ to identify $E[2]$ with its dual, the resulting Selmer structure is self dual (see e.g. \cite{MR2833483}*{Proposition 4.10}). That is, each $\SSS(E/F_v)$ is a maximal isotropic subspace of $H^1(F_v,E[2])$ with respect to the local Tate pairing. We note  that \Cref{wiles greenberg formula} gives
\begin{equation} \label{sum of cokernel of 2 is trivial}
\sum_{v~\textup{place of }F}\left(\dim E(F_v)/2E(F_v)-\dim E(F_v)[2] \right)=0.
\end{equation}
One can also give an elementary proof of this by computing the local terms individually (see e.g. \cite{MR1370197}*{Proposition 3.9}).
\end{example}

\section{\texorpdfstring{$2$}{2}-Selmer groups over quadratic extensions} 
\label{sec:seloverquad}

For the rest of the paper we fix a quadratic extension $K/\mathbb{Q}$. Write $K=\mathbb{Q}(\sqrt{\theta})$ for a squarefree integer $\theta$, and write $G=\textup{Gal}(K/\mathbb{Q})$. Moreover we fix an elliptic curve $E/\mathbb{Q}$.  Note that at this point we make no assumption on the $2$-torsion of $E$, in later sections (\Cref{sec:main results} onwards) it will be necessary to reduce to the case of full $2$-torsion but we shall be clear when this restriction is made.  Denote by $\textup{Sel}^2(E/K)$ the $2$-Selmer group of $E/K$. The conjugation action of $G$ on $H^1(K,E[2])$ makes $\textup{Sel}^2(E/K)$  into an $\mathbb{F}_2[G]$-module.

The structure of $\textup{Sel}^2(E/K)$ has been studied by Kramer in \cite{MR597871}. In this section, since it will be useful for what follows, we give a  reinterpretation of part of this work in the language of Selmer structures (see also work of Mazur--Rubin \cites{MR2373150,MR2660452} for a similar perspective). The results in  this section can be adapted in a straightforward way to  general quadratic extensions of number fields (and this is the setting in which Kramer proves his results). However, we stick to quadratic extensions of $\mathbb{Q}$ since this is the setting in which all our applications are carried out.

As in \Cref{subsec:quadratic twists}, associated to the squarefree integer $\theta$ is the quadratic twist $E_\theta$, which comes equipped with the isomorphism $\psi=\psi_\theta$ from $E$ to $E_\theta$. Whilst this isomorphism is only defined over $K$, it restricts to an isomorphism of $G_\mathbb{Q}$-modules  from $E[2]$ to $E_\theta[2]$. We use this to identify $H^1(\mathbb{Q}_v,E[2])$ and $H^1(\mathbb{Q}_v,E_\theta[2])$ for each place of $v$, and identify the corresponding global cohomology groups similarly. In particular, for each place $v$ of $\mathbb{Q}$ we may view both the Kummer images $\SSS(E/\mathbb{Q}_v)$ and $\SSS(E_\theta/\mathbb{Q}_v)$ (cf. \Cref{subsec:Kummer}) as subgroups of $H^1(\mathbb{Q}_v,E[2])$. Similarly, we view both $\textup{Sel}^2(E/\mathbb{Q})$ and $\textup{Sel}^2(E_\theta/\mathbb{Q})$ as subgroups of $H^1(\mathbb{Q},E[2])$.

\subsection{Selmer structures associated to $E/K$}

We begin by defining two Selmer structures for $E[2]$ over $\mathbb{Q}$, each of which will capture a part of $\textup{Sel}^2(E/K)$. 

\begin{defi} \label{the selmer structure f}
Define the Selmer structure $\FFF$ for the $G_\mathbb{Q}$-module $E[2]$  by setting, for each place $v$ of $\mathbb{Q}$, 
\[\FFF_{v}=\FFF(E/\mathbb{Q}_v):=\textup{res}_{K_w/\mathbb{Q}_v}^{-1}(\SSS(E/K_w))\leq H^1(\mathbb{Q}_v,E[2])\]
where $w$ is any choice of place of $K$ extending $v$ (the definition does not depend on this choice). Let $\textup{Sel}_\FFF(\mathbb{Q},E[2])\leq H^1(\mathbb{Q},E[2])$ denote the resulting Selmer group. We further define the Selmer structure $\CCC$ for $E[2]$ as the dual of $\FFF$, and denote the local conditions by $\CCC(E/\mathbb{Q}_v)$. We denote the resulting Selmer group $\textup{Sel}_\CCC(\mathbb{Q},E[2])$. 
\end{defi}

\begin{lemma} \label{inverse image of restriction}
We have $\textup{Sel}_\FFF(\mathbb{Q},E[2])=\textup{res}_{K/\mathbb{Q}}^{-1}\left(\textup{Sel}^2(E/K)\right)$.
\end{lemma}

\begin{proof}
This follows from the compatibility of local and global restriction maps. 
\end{proof}

Recall the definition of the local norm map from \Cref{def:local_norm_map}.

\begin{lemma} \label{the properties of the selmer structure C}
The following properties hold for the Selmer structure $\CCC$.
\begin{itemize}
\item[(i)] For each place $v$ of $\mathbb{Q}$ we have \[\CCC(E/\mathbb{Q}_v)=\textup{cor}_{K_w/\mathbb{Q}_v}(\SSS(E/K_w))\leq H^1(\mathbb{Q}_v,E[2]),\]
where $w$ is any choice of place of $K$ extending $v$.
\item[(ii)] For each place $v$ of $\mathbb{Q}$ we moreover have
\[\CCC(E/\mathbb{Q}_v)=\delta_v(N_{K_w/\mathbb{Q}_v}E(K_w))=\SSS(E/\mathbb{Q}_v)\cap \SSS(E_\theta/\mathbb{Q}_v),\]
where $\delta_v:E(\mathbb{Q}_v)/2E(\mathbb{Q}_v)\hookrightarrow H^1(\mathbb{Q}_v,E[2])$ is the local Kummer map \Cref{local_kummer_map} and the intersection takes place in $H^1(\mathbb{Q}_v,E[2])$.
\item[(iii)] Globally we have $\textup{Sel}_\CCC(\mathbb{Q},E[2])=\textup{Sel}^2(E/\mathbb{Q})\cap \textup{Sel}^2(E_\theta/\mathbb{Q})$.  Moreover,  we have
\[\textup{cor}_{K/\mathbb{Q}}\left(\textup{Sel}^2(E/K)\right)\subseteq \textup{Sel}_{\CCC}(\mathbb{Q},E[2]).\]
 \end{itemize}
 \end{lemma}

\begin{proof}
(i): That $\textup{cor}_{K_w/\mathbb{Q}_v}(\SSS(E/K_w))$ and $\textup{res}_{K_w/\mathbb{Q}_v}^{-1}(\SSS(E/K_w))$ are orthogonal complements under the local Tate pairing is noted by Kramer in the paragraph following Equation 10 in \cite{MR597871}. Specifically, it follows from \cite{MR0219512}*{Proposition 9} and  \cite{MR2392026}*{Corollary 7.1.4} that $\textup{res}_{K_w/\mathbb{Q}_v}$ and $\textup{cor}_{K_w/\mathbb{Q}_v}$ are adjoints with respect to the local Tate pairings. It  follows  that we have inclusions 
\[\textup{cor}_{K_w/\mathbb{Q}_v}(\SSS(E/K_w))\subseteq\FFF_v^*\]
 and 
 \[\textup{res}_{K_w/\mathbb{Q}_v}\left(\textup{cor}_{K_w/\mathbb{Q}_v}(\SSS(E/K_w))^*\right)\subseteq \SSS(E/K_w)^*.\] 
 Since $\SSS(E/K_w)$ is its own orthogonal complement, the result follows.

(ii): The first equality follows from the fact that the coboundary maps arising from the respective Kummer sequences \Cref{elliptic curve kummer sequence} over $K_w$ and $\mathbb{Q}_v$ commute with corestriction. The second equality is \cite{MR597871}*{Proposition 7}.

(iii): The claim that $\textup{Sel}_\CCC(\mathbb{Q},E[2])=\textup{Sel}^2(E/\mathbb{Q})\cap \textup{Sel}^2(E_\theta/\mathbb{Q})$ is a formal consequence of (ii). The inclusion
\[\textup{cor}_{K/\mathbb{Q}}\left(\textup{Sel}^2(E/K)\right)\subseteq \textup{Sel}_{\CCC}(\mathbb{Q},E[2])\] follows from (i) and compatibility of the local and global corestriction maps. 
\end{proof}

\begin{remark} \label{f is sum remark}
Let $v$ be a place of $\mathbb{Q}$. Since the Selmer structure $\FFF$ is dual to $\CCC$, it follows formally from \Cref{the properties of the selmer structure C} and the fact that $\SSS(E/\mathbb{Q}_v)$ is its own orthogonal complement, that we have 
\[\FFF(E/\mathbb{Q}_v)=\SSS(E/\mathbb{Q}_v)+\SSS(E_\theta/\mathbb{Q}_v),\]
where the sum is taken inside $H^1(\mathbb{Q}_v,E[2])$.
\end{remark}

We may use \Cref{wiles greenberg formula} to determine the difference between the dimensions of the Selmer groups $\textup{Sel}_\FFF(\mathbb{Q},E[2])$ and  $\textup{Sel}_\CCC(\mathbb{Q},E[2])$.

\begin{lemma} \label{wiles greenberg application}
We have 
\[\dim \textup{Sel}_\FFF(\mathbb{Q},E[2])- \dim \textup{Sel}_\CCC(\mathbb{Q},E[2])=\sum_{v\textup{ place of }\mathbb{Q}} \dim E(\mathbb{Q}_v)/N_{K_w/\mathbb{Q}_v}E(K_w). \]
\end{lemma}

\begin{proof} 
Since for each place $v$ of $\mathbb{Q}$, the groups $\CCC(E/\mathbb{Q}_v)$ and $\FFF(E/\mathbb{Q}_v)$ are orthogonal complements under the local Tate pairing, we have 
\[\dim\FFF(E/\mathbb{Q}_v)=\dim H^1(\mathbb{Q}_v,E[2])-\dim \CCC(E/\mathbb{Q}_v).\]
Moreover, since $\SSS(E/\mathbb{Q}_v)$ is its own orthogonal complement we have \[\dim H^1(\mathbb{Q}_v,E[2])=2\dim E(\mathbb{Q}_v)/2E(\mathbb{Q}_v).\] Along with \Cref{the properties of the selmer structure C}(ii) this gives
\begin{eqnarray*}
\dim \FFF_v&=& 2\dim E(\mathbb{Q}_v)/2E(\mathbb{Q}_v)-\dim N_{K_w/\mathbb{Q}_v}E(K_w)/2E(\mathbb{Q}_v)\\ &=& \dim E(\mathbb{Q}_v)/2E(\mathbb{Q}_v)+ \dim E(\mathbb{Q}_v)/N_{K_w/\mathbb{Q}_v}E(K_w).
\end{eqnarray*}
\Cref{wiles greenberg formula} then gives
\begin{eqnarray*}
\dim \textup{Sel}_{\FFF}(\mathbb{Q},E[2])- \dim \textup{Sel}_\CCC(\mathbb{Q},E[2])&=&\sum_{v\textup{ place of }\mathbb{Q}} \dim E(\mathbb{Q}_v)/N_{K_w/\mathbb{Q}_v}E(K_w) 
\end{eqnarray*}
\[\phantom{move_over_the_page_some_moooreeeee}+\sum_{v\textup{ place of }\mathbb{Q}} \left(\dim E(\mathbb{Q}_v)/2E(\mathbb{Q}_v)- \dim E(\mathbb{Q}_v)[2]\right).\]
and the result follows from \Cref{sum of cokernel of 2 is trivial}.
\end{proof}

\subsection{The $2$-Selmer group of $E/K$}

We now apply the results above to study the $2$-Selmer group of $E/K$. 

 \begin{lemma}{\textup{(}cf. \cite{MR597871}*{Lemma 3}\textup{)}.}  \label{prop:ResCoresExactSequence}  
We have an exact sequence
\begin{equation} \label{res cores sequence}
0\longrightarrow H^1\left(K/\mathbb{Q},E(K)[2] \right)\stackrel{\textup{inf}}{\longrightarrow}\textup{Sel}_\FFF(\mathbb{Q},E[2])\stackrel{\textup{res}_{K/\mathbb{Q}}}{\longrightarrow}\textup{Sel}^2(E/K)\stackrel{\textup{cor}_{K/\mathbb{Q}}}{\longrightarrow} \textup{Sel}_\CCC(\mathbb{Q},E[2]).
\end{equation}
 \end{lemma}
 
 \begin{proof}
 We first claim that the sequence
 \[H^1(\mathbb{Q},E[2])\stackrel{\textup{res}}{\longrightarrow}H^1(K,E[2])\stackrel{\textup{cor}}{\longrightarrow}H^1(\mathbb{Q},E[2])\]
 is exact. To see this, consider the exact sequence of $G_\mathbb{Q}$-modules
 	\[0\longrightarrow \mathbb{F}_2\longrightarrow \mathbb{F}_2[G]\stackrel{\epsilon}{\longrightarrow}\mathbb{F}_2\longrightarrow 0,\] 
where $\epsilon$ is the augmentation map (sending $\sum_{g\in G}\lambda_g g$ to $\sum_{g\in G}\lambda$) and $G_\mathbb{Q}$ acts on $G$ via the quotient map $G_\mathbb{Q}\twoheadrightarrow G$.  Taking the tensor product over $\mathbb{F}_2$  with $E[2]$, and then taking Galois cohomology over $\mathbb{Q}$, gives an exact sequence of $G_\mathbb{Q}$-modules
\[H^1(\QQ, E[2])\longrightarrow H^1(\QQ, E[2]\otimes_{\mathbb{F}_2} \F_2[G])\longrightarrow H^1(\QQ,E[2]).\]
Using Shapiro's Lemma to identify $H^1(\mathbb{Q},E[2]\otimes_{\mathbb{F}_2} \mathbb{F}_2[G])$ with $H^1(K,E[2])$ yields the sought exact sequence.
 
 Having shown the claim, the result now follows  by  combining the inflation-restriction exact sequence with \Cref{inverse image of restriction} and \Cref{the properties of the selmer structure C}(iii).
 \end{proof}
 
\begin{cor} \label{what happens when selc vanishes}

If  $\textup{Sel}_\CCC(\mathbb{Q},E[2])=0$, then all of the following hold.
\begin{itemize}
\item[(i)] There is a short exact sequence
 \[0\longrightarrow H^1(K/\mathbb{Q},E(K)[2] )\stackrel{\textup{inf}}{\longrightarrow}\textup{Sel}_\FFF(\mathbb{Q},E[2])\stackrel{\textup{res}_{K/\mathbb{Q}}}{\longrightarrow}\textup{Sel}^2(E/K)\longrightarrow 0,\]
 where the first map is inflation.
\item[(ii)] We have
\[\dim \textup{Sel}^2(E/K) =- \dim \left(\frac{E(\mathbb{Q})[2]}{N_{K/\mathbb{Q}}(E(K)[2])} \right)+ \sum_{v~\textup{place of }\mathbb{Q}}\dim E(\mathbb{Q}_v)/N_{K_w/\mathbb{Q}_v}E(K_w).\]
\item[(iii)] The $G$-action on $\textup{Sel}^2(E/K)$ is trivial.
\end{itemize}
\end{cor}

\begin{proof}
(i): follows immediately from \Cref{prop:ResCoresExactSequence}.

 (ii): follows from (i) and \Cref{wiles greenberg application} upon noting that, since $\textup{Gal}(K/\mathbb{Q})$ is cyclic, we have 
\[H^1(K/\mathbb{Q}, E(K)[2])\cong \frac{E(\mathbb{Q})[2]}{N_{K/\mathbb{Q}}(E(K))[2]}  .\]
(See e.g. \cite[Section 8]{MR0219512} for the description of the cohomology of cyclic groups we are using in the above.) 
 
 (iii): follows from (i) and the fact that the image of the restriction map from $H^1(\mathbb{Q},E[2])$ to $H^1(K,E[2])$ is contained in the invariant subspace $H^1(K,E[2])^G$. 
\end{proof}

 For a similar result to \Cref{what happens when selc vanishes} (ii) which holds  when $K/\mathbb{Q}$ is replaced by a cyclic degree $p$ extension for an odd prime $p$, see \cite[Theorem 1.2]{B2014}. 
 
 \begin{remark} \label{wiles greenberg formula 2-selmer}
Combining \Cref{wiles greenberg application} with \Cref{prop:ResCoresExactSequence} allows one to recover the formula for the rank of $E/K$ given in \cite{MR597871}*{Theorem 1}. In the second part of that theorem, Kramer studies the group $\textup{Sel}_\CCC(\mathbb{Q},E[2])/\textup{cor}_{K/\mathbb{Q}}(\textup{Sel}^2(E/K))$, which he refers to as the everywhere local/global norms group, and shows that it carries a non-degenerate alternating pairing given by the sum of the Cassels--Tate pairings on $\textup{Sel}^2(E/\mathbb{Q})$ and $\textup{Sel}^2(E_\theta/\mathbb{Q})$ (recall  from \Cref{the properties of the selmer structure C}(iii) that $\textup{Sel}_\CCC(\mathbb{Q},E[2])=\textup{Sel}^2(E/\mathbb{Q})\cap \textup{Sel}^2(E_\theta/\mathbb{Q})$). In particular, this group has even dimension. 
 \end{remark} 
 
 When $\textup{Sel}_\CCC(\mathbb{Q},E[2])$ is not necessarily trivial we still get a lower bound for the dimension of the $2$-Selmer group of $E$ over $K$. 
 
 \begin{lemma} \label{selmer_lower_bound_lemma}
We have
\begin{eqnarray*}
\dim \textup{Sel}^2(E/K) &\geq& -2+ \sum_{v\textup{ place of }\mathbb{Q}} \dim E(\mathbb{Q}_v)/N_{K_w/\mathbb{Q}_v}E(K_w).
 \end{eqnarray*}
\end{lemma}

\begin{proof}
From \Cref{prop:ResCoresExactSequence} we find   
\begin{eqnarray*}
\dim \textup{Sel}^2(E_d/K)&\geq& \dim \textup{Sel}_{\FFF_d}(\mathbb{Q},E_d[2])- \dim \textup{Sel}_{\CCC_d}(\mathbb{Q},E[2]) -\dim H^1(K/\mathbb{Q},E(K)[2]).
\end{eqnarray*}
The result now follows from  \Cref{wiles greenberg application}, noting that  
\[\dim H^1(K/\mathbb{Q},E(K)[2])\leq 2,\]
which is a consequence of the explicit description of cohomology of cyclic groups.\end{proof}

\subsection{The Weil restriction of scalars} \label{weil restriction subsection}

Here we give a slight reinterpretation of the above material in terms of the restriction of scalars of $E$ from $K$ to $\mathbb{Q}$. The material of this section is closely related to, and inspired by, that appearing in \cite{MR2373150}*{\S3}. Following Milne \cite{MR0330174}*{\S2}, the restriction of scalars may be described as a special case of a general construction of twists of powers of $E$, which we now recall. In what follows, for abelian varieties $A$ and $B$ defined over $\mathbb{Q}$, we endow   $\textup{Hom}_{\bar{\mathbb{Q}}}(A,B)$ (the group of $\bar{\mathbb{Q}}$-homomorphisms from $A$ to $B$) with the  $G_\mathbb{Q}$ action $\varphi \mapsto {}^\sigma \varphi$, where for $\sigma \in G_\mathbb{Q}$  the homomorphism ${}^\sigma \varphi$ sends $P\in A(\bar{\mathbb{Q}})$ to $\sigma \varphi(\sigma^{-1}P)$.

\begin{defi}  \label{milnes forms}
Let $n\geq 1$. To each matrix  $M=(m_{i,j})$ in $\textup{Mat}_n(\mathbb{Z})$ we can associate an endomorphism of $E^n$ given by
\[(P_1,...,P_n)\longmapsto \bigg(\sum_{j=1}^nm_{1,j}P_j , ..., \sum_{j=1}^nm_{n,j}P_j\bigg).\]
In this way we view $\textup{GL}_n(\mathbb{Z})$ as a subgroup of $\textup{Aut}_{\bar{\mathbb{Q}}}(E^n)$. Now suppose that $\Lambda$ is a free rank-$n$ $\mathbb{Z}$-module equipped with a continuous $G_\mathbb{Q}$-action. Choosing a basis for $\Lambda$ gives rise to a homomorphism
\[\rho_\Lambda: G_\mathbb{Q}\longrightarrow \textup{GL}_n(\mathbb{Z}),\]
which we view as a $1$-cocycle valued in $\textup{Aut}_{\bar{\mathbb{Q}}}(E^n)$. The class of $\rho_\Lambda$ in $H^1(\mathbb{Q},\textup{Aut}_{\bar{\mathbb{Q}}}(E^n))$ does not depend on the choice of basis. Associated to this cocycle class is a twist of $E^n$, which we denote $\Lambda \otimes E$. This is an abelian variety over $\mathbb{Q}$ of dimension $n$, equipped with a $\bar{\mathbb{Q}}$-isomorphism 
$\varphi_\Lambda:E^n\rightarrow \Lambda \otimes E$
 satisfying $\varphi_\Lambda^{-1}\circ {}^\sigma\varphi_\Lambda=\rho_\Lambda(\sigma)$
for all $\sigma \in G_\mathbb{Q}$. 
\end{defi}

The relevant restriction of scalars can then be defined as follows. 

\begin{defi}\label{weil_restriction}
Denote by $\mathbb{Z}[G]$ the integral group ring of $G=\textup{Gal}(K/\mathbb{Q})$. We define the restriction of scalars of $E$ relative to $K/\mathbb{Q}$  to be the abelian surface $E\otimes \mathbb{Z}[G]$ . We denote it $\textup{Res}_{K/\mathbb{Q}}E$. By the above, it comes equipped with an isomorphism $\varphi:E\times E \rightarrow \textup{Res}_{K/\mathbb{Q}}E$, defined over $K$, and such that for all $\sigma \in G_\mathbb{Q}$, and all $P,Q\in E(\bar{\mathbb{Q}})$, we have
\[(\varphi^{-1}\circ {}^\sigma\varphi) (P,Q) =\begin{cases}  ( P, Q) ~~&~~\chi_\theta(\sigma)=1,\\ 
 ( Q,  P)~~&~~\chi_\theta(\sigma)=-1.\end{cases}\]
In particular, $\varphi^{-1}$ composed with projection onto the first coordinate gives an isomorphism
\[\textup{Res}_{K/\mathbb{Q}}E(\mathbb{Q})\cong E(K).\]
\end{defi}

\begin{remark}
The restriction of scalars $\textup{Res}_{K/\mathbb{Q}}(E)$ is more typically defined as the unique scheme over $\mathbb{Q}$ representing the functor  on $\mathbb{Q}$-schemes 
\[T\longmapsto E(T\times_{\mathbb{Q}}K).\]
As in \cite{MR2373150}*{Section 2}, this is equivalent to the construction given above.
\end{remark}

\begin{notation}
To ease notation, in what follows we write $A=\textup{Res}_{K/\mathbb{Q}}(E)$. Thus $A$ is an abelian surface defined over $\mathbb{Q}$. 
\end{notation}

One has  
\begin{equation}
\textup{Sel}^2(E/K)\cong \textup{Sel}^2\left(A/\mathbb{Q}\right).
\end{equation}
Indeed, the corresponding result for Shafarevich--Tate groups is given in \cite{MR0330174}*{Proof of Theorem 1}, and the same argument works here (see also \cite{MR2373150}*{Proposition 3.1}).
In particular, in the case that $\textup{Sel}_\CCC(\mathbb{Q},E[2])=0$, \Cref{what happens when selc vanishes} can be interpreted as giving a description of   the $2$-Selmer group over $\mathbb{Q}$ of the abelian surface $A$. 

Moreover, it turns out that the groups $\textup{Sel}_\CCC(\mathbb{Q},E[2])$ and $\textup{Sel}_\FFF(\mathbb{Q},E[2])$  are the Selmer groups associated to a certain isogeny between $A$ and $E\times E_\theta$  as we now explain.

\begin{defi} \label{construction of the isogeny}
Consider the isogeny $\phi_0: E\times E  \rightarrow E\times E$ given by the formula
\[\phi_0(P,Q)=(P+Q,P-Q).\]
Let $\varphi:E\times E \rightarrow A$ be as in \Cref{weil_restriction}, and let $\psi=\psi_\theta:E \stackrel{\sim}{\longrightarrow}E_\theta$ be as in \Cref{eq:QuadraticTwistMap}. Now define the  isogeny (a priori over $K$)
\[\phi=(1\times \psi)\circ \phi_0 \circ\varphi^{-1}:A \longrightarrow E\times E_\theta.\]
One readily computes that in fact $\phi$ is defined over $\mathbb{Q}$. 

We denote by $\textup{Sel}^\phi(A/\mathbb{Q})$   the Selmer group associated to $\phi$.  For each place $v$ of $\mathbb{Q}$, we denote by $\delta_{\phi,v}$ the coboundary map 
\[\delta_{\phi,v}: E(\mathbb{Q}_v)\times E_\theta(\mathbb{Q}_v)\rightarrow H^1(\mathbb{Q}_v,A[\phi])\] 
associated to the short exact sequence
\[0\longrightarrow A[\phi]\longrightarrow A(\bar{\mathbb{Q}}_v)\stackrel{\phi}{\longrightarrow}(E\times E_\theta)(\bar{\mathbb{Q}}_v)\longrightarrow 0.\]
Then the collection $\{\textup{im}(\delta_{\phi,v})\}_{v}$ defines a Selmer structure for $A[\phi]$, with associated Selmer group is $\textup{Sel}^\phi(A/\mathbb{Q})$.
\end{defi}

\begin{lemma} \label{geometric sel f}
We have a canonical isomorphism $\textup{Sel}^\phi(A/\mathbb{Q})\cong \textup{Sel}_\FFF(\mathbb{Q},E[2])$.
\end{lemma}

\begin{proof}
With $\varphi:E\times E \rightarrow A$  as in \Cref{weil_restriction}, one readily checks that 
$\varphi^{-1}$ restricts to a $G_\mathbb{Q}$-isomorphism between $A[\phi]$ and the diagonal embedding of $E[2]$ into $E\times E$. In this way we identify $H^1(\mathbb{Q},A[\phi])$ and $H^1(\mathbb{Q},E[2])$. We make corresponding identifications  locally at each place of $\mathbb{Q}$ also. We will show that this identification maps $\textup{Sel}^\phi(A/\mathbb{Q})$ onto $\textup{Sel}_\FFF(\mathbb{Q},E[2])$.

For $i=1,2$ write $\Delta_i:E \rightarrow E\times E$ for the homomorphisms defined by
\[\Delta_1(P)=(P,P)~~\textup{ and }~~\Delta_2(P)=(P,-P).\]
This gives maps
\[\varphi \circ \Delta_1 :E\longrightarrow A\quad \textup{ and }\quad \varphi \circ \Delta_2\circ \psi^{-1} : E_\theta \longrightarrow A,\]
which are readily checked to be defined over $\mathbb{Q}$. For each place $v$ of $\mathbb{Q}$ these maps fit into a commutative diagram
\[
\xymatrix{0\ar[r]&E[2] \ar[r]^{}\ar@{=}[d]&E(\bar{\mathbb{Q}}_v)\ar[r]^{2}\ar[d]& E(\bar{\mathbb{Q}}_v)\ar[r]\ar[d]&0\\0\ar[r]&A[2]\ar[r]^{}\ar@{=}[d]& A(\bar{\mathbb{Q}}_v)\ar[r]^{\phi\phantom{hellooo}}& E(\bar{\mathbb{Q}}_v) \times E_\theta(\bar{\mathbb{Q}}_v)\ar[r]&0\\ 0\ar[r]&E_\theta[2]\ar[r]&
E_\theta(\bar{\mathbb{Q}}_v)\ar[u]\ar[r]^{2} &E_\theta(\bar{\mathbb{Q}}_v)\ar[r]\ar[u]&0,}
\] 
where the right-most vertical maps are induced by the natural inclusions into the respective factors. On cohomology this induces a commutative diagram
\[
\xymatrix{E(\bar{\mathbb{Q}}_v) ~~\ar[r]^{\delta_v}\ar[d]&H^1(\mathbb{Q}_v,E[2])\ar@{=}[d]\\
 E(\bar{\mathbb{Q}}_v) \times E_\theta(\bar{\mathbb{Q}}_v)\ar[r]^{\delta_{\phi,v}}&H^1(\mathbb{Q}_v,A[\phi])\ar@{=}[d]\\
E_\theta(\bar{\mathbb{Q}}_v)~~\ar[u]\ar@{^{(}->}[r]^{\phantom{hihihi}\delta_v}&H^1(\mathbb{Q}_v,E_\theta[2]).}
\]
The result now follows from \Cref{f is sum remark}. 
\end{proof}

\begin{remark} \label{sel c remark}
One can show that the product polarisation on $E\times E$ descends to a polarisation on $A$ defined over $\mathbb{Q}$ rather than just $K$ as is a priori the case (this follows from the material in \cite{MR1827021}*{\S2}). Thus $A$ is a principally polarised abelian surface. We can then view the dual isogeny to $\phi$ as an isogeny 
\[\widehat{\phi}: E\times E_\theta \longrightarrow A.\]
Denote by $\textup{Sel}^{\widehat{\phi}}(E\times E_\theta/\mathbb{Q})$ the associated Selmer group. 
It follows formally from  \Cref{geometric sel f} and the fact that the Selmer structure $\CCC$ is dual to $\FFF$, that we have 
\[\textup{Sel}^{\widehat{\phi}}(E\times E_\theta/\mathbb{Q}) \cong \textup{Sel}_\CCC(\mathbb{Q},E[2]).\]
With more work, one can show that the composition (in either direction) of $\phi$ and $\widehat{\phi}$ is multiplication by $2$, and that the maps 
\[A\stackrel{\phi}{\longrightarrow}E\times E_\theta \stackrel{\widehat{\phi}}{\longrightarrow}A\]
induce the sequence \Cref{res cores sequence}. 
\end{remark}

\begin{remark}
In the terminology of  \cite{MR3605974}*{\S2}, the quantity
\[\frac{|\textup{Sel}^\phi(A/\mathbb{Q})|}{|\textup{Sel}^{\widehat{\phi}}(E\times E_\theta/\mathbb{Q})|}\stackrel{\Cref{geometric sel f}}{=}\frac{|\textup{Sel}_\FFF(\mathbb{Q},E[2])|}{|\textup{Sel}_\CCC(\mathbb{Q},E[2])|}\]
is called the \textit{Tamagawa ratio} associated to the isogeny $\phi$. 
That the Tamagawa ratio for elliptic curves is given by a local formula goes back to Cassels \cite{MR179169}*{Theorem 1.1}. The corresponding result for abelian varieties, which in particular can be applied to $A$ and $\phi$,  is given by Milne in \cite{MR2261462}*{\S I.7}. This gives an alternative approach to the local formula of \Cref{wiles greenberg application}. We remark though that Milne's result is very closely related to \Cref{wiles greenberg formula}, so this is not really a different proof.
\end{remark}

In the next section  we will consider the $2$-Selmer groups $\textup{Sel}^2(E_d/K)$ associated to quadratic twists of $E$ by squarefree integers $d$. As the next lemma shows, this is equivalent to considering the $2$-Selmer groups associated to the quadratic twist family over $\mathbb{Q}$ of $A$. 

\begin{lemma} \label{weil_res_quad_twist}
Let $d$ be a square free integer. Let $E_d$ denote the quadratic twist of $E$ by $d$, and let $A_d$ denote the quadratic twist of $A$ by $d$. Then we have a $\mathbb{Q}$-isomorphism 
\[\textup{Res}_{K/\mathbb{Q}}(E_d)\cong A_d\]
of abelian surfaces. In particular, we have
\[\textup{Sel}^2(E_d/K)\cong \textup{Sel}^2\left(A_d/\mathbb{Q}\right).\]
\end{lemma}

\begin{proof}
Both $\textup{Res}_{K/\mathbb{Q}}(E_d)$ and $A_d$ are twists of $E\times E$, so we need only show that the resulting classes in $H^1(G_\mathbb{Q},\textup{Aut}_{\bar{\mathbb{Q}}}(E\times E))$ agree. Write $\chi_d$ and $\chi_\theta$ for the quadratic characters associated to $\mathbb{Q}(\sqrt{d})/\mathbb{Q}$ and $K/\mathbb{Q}$ respectively.

Fix $\sigma \in G_\mathbb{Q}$. Fix an isomorphism $\phi_1:E\stackrel{\sim}{\longrightarrow} E_d$ such that $\varphi_1^{-1}\circ {}^\sigma\varphi_1=\chi_d(\sigma)$, and write $\varphi_2$ for the isomorphism $\ E_d\times E_d\stackrel{\sim}{\longrightarrow} \textup{Res}_{K/\mathbb{Q}}(E_d)$  of \Cref{weil_restriction}. This gives a $\bar{\mathbb{Q}}$-isomorphism
\[\Upsilon=\varphi_2\circ (\varphi_1\times \varphi_1)\colon E\times E  \stackrel{\sim}{\longrightarrow}\textup{Res}_{K/\mathbb{Q}}(E_d).\]
The resulting cocycle satisfies, for $P,Q\in E(\bar{\mathbb{Q}})\times E(\bar{\mathbb{Q}})$,
\[(\Upsilon^{-1}\circ {}^\sigma\Upsilon)(P,Q)=\begin{cases}\left(\chi_d(\sigma)P,\chi_d(\sigma)Q\right)~~&~~\chi_\theta(\sigma)=1,\\ \left(\chi_d(\sigma)Q,\chi_d(\sigma)P\right)~~&~~\chi_\theta(\sigma)=-1.\end{cases}\]

On the other hand, fix $\psi_1: E\times E \stackrel{\sim}{\longrightarrow}A$ as in \Cref{weil_restriction}, and fix also $\psi_2:A\stackrel{\sim}{\longrightarrow} A_d$ such that $\psi_2^{-1}\psi_2=\chi_d$. Writing $\Upsilon'=\psi_2\circ \psi_1$, one readily computes that $\Upsilon'^{-1}\circ {}^\sigma\Upsilon'$ is given by the same formula as $\Upsilon^{-1}{}^\sigma \Upsilon$, giving the result.
\end{proof}

\begin{remark} \label{the isogeny phi_d}
For each squarefree integer $d$, \Cref{construction of the isogeny} gives an isogeny 
\[\phi_d: \textup{Res}_{K/\mathbb{Q}}(E_d) \longrightarrow E_d\times E_{d\theta}.\]
Via \Cref{weil_res_quad_twist} we view $\phi_d$  as an isogeny from $A_d$ to $E_d\times E_{d\theta}.$ One readily checks that  the standard identification of $A[2]$ with $A_d[2]$ identifies $A[\phi]$ and $A_d[\phi_d]$. 
\end{remark}

\section{Quadratic twists and a distribution result} \label{the function g section}

Recall that $K=\mathbb{Q}(\sqrt{\theta})/\mathbb{Q}$ is a quadratic extension, $G=\gal(K/\QQ)$ and $E/\mathbb{Q}$ is an elliptic curve. We now  consider the effect of replacing $E/\mathbb{Q}$ by its quadratic twist $E_d/\mathbb{Q}$, for a squarefree integer $d$.  We denote by $\FFF_d$ and $\CCC_d$ the Selmer structures of the previous section with local conditions $\FFF(E_d/\mathbb{Q}_v)$ and $\CCC(E_d/\mathbb{Q}_v)$ respectively. We have  associated Selmer groups   $\textup{Sel}_{\FFF_d}(\mathbb{Q},E_d[2])$ and $\textup{Sel}_{\CCC_d}(\mathbb{Q},E_d[2])$. For a squarefree integer $d$ we write $\chi_d:G_\mathbb{Q}\rightarrow \{\pm 1\}$ for the associated quadratic character defined by
\[\chi_d(\sigma)=\sigma(\sqrt{d})/\sqrt{d}.\]
 
 \subsection{The cokernel of the local norm map}
 
It turns out that the cokernel of the local norm map varies in a predictable way as we vary $d$. First, we fix some notation.
 
\begin{notation} \label{Notation:Sigma}
Fix a choice $\Sigma$ of a finite set of places of $\mathbb{Q}$ containing the real place, $2$, all primes which ramify in $K/\mathbb{Q}$, and all primes at which  $E$ has bad reduction.
\end{notation}

We begin with the following observation. 

\begin{lemma}\label{lem:No4Torsion}
Let $p\notin \Sigma$ be a prime divisor of $d$. Then $E_d(\mathbb{Q}_p^\textup{nr})$ has no points of exact order 4.  In particular, the same is true of $E_d(\mathbb{Q}_p)$.
\end{lemma}
\begin{proof}
	By assumption $E$ has good reduction at $p$, so $E[4]$ is unramified at $p$ (that is, the inertia group $I_p$ at $p$ acts trivially on $E[4]$).  Thus any element $\sigma$  of $I_p$ acts on $E_d[4]$ as multiplication by $\chi_d(\sigma)$. Since $\chi_d$ is ramified at $p$ by assumption, the restriction of $\chi_d$ to $I_p$ is non-trivial and one has
	\[E_d[4]^{I_p} =\{P\in E_d[4]~~\mid ~~P=-P\}=E_d[2],\]
giving the result.
\end{proof}

\begin{lemma} \label{explicit_coker}
Let $d$ be a squarefree integer, let $p\notin \Sigma$ be a prime, and let $\mathfrak{p}$ be a prime of $K$ lying over $p$. Then
\[\dim E_d(\mathbb{Q}_p)/N_{K_\mathfrak{p}/\mathbb{Q}_p}E_d(K_\mathfrak{p})=\begin{cases} 2~~&~~ p\mid d, ~p\textup{ inert in }K/\mathbb{Q},~\dim E(\mathbb{Q}_p)[2]=2, \\ 0~~&~~\textup{otherwise}.\end{cases}\]
\end{lemma}

\begin{proof}
If $p$ splits in $K$, then the local extension $K_\mathfrak{p}/\mathbb{Q}_p$ is trivial, so that $N_{K_\mathfrak{p}/\mathbb{Q}_p}$ is the identity map on $E_d(\mathbb{Q}_p)$.  

Next, suppose that $p\nmid d$. Since also $p\notin \Sigma$, $E_d$ has good reduction at $p$, and $K_\mathfrak{p}/\mathbb{Q}_p$ is unramified.  It follows from \cite{MR0444670}*{Corollary 4.4} that $N_{K_w/\mathbb{Q}_p}$ is surjective, giving the result.

Now suppose that $p\mid d$ and $p$ is inert in $K/\mathbb{Q}$. In particular, the local extension $K_\mathfrak{p}/\mathbb{Q}_p$ is unramified of degree $2$. \Cref{lem:No4Torsion} and a dimension count then show that the horizontal maps (induced by the inclusion of $E_d(K_\mathfrak{p})[2]$ into $E_d(K_\mathfrak{p})$) in the commutative square
\[\xymatrix{E_d(K_\mathfrak{p})[2]\ar[r]^{\sim\phantom{ello}} \ar[d]^{N_{K_\mathfrak{p}/\mathbb{Q}_p}}&E_d(K_\mathfrak{p})/2E_d(K_\mathfrak{p})\ar[d]^{N_{K_\mathfrak{p}/\mathbb{Q}_p}}\\
E_d(\mathbb{Q}_p)[2]\ar[r]^{\sim\phantom{ello}}&E_d(\mathbb{Q}_p)/2E_d(\mathbb{Q}_p),}\]
are isomorphisms.  Let $\sigma$ denote the non-trivial element of $\textup{Gal}(K_\mathfrak{p}/\mathbb{Q}_p)$.  Since $-1$ acts trivially on $E_d(K_\mathfrak{p})[2]$,  we have a short exact sequence
\[0\rightarrow E_d(\mathbb{Q}_p)[2] \longrightarrow E_d(K_\mathfrak{p})[2] \stackrel{1+\sigma}{\longrightarrow} N_{K_\mathfrak{p}/\mathbb{Q}_p}\left(E_d(K_\mathfrak{p})[2]\right)\rightarrow 0.\]
We thus have 
\begin{eqnarray*}
\dim E_d(\mathbb{Q}_p)/N_{K_\mathfrak{p}/\mathbb{Q}_p}E_d(K_\mathfrak{p})&=&\dim E_d(\mathbb{Q}_p)[2]/N_{K_\mathfrak{p}/\mathbb{Q}_p}\left(E_d(K_\mathfrak{p})[2]\right)\\
&=&2\dim E_d(\mathbb{Q}_p)[2]-\dim E_d(K_\mathfrak{p})[2]\\
 &=&2\dim E(\mathbb{Q}_p)[2]-\dim E(K_\mathfrak{p})[2].
\end{eqnarray*}
It remains to break into cases according to $\dim E(\mathbb{Q}_p)[2]=0,1,2$.  If $\dim E(\QQ_p)[2]\neq 1$ then $\dim E(\QQ_p)[2]=\dim E(K_\mathfrak{p})[2]$ since the $2$-torsion is either already full over $\QQ_p$ or given by the splitting of an irreducible cubic.  In the case that $\dim E(\QQ_p)[2]=1$, noting that since $E$ has good reduction at $p$, $\mathbb{Q}_p(E[2])/\mathbb{Q}_p$ is unramified, we have $\dim E(K_\mathfrak{p})[2]=2$, completing the proof.
\end{proof}

\begin{remark}
At primes $p\in \Sigma$ the cokernel of the local norm map is more complicated and depends on the reduction type of $E_d/\mathbb{Q}_p$. See \cite{MR597871} or \cite{MR664648} for more details. However, since the isomorphism class of $E_d$ over $\mathbb{Q}_p$ depends only on the class of $d$ in $\mathbb{Q}^{\times}_p/\mathbb{Q}_p^{\times 2}$, the same is true of the cokernel of the local norm map.  
\end{remark}
 
  To ease notation in what follows, we make the following definition. 
 
\begin{notation} \label{the function g of d}
For a squarefree integer $d$, write 
\[g(d):=\sum_{v\textup{ place of }\mathbb{Q}} \dim E_d(\mathbb{Q}_v)/N_{K_w/\mathbb{Q}_v}E_d(K_w) \]
where for a place $v$ of $\QQ$, we denote by $w$ a choice of extension of $v$ to $K$. Further, write 
	\[\omega_{E,K}(d):=\#\set{p\mid d~:~\substack{p\not\in \Sigma\\p\textnormal{ inert in }K/\QQ\\\dim E(\QQ_p)[2]=2}}.\]
\end{notation}

Note that by \Cref{selmer_lower_bound_lemma}, the function $g(d)-2$ gives a lower bound for $\dim \textup{Sel}^2(E_d/K)$. 

\begin{proposition} \label{Genus Theory as 2-Torsion Frobenii}
As $d$ varies in squarefree integers, we have
	\[g(d)=2\omega_{E,K}(d)+O(1)\]
where the implied constant depends only on the initial  curve $E$ and the quadratic field $K$. 
\end{proposition}

\begin{proof}
Since the places in $\Sigma$ contribute $O(1)$ to $g(d)$, we may ignore them. The result now follows from \Cref{explicit_coker}.
\end{proof}

\subsection{The distribution of $g(d)$}

\begin{notation}
Let $\delta_{E,K}$ be the natural density of primes $p$ such that $\omega_{E,K}(p)=1$. 
\end{notation}

The possible values of $\delta_{E,K}$ may be computed by applying the Chebotarev density theorem to the extension $K(E[2])/\mathbb{Q}$ and are given by the following table:

\begin{table}[h]
	\begin{center}
	\begin{tabular}{|| c | c c c c c c ||}
				\hline
				$\gal(\QQ(E[2])/\QQ)$&$\set{1}$&$\substack{\ZZ/2\ZZ\\K\neq\QQ(E[2])}$&$\substack{\ZZ/2\ZZ\\K= \QQ(E[2])}$&$\ZZ/3\ZZ$&$\substack{S_3\\K\not\subseteq \QQ(E[2])}$&$\substack{S_3\\K\subseteq \QQ(E[2])}$\\
				\hline\hline
				$\delta_{E,K}$&$1/2$&$1/4$&$0$&$1/6$&$1/12$&$0$\\
				\hline
	\end{tabular}
	\end{center}
\end{table}
In the following result of  Erd\H{o}s--Kac type, we determine the asymptotic distribution of the function $g(d)$ when the 2-torsion field of $E$ does not interact with $K$.  Since  $\dim \textup{Sel}^2(E_d/K)\geq g(d)-2$ by \Cref{selmer_lower_bound_lemma}, this shows that $\dim  \textup{Sel}^2(E_d/K)$ is (in a precise sense) typically at least as large as a constant times $\log \log (d)$.

\begin{proposition}\label{thm:genus theory is normally distributed}
Suppose that $\QQ(E[2])\cap K=\QQ$. Further, for a squarefree integer $d$ write
	\begin{equation*}
		\mu(d):=2\delta_{E,K}\log\log|d| \quad\textup{ and } 
		 \quad\sigma(d):=\sqrt{4\delta_{E,K}\log\log|d|}.
	\end{equation*}
Then the quantity
\[\frac{g(d)-\mu(d)}{\sigma(d)}\]
follows a standard normal distribution.   That is, for all $z\in \mathbb{R}$ we have 
	\[\lim_{X\to\infty}\frac{
	\#\set{\vert d\vert\leq X \textnormal{ squarefree}~:~\frac{g(d)-\mu(d)}{\sigma(d)}\leq z}
	}{\#\set{\vert d\vert\leq X \textnormal{ squarefree}}}=\frac{1}{\sqrt{2\pi}}\int_{-\infty}^ze^{-t^2/2}dt.\]
\end{proposition}
\begin{proof}
	Let $\gamma(d):=2\omega_{E,K}(d)$. Since by \Cref{Genus Theory as 2-Torsion Frobenii} this differs from $g(d)$ by a bounded amount, it is enough to prove the same assertion with $g$ replaced by $\gamma$.  Moreover, since this function satisfies $\gamma(d)=\gamma(-d)$, it is enough to prove that $\gamma$ has this distribution on the positive squarefree integers.  We will do this by combining the method of moments with \cite{granville2007sieving}*{Prop. 4}. Specifically, in the notation of that proposition, take 
\[\mathcal{A}:=\set{d\textnormal{ squarefree}~:~1\leq d\leq X}\]
and 
\[\mathcal{P}:=\set{p\textnormal{ prime}: p\leq X^{\epsilon(X)}}\]
 for a function $\epsilon(X)=o(1)$ to be chosen later.  Further, let $\gamma_{\mathcal{P}}$ be the strongly additive function  which agrees with $\gamma$ for $p\in\mathcal{P}$, and takes the value $0$ on primes $p\not\in\mathcal{P}$.  Note that, still using the notation of \cite{granville2007sieving}*{Prop. 4} we can take
 \[h(d)=\prod_{p\mid d}\frac{p}{p+1}, \quad\quad   r_d\ll d\sqrt{X},   \quad\quad x=\frac{6X}{\pi^2}+O(\sqrt{X}), \quad \textup{ and }\quad M=2,\]
 along with 
 \[\mu_{\mathcal{P}}(\gamma) =\sum_{p\in\mathcal{P}}2\omega_{E,K}(p)\frac{1}{p+1}\]
 and 
 \[\sigma_{\mathcal{P}}(\gamma)^2 =\sum_{p\in\mathcal{P}}4\omega_{E,K}(p)\frac{p}{(p+1)^2}.\]
Using the explicit form of the Chebotarev density theorem given in \cite{MR0447191}, standard arguments give 
\[\mu_\mathcal{P}(\gamma)=2\delta_{E,K} \log \log (X)+O(\log \epsilon(X))\quad \textup{ and }\quad \sigma_{\mathcal{P}}(\gamma)^2 = 4\delta_{E,K}\log \log(X)+O(\log \epsilon(X)).\]
	 
Taking $X$ sufficiently large in the conclusion of \cite{granville2007sieving}*{Prop. 4}  shows that for any $k\geq 0$ we have
	\begin{align*}
	&\frac{1}{\#\mathcal{A}}\sum_{d\in\mathcal{A}}\left(\gamma_{\mathcal{P}}(d)-\mu_{\mathcal{P}}(\gamma)\right)^k
	\\&\begin{cases}
	=(k-1)!!\sigma_{\mathcal{P}}(\gamma)^k+O_k\left(\sigma_{\mathcal{P}}(\gamma)^{k-2}+\log\log(X)^kX^{2k\epsilon(X)-1/2}\right)&k\textnormal{ even,}
	\\\ll_k\sigma_{\mathcal{P}}(\gamma)^{k-1}+\log\log(X)^kX^{2k\epsilon(X)-1/2}&k\textnormal{ odd.}
	\end{cases}
	\end{align*}
In particular, the $k$th moments of $(\gamma_{\mathcal{P}}-\mu_{\mathcal{P}}(\gamma))/\sigma_{\mathcal{P}}(\gamma)$ converge to those of a normal random variable with mean 0 and variance 1.  Note that for $n\leq X$ we have
	\[\gamma(n)-\gamma_{\mathcal{P}}(n)\leq 2\#\set{p\mid n~:~p>X^{\epsilon(X)}}\leq \frac{\log(n)}{\epsilon(X)\log(X)}\leq \epsilon(X)^{-1}.\]
Induction on $k$ (cf. \cite{granville2007sieving}*{Deduction of Theorem 1}) now shows that, taking $\epsilon(X)=\log\log\log(X)^{-1}$, we have
	%Inductive arg gives E_{k+1}=E_k M_1+M_k E_1 where M is moment and E error in moment.  E_1 is X/\epsilon(X) so done
	\[\frac{1}{\#\mathcal{A}}\sum_{d\in\mathcal{A}}(\gamma(d)-2\delta_{E,K}\log \log(X))^k=\frac{1}{\#\mathcal{A}}\sum_{d\in\mathcal{A}}(\gamma_{\mathcal{P}}(d)-\mu_{\mathcal{P}}(\gamma))^k+o(\log\log(X)^{k/2}).\]
Thus the $k$th moments of $(\gamma-2\delta_{E,K}\log\log(X))/\sqrt{4\delta_{E,K}\log\log(X)}$ converge as $X\to\infty$ to those of the standard normal distribution.  It then follows from \cite{MR1324786}*{Theorem 30.2, Example 30.1} that $\gamma$ becomes normally distributed with mean $2\delta_{E,K}\log\log(X)$ and variance $4\delta_{E,K}\log\log(X)$ in the limit $X\to \infty$, i.e.
	\[\lim_{X\to\infty}\frac{
	\#\set{\vert d\vert\leq X \textnormal{ squarefree}~:~\frac{\gamma(d)-\delta_{E,K}\log\log(X)}{\sqrt{4\delta_{E,K}\log\log(X)}}\leq z}
	}{\#\set{\vert d\vert\leq X \textnormal{ squarefree}}}=\frac{1}{\sqrt{2\pi}}\int_{-\infty}^ze^{-t^2/2}dt.
	\]
The result now follows.
\end{proof}

\begin{remark}
In the last step of the proof  we have used the standard result that a function $f$ becomes normal as $X\to\infty$ with mean $\mu(X):=C_0\log\log(X)$ and variance $\sigma^2(X):=C_1\log\log(X)$ for some constants $C_0,C_1>0$ if and only if the function $(f(d)-\mu(d))/\sigma(d)$ becomes normal as $X\to\infty$ with mean $0$ and variance $1$.  This can be proved directly. 
\end{remark}

\begin{remark}
	In the case that $K\subseteq\QQ(E[2])$, the function $\gamma(d)$ in the proof of \Cref{thm:genus theory is normally distributed} is $0$.  In particular, by  \Cref{Genus Theory as 2-Torsion Frobenii}, we have that the $k$th moments of $g(d)$ are bounded.
\end{remark}

We have the following basic corollary showing that, for $100\%$ of $d$, $\dim\textup{Sel}^2(E_d/K)$ is larger than any fixed integer whenever the 2-torsion of $E$ field does not interact with $K$. This is in stark contrast with the situation for the Selmer groups $\textup{Sel}^2(E_d/\mathbb{Q})$, whose distribution is determined by Kane in \cite{kane2013ranks}*{Thm. 3}.

\begin{cor} \label{is_usually_large}
	If $K\cap \QQ(E[2])=\QQ$, then for any $z\in\mathbb{R}$ we have
	\[\lim_{X\to\infty}\frac{\#\set{\vert d\vert\leq X\textnormal{ squarefree}~:~\dim(\sel{2}(E_d/K))\leq z}}{\#\set{\vert d\vert\leq X\textnormal{ squarefree}}}=0.\]
\end{cor}
\begin{proof}
By \Cref{selmer_lower_bound_lemma} we have $\dim \textup{Sel}^2(E_d/K)\geq g(d)-2$. The result now follows from \Cref{thm:genus theory is normally distributed}.
\end{proof}

\begin{remark}
By \Cref{weil_res_quad_twist}, \Cref{is_usually_large} also applies with $\textup{Sel}^2(E_d/K)$ replaced by the Selmer groups $\textup{Sel}^2\left((\textup{Res}_{K/\mathbb{Q}}E)_d/\mathbb{Q}\right)$ associated to the quadratic twists of the Weil restriction of $E$ from $K$ to $\mathbb{Q}$. 
\end{remark}

\section{Main results}
\label{sec:main results}

Recall that $K=\QQ(\sqrt{\theta})/\QQ$ is a quadratic extension with $G=\gal(K/\QQ)$.  From this section onwards, we make the restriction that our choice of elliptic curve $E/\QQ$ has $E[2]\subseteq E(\QQ)$.

 For a squarefree integer $d$, a consequence of \Cref{wiles greenberg application,prop:ResCoresExactSequence}   is that, roughly speaking, the auxiliary Selmer group $\textup{Sel}_{\CCC_d}(\mathbb{Q},E_d[2])$ controls  the discrepancy between $\dim \textup{Sel}^2(E_d/K)$ and the function $g(d)$ of \Cref{the function g of d}. Thus to improve on \Cref{thm:genus theory is normally distributed} and gain full control of the Selmer groups $\textup{Sel}^2(E_d/K)$ as $d$ varies, it suffices to control these auxiliary groups. We achieve this under the assumption that all $2$-torsion of $E$ is   defined over $\mathbb{Q}$. Specifically, across Sections \ref{case of full 2-tors section} and \ref{big analysis section} we will prove that, under this assumption, the Selmer group $\textup{Sel}_{\CCC_d}(\mathbb{Q},E_d[2])$ is trivial for $100\%$ of $d$. That is:

\begin{theorem} \label{main statistical theorem written in sec 5}
We have
\[\lim_{X\rightarrow \infty}\frac{\#\{d\textup{ squarefree }\mid ~|d|<X,~\textup{Sel}_{\CCC_d}(\mathbb{Q},E_d[2])=0\}}{\#\{d\textup{ squarefree }\mid ~|d|<X\}}=1.\]
\end{theorem}

\begin{remark}
We will in fact show that the number of squarefree $d$ with $|d|<X$ for which $\textup{Sel}_{\CCC_d}(\mathbb{Q},E_d[2])\neq 0$ is  $\ll X\log(X)^{-0.0394}$. See \Cref{main statistical theorem}. It is likely that with more work this bound could be improved significantly, however we have not attempted to do so.
\end{remark}

\begin{remark}
By \Cref{the properties of the selmer structure C} we have  
\[\textup{Sel}_{\CCC_d}(\mathbb{Q},E_d[2])=\textup{Sel}^2(E_d/\mathbb{Q})\cap \textup{Sel}^2(E_{d\theta}/\mathbb{Q})\]
where the intersection is taken inside $H^1(\mathbb{Q},E[2])$. Thus \Cref{main statistical theorem written in sec 5} shows that for $100\%$ of squarefree $d$, the groups $\textup{Sel}^2(E_d/\mathbb{Q})$ and $\textup{Sel}^2(E_{d\theta}/\mathbb{Q})$ share only the identity element. 
\end{remark}

Before embarking on the proof, we use the results of previous sections to draw several consequences of this theorem. 

\subsection{Statistical results for $2$-Selmer groups}

An immediate consequence of \Cref{main statistical theorem written in sec 5} is that the conclusion of \Cref{what happens when selc vanishes} holds for $100\%$ of squarefree $d$ when we have full 2-torsion.

\begin{cor} \label{cor holds a lot}
For $100\%$ of squarefree $d$ (ordered by absolute value), the $\textup{Gal}(K/\mathbb{Q})$-action on $\textup{Sel}^2(E_d/K)$ is trivial, and 
 we have
\begin{equation} \label{Selmer formula 100 percent}
\dim \textup{Sel}^2(E_d/K) =-2+ \sum_{v~\textup{place of }\mathbb{Q}}\dim E_d(\mathbb{Q}_v)/N_{K_w/\mathbb{Q}_v}E_d(K_w).
\end{equation}
\end{cor}
 
 As a consequence, we can upgrade \Cref{thm:genus theory is normally distributed} to the following Erd\H{o}s--Kac type result  determining the distribution of the full $2$-Selmer group.

\begin{cor}\label{cor:Erdos Kac For Selmer}  
	The quantity 
	\[\frac{\dim \textup{Sel}^2(E_d/K)-\log\log|d|}{\sqrt{2\log\log|d|}}\] follows a standard normal distribution. That is,  for every $z\in\mathbb{R}$ we have 
	\[\lim_{X\to\infty}\frac{
	\#\set{\vert d\vert\leq X \textnormal{ squarefree }: \ \frac{\dim \textup{Sel}^2(E_d/K)-\log\log|d|}{\sqrt{2\log\log|d|}}\leq z}
	}{
	\#\set{\vert d\vert\leq X \textnormal{ squarefree }}
	}=\frac{1}{\sqrt{2\pi}}\int_{-\infty}^ze^{-t^2/2}dt.\]
\end{cor}

\begin{proof}
	By \Cref{cor holds a lot}, amongst all squarefree integers $d$ with $|d|<X$, outside a set of cardinality   $o(X)$  we have
	\begin{eqnarray*} 
	\dim \textup{Sel}^2(E_d/K)=-2+ \sum_{v~\textup{place of }\mathbb{Q}}\dim E_d(\mathbb{Q}_v)/N_{K_w/\mathbb{Q}_v}E_d(K_w)=g(d)-2.
	\end{eqnarray*}
	The result now follows from \Cref{thm:genus theory is normally distributed} noting that since $E[2]\subseteq E(\QQ)$, we have that $\delta_{E,K}=1/2$.
\end{proof}

\subsection{Statistical results for Shafarevich--Tate groups}
A consequence of \Cref{cor:Erdos Kac For Selmer} is that $\dim \textup{Sel}^2(E_d/K)$ typically has size around $\log \log|d|$. By contrast, when $E$ has no cyclic $4$-isogeny defined over $\mathbb{Q}$, the dimensions of the $2$-Selmer groups of the $E_d$ over $\mathbb{Q}$ are known to be bounded on average thanks to a result of Kane \cite{kane2013ranks}*{Thm. 3}. Thus in this case   the majority of $\dim \textup{Sel}^2(E_d/K)$ is attributable to the Shafarevich--Tate group. Formalising this observation allows us to prove the analogue of  \Cref{cor:Erdos Kac For Selmer} for Shafarevich--Tate groups also. 
 
 \begin{cor}\label{sha corollary}
Assume that $E$ has no cyclic $4$-isogeny defined over $\mathbb{Q}$. Then the quantity 
\[\frac{\dim \Sha(E_d/K)[2]-\log\log|d| }{\sqrt{2\log\log|d|}}\]
follows a standard normal distribution.   That is,  for all $z\in \mathbb{R}$ we have 
	\[\lim_{X\to\infty}\frac{
	\#\set{\vert d\vert\leq X \textnormal{ squarefree}~:~\frac{\dim \Sha(E_d/K)[2]-\log\log|d|}{\sqrt{2\log\log|d|}}\leq z}
	}{\#\set{\vert d\vert\leq X \textnormal{ squarefree}}}=\frac{1}{\sqrt{2\pi}}\int_{-\infty}^ze^{-t^2/2}dt.\]
\end{cor}
\begin{proof}
Since $\dim\Sha(E_d/K)[2]\leq \dim \textup{Sel}^2(E_d/K)$ for all $d$, by \Cref{cor:Erdos Kac For Selmer} we need only show that the limit in the statement (or more precisely the limit superior of the left hand side of the statement)  is bounded above by $\Phi(z)=\tfrac{1}{\sqrt{2\pi}}\int_{-\infty}^ze^{-t^2/2}dt$. 

This  follows from \Cref{cor:Erdos Kac For Selmer} thanks to \cite{kane2013ranks}*{Thm. 3}, which gives adequate control of the Mordell--Weil component of $\textup{Sel}^2(E_d/K)$.  First, for any squarefree integer $d$, the standard short exact sequence 
\[0\longrightarrow E_d(K)/2E_d(K)\longrightarrow \textup{Sel}^2(E_d/K)\longrightarrow \Sha(E_d/K)[2]\longrightarrow 0\]
gives 
\[\dim \Sha(E_d/K)[2]=\dim \textup{Sel}^2(E_d/K)- \dim E_d(K)/2E_d(K).\]
Since $K=\QQ(\sqrt{\theta})$ and $\dim E_d(K)[2]=2$ we have 
\[\dim E_d(K)/2E_d(K)=2+\textup{rk}(E_d/\mathbb{Q})+\textup{rk}(E_{d\theta}/\mathbb{Q}),\] 
giving the equality \[\dim \Sha(E_d/K)[2]=\dim \textup{Sel}^2(E_d/K)-\textup{rk}(E_d/\mathbb{Q})-\textup{rk}(E_{d\theta}/\mathbb{Q})-2.\]

Now fix a real number $z$ and a positive real number $M$.  Partitioning into cases according to 
\[\textup{rk}(E_d/\mathbb{Q})+\textup{rk}(E_{d\theta}/\mathbb{Q})\leq M\quad \textup{ or }\quad \textup{rk}(E_d/\mathbb{Q})+\textup{rk}(E_{d\theta}/\mathbb{Q})>M\]  we find
\begin{align*}
	&\#\set{\vert d\vert\leq X\textnormal{ squarefree}~:~\frac{\dim \Sha(E_d/K)[2]-\log\log|d|}{\sqrt{2\log\log|d|}}\leq z}
	\\ &\leq \#\set{\vert d\vert\leq X\textnormal{ squarefree}~:~ \frac{\dim \textup{Sel}^2(E_d/K)-\log\log|d|}{\sqrt{2\log\log|d|}}\leq z+\frac{M+2}{\sqrt{2\log\log|d|}} } \\&\quad+ \#\Bigg \{\vert d\vert\leq X\textnormal{ squarefree}~:~\textup{Sel}^2(E_d/\mathbb{Q})> M/2\Bigg \}\\&\quad+ \#\Bigg \{\vert d\vert\leq X\textnormal{ squarefree}~:~\textup{Sel}^2((E_\theta)_d/\mathbb{Q})> M/2\Bigg \} .
	 \end{align*}
Dividing through by the number of squarefree integers $d$ with $|d|\leq X$, taking the limsup $X\to \infty$, and applying Kane's theorem \cite{kane2013ranks}*{Thm. 3} to both $E$ and $E_\theta$ (since $E$ has no cyclic $4$-isogeny defined over $\mathbb{Q}$ the same is true for $E_\theta$, allowing us to apply Kane's result without further assumptions), we find as a consequence of  \Cref{cor:Erdos Kac For Selmer} that
\begin{align*}
	&\limsup_{X\to\infty}\frac{\#\set{\vert d\vert\leq X\textnormal{ squarefree}~:~\frac{\dim \Sha(E_d/K)[2]-\log\log|d|}{\sqrt{2\log\log|d|}}\leq z}}
	{\#\set{\vert d\vert\leq X\textnormal{ squarefree}}}
	 &\leq \Phi(z)+2\sum_{r\geq M/2}\alpha_r,
\end{align*}
where the $\alpha_r$ are defined in Kane's Theorem 2.  Since the $\alpha_r$ determine a probability distribution on the set of $r\in\ZZ_{\geq0}$, taking the limit $M\to\infty$ gives the result.
\end{proof}

\begin{remark}
It seems reasonable to expect that \Cref{sha corollary} remains true without the assumption that $E$ has no cyclic $4$-isogeny defined over $\mathbb{Q}$. However, since no analogue of Kane's result is known in this setting we have not been able to prove this.
\end{remark}

\subsection{Statistical Results for Mordell--Weil groups}
We now give some consequences   for the Mordell--Weil groups of the $E_d/K$. We begin with the following algebraic results. Write $G=\textup{Gal}(K/\mathbb{Q})$.

\begin{notation}
We write
	\[\Lambda(E_d/K):=E_d(K)/E_d(K)_{\operatorname{tors}}.\]
	We refer to this as the Mordell--Weil lattice. The action of $G$ on $E_d(K)$ makes $\Lambda(E_d/K)$ into a $G$-module. 
\end{notation}

For a $G$-module $M$, we denote by $M(-1)$ the $G$-module which is isomorphic to $M$ as an abelian group but with $G$-action twisted by multiplication by $-1$. That is, the new $G$-action of the generator $\sigma$ of $G$ is given by
\[m\longmapsto -\sigma(m).\]
\begin{lemma}\label{lem:MWLatticesConcrete}
	If $\sel{\CCC_d}(\QQ, E_d[2])=0$ then there is an isomorphism of $\ZZ[G]$-modules
	\[\Lambda(E_d/K)\cong \Lambda(E_d/\QQ)\oplus \Lambda(E_{d\theta}/\QQ)(-1).\]
\end{lemma}
\begin{proof} 
	By \cite{MR1038525}*{Theorem 34.31}, there exist unique $a,b,c\in\ZZ_{\geq 0}$ such that
	\[\Lambda(E_d/K)\cong \ZZ^a\oplus\ZZ(-1)^{b}\oplus \ZZ[G]^c,\]
	where $\ZZ$ denotes a rank 1 free $\ZZ$-module with trivial $G$-action.  Note that we have an inclusion of $G$-modules
	\[\Lambda(E_d/K)/2\Lambda(E_d/K)\subseteq \sel{2}(E_d/K)/\delta(E_d[2]).\] 
	The right hand side has trivial $G$-action, as follows from the vanishing of  $\sel{\CCC_d}(\QQ, E_d[2])$ combined with Corollary \ref{what happens when selc vanishes} (iii). Thus    $\Lambda(E_d/K)/2\Lambda(E_d/K)$ 	has trivial $G$-action also.  Thus, $c=0$. Via the natural $K$-isomorphism $E_d\cong E_{d\theta}$, we  can identify the points of $E_d(K)$ on which the generator of $G$ acts as multiplication by $-1$ with $E_{d\theta}(\mathbb{Q})$. The result follows.
\end{proof}

\begin{proposition}\label{prop:MWGroupsConcrete} 
	Suppose we have $E_d(K)_{\operatorname{tors}} = E_d[2]$ and $\sel{\CCC_d}(\QQ, E_d[2])=0$. Then there is an isomorphism of $\ZZ[G]$-modules
	\[E_d(K)\cong \F_2^2\oplus \Lambda(E_d/\QQ)\oplus \Lambda(E_{d\theta}/\QQ)(-1).\]
\end{proposition}
\begin{proof}
	By \Cref{lem:MWLatticesConcrete} we must have 
	\begin{equation}\label{eq:splitting of lambda/K in MWGroupsConcrete proof}\Lambda(E_d/K)\cong \Lambda(E_d/\QQ)\oplus\Lambda(E_{d\theta}/\QQ)(-1).\end{equation}
	As a consequence, take $\mathcal{B}$ to be a $\ZZ$-basis for $\Lambda(E_d/K)$ such that for all $v\in \mathcal{B}$ we have $\sigma(v)\in\set{v,-v}$.  Let $\tilde{\mathcal{B}}$ be a lift of $\mathcal{B}$ to $E_d(K)$.  Note that $E_d(K)/2E_d(K)$ has a basis comprising of the images of the elements of $\tilde{\mathcal{B}}$ and two linearly independent vectors from the submodule $E_d(K)_{\operatorname{tors}}=E_d[2]\cong\F_2^2$.

	For each $v\in\tilde{\mathcal{B}}$, we have $\sigma(v)=\pm v+u$ for some $u\in E_d[2]$. Since $\sel{\CCC_d}(\QQ, E_d[2])=0$, the $G$-action on $E_d(K)/2E_d(K)$ is trivial by \Cref{what happens when selc vanishes}(iii).  In particular $\pm v+u=\sigma(v)\equiv v$ in $E_d(K)/2E_d(K)$, and so $u\in 2E_d(K)$.  Since $E_d(K)$ has no $4$--torsion, $u=0$ and so $\sigma(v)=\pm v$.  Thus the morphism of abelian groups $\Lambda(E_d/K)\to E_d(K)$ induced by the lift $\tilde{\mathcal{B}}$ of $\mathcal{B}$ is one of $\ZZ[G]$--modules, so we have
	\[E_d(K)\cong E_d[2]\oplus\Lambda(E_d/K).\]
	The result then follows from \Cref{eq:splitting of lambda/K in MWGroupsConcrete proof}.
\end{proof}

\begin{cor}\label{prop:MWGroupsStatistical}
	For $100\%$ of squarefree $d$, there is an isomorphism of $\ZZ[G]$-modules
	\begin{equation}\label{eq:DecompOfMWForStatistical}
	E_d(K)\cong \F_2^2\oplus \Lambda(E_d/\QQ)\oplus \Lambda(E_{d\theta}/\QQ)(-1).
	\end{equation}
	More precisely, we have
	\[\lim_{X\rightarrow \infty}\frac{\#\{d\textup{ squarefree }\mid ~|d|<X,~\eqref{eq:DecompOfMWForStatistical}\textnormal{ holds}\}}{\#\{d\textup{ squarefree }\mid ~|d|<X\}}=1.\]
\end{cor}
\begin{proof}
Note that for each odd prime $p$, at most 2 quadratic twists of $E$ have rational $p$-torsion (otherwise $E$ would have at least $3$ dimensional $p$-torsion over a multiquadratic extension, which is impossible). In particular, for each odd prime $p$, only finitely many twists of $E$ can have $p$-torsion over $K$. Consequently, by Mazur's torsion theorem \cite{MR488287}*{Theorem 8}, outside of a finite set of $d$ we have $E_d(K)_{\operatorname{tors}}\subseteq E[2^\infty]$.  Moreover, by Lemma \ref{lem:No4Torsion}, only finitely many quadratic twists have a point of order 4.  The result now follows from Theorem \ref{main statistical theorem} and Proposition \ref{prop:MWGroupsConcrete}.
\end{proof}

\subsection{Twists of the Weil Restriction of Scalars}

In light of \Cref{weil restriction subsection} we can recast the above results in terms of the restriction of scalars $\textup{Res}_{K/\mathbb{Q}}(E)$ of $E$ from $K$ to $\mathbb{Q}$. To ease notation we write $A$ in place of $\textup{Res}_{K/\mathbb{Q}}(E)$. Thus $A$ is a principally polarised abelian surface over $\mathbb{Q}$. For a squarefree integer $d$, we write $A_d$ for the quadratic twist of $A$ by $d$. For each $d$, write 
\[\phi_d: A_d \longrightarrow E_d\times E_{d\theta}\]
for the isogeny of \Cref{the isogeny phi_d}, write $\widehat{\phi}_d$ for its dual, and denote by $\textup{Sel}^\phi(A_d/\mathbb{Q})$ and $\textup{Sel}^{\widehat{\phi}_d}(E_d\times E_{d\theta}/\mathbb{Q})$   the associated Selmer groups.  

\begin{theorem} \label{main result for Weil restrictions}
In the notation above, we have the following.
\begin{itemize}
\item[(i)] The quantity 
\[\frac{\dim \textup{Sel}^2(A_d/\mathbb{Q})-\log \log |d|}{\sqrt{2\log\log|d|}}\]
follows a standard normal distribution.  That is,  for every $z\in\mathbb{R}$ we have 
	\[\lim_{X\to\infty}\frac{
	\#\set{\vert d\vert\leq X \textnormal{ squarefree }: \ \frac{ \dim\textup{Sel}^2(A_d/\mathbb{Q})-\log\log|d|}{\sqrt{2\log\log|d|}}\leq z}
	}{
	\#\set{\vert d\vert\leq X \textnormal{ squarefree }}
	}=\frac{1}{\sqrt{2\pi}}\int_{-\infty}^ze^{-t^2/2}dt.\]
\item[(ii)] For $100\%$ of squarefree $d$ ordered by absolute value, the  group $\textup{Sel}^{\widehat{\phi}_d}(E_d\times E_{d\theta}/\mathbb{Q})$ is trivial, and we have 
\[\dim \textup{Sel}^2(A_d/\mathbb{Q})=\dim \textup{Sel}^{\phi_d}(A_d/\mathbb{Q})-2,\]
and $\dim \textup{Sel}^2(A_d/\mathbb{Q})$ is given by the formula on the right hand side of \Cref{Selmer formula 100 percent}.
\end{itemize}
\end{theorem}

\begin{proof}
The first part follows from \Cref{cor:Erdos Kac For Selmer}  and \Cref{weil_res_quad_twist}.  Using \Cref{geometric sel f}, \Cref{sel c remark} and \Cref{weil_res_quad_twist} the second part is then an immediate consequence of \Cref{main statistical theorem written in sec 5}, \Cref{what happens when selc vanishes}(i) and \Cref{cor holds a lot}.
\end{proof}

\begin{remark}
Assume that $E$ has no cyclic $4$-isogeny defined over $\mathbb{Q}$. Then, since $\Sha(A_d/\mathbb{Q})\cong \Sha(E_d/K)$, as a consequence of \Cref{sha corollary} we can replace $ \textup{Sel}^2(A_d/\mathbb{Q})$ with $\Sha(A_d/\mathbb{Q})[2]$ in \Cref{main result for Weil restrictions}.
\end{remark}

\section{Explicit local conditions for full \texorpdfstring{$2$}{2}-torsion} \label{case of full 2-tors section}

In this section we make preparations for the proof of \Cref{main statistical theorem written in sec 5} by making the results of \Cref{sec:seloverquad} explicit in the case that $E$ has full rational $2$-torsion. 

Recall that $K=\mathbb{Q}(\sqrt{\theta})/\mathbb{Q}$ is a quadratic extension and $E/\mathbb{Q}$ is a fixed elliptic curve with $E[2]\subseteq E(\mathbb{Q})$.  Further, we fix a Weierstrass equation
\begin{equation} \label{choice of weierstrass equation}
E/\mathbb{Q}:y^2=(x-a_1)(x-a_2)(x-a_3)
\end{equation}
for $E$ where, without loss of generality, $a_1, a_2, a_3\in \mathbb{Z}$. Set $\alpha=a_1-a_2$, $\beta=a_1-a_3$, and $\gamma=a_2-a_3$. Note that the primes of bad reduction for $E$ all divide $2\alpha \beta \gamma$, and that $E[2]=\{O,P_{1},P_{2},P_{3}\}$ where $P_{i}=(a_i,0)$.

As in \Cref{Notation:Sigma} we fix a finite set $\Sigma$ of places of $\mathbb{Q}$ containing the real place, the prime $2$, all primes which ramify in $K/\mathbb{Q}$, and all primes at which $E$ has bad reduction. Note in particular that $\Sigma$ contains all primes dividing $2\alpha \beta \gamma$.

\subsection{Quadratic twists}
Let $d$ be a squarefree integer. The quadratic twist $E_d/\mathbb{Q}$ is given by the Weierstrass equation
\[E_d:y^2=(x-da_1)(x-da_2)(x-da_3).\]
We have $E_d[2]=\{O,P_{1,d},P_{2,d},P_{3,d}\}$ where $P_{i,d}=(da_i,0)$.

The following lemma describes the local conditions $\CCC(E_d/\mathbb{Q}_v)$ of \Cref{the selmer structure f}  at primes $p\notin \Sigma$. For a place $v$ of $\mathbb{Q}$, we denote by $\delta_{d,v}:E_d(\mathbb{Q}_v)/2E_d(\mathbb{Q}_v)\hookrightarrow H^1(\mathbb{Q}_v,E_d[2]) $ the coboundary map associated to the sequence \Cref{elliptic curve kummer sequence} with $A=E_d$     and $F=\mathbb{Q}$.

\begin{lemma} \label{the local conditions lemma}
Let $p$ be a prime with $p\notin \Sigma$. Then 
\begin{itemize}
\item[(i)] if $p\nmid d$, we have
\[\CCC(E_d/\mathbb{Q}_p)=\SSS(E_d/\mathbb{Q}_p)=H^1_{\textup{nr}}(\mathbb{Q}_p,E_d[2]),\]
\item[(ii)] if $p\mid d$ is split in $K/\mathbb{Q}$, we have
\[\CCC(E_d/\mathbb{Q}_p)=\SSS(E_d/\mathbb{Q}_p)=\delta_{d,p}(E_d[2]),\]
\item[(iii)] if $p\mid d$ is inert in $K/\mathbb{Q}$, we have 
\[\CCC(E_d/\mathbb{Q}_p)=0.\]
\end{itemize} 
\end{lemma}

\begin{proof}
Let $\mathfrak{p}$ be a prime of $K$ lying over $p$. 
(i): By \Cref{explicit_coker} we have  $N_{K_\mathfrak{p}/\mathbb{Q}_p}E_d(K_\mathfrak{p})=E_d(\mathbb{Q}_p)$.  The first equality in \Cref{the properties of the selmer structure C}(ii) thus gives
\[\CCC(E_d/\mathbb{Q}_p)=\delta_p(E_d(\mathbb{Q}_p))=\SSS(E_d/\mathbb{Q}_p).\]  The second equality follows from the fact that $p$ is odd and $E_d$ has good reduction at $p$. 

(ii): when $p$ splits in $K/\mathbb{Q}$ the local extension $K_\mathfrak{p}/\mathbb{Q}_p$ is trivial, so $\CCC(E_d/\mathbb{Q}_p)=\SSS(E_d/\mathbb{Q}_p)$ by definition. For the second equality, since $p\nmid 2\infty$,  we have $\dim\SSS(E_d/\mathbb{Q}_p)=\dim E_d[2]$. In particular, it suffices to show that the restriction of $\delta_{d,p}$ to $E_d[2]$ is injective, which follows from \Cref{lem:No4Torsion}.

(iii): by \Cref{explicit_coker} and the fact that $E$ has full $2$-torsion, it follows from a dimension count that $N_{K_\mathfrak{p}/\mathbb{Q}_p}E(K_\mathfrak{p})=2E(\mathbb{Q}_p)$. The result now follows from \Cref{the properties of the selmer structure C}.
\end{proof}

\begin{remark}
Taking orthogonal complements, the above result also determines the local groups $\FFF(E_d/\mathbb{Q}_p)$ for $p\notin \Sigma$.
\end{remark}

\subsection{Explicit local conditions}  \label{explicit conditions subsection}
We now use the fact that $E_d$ has full rational $2$-torsion to give an explicit description of $\textup{Sel}_{\CCC_d}(\mathbb{Q},E_d[2])$ as a subgroup of  $(\mathbb{Q}^\times/\mathbb{Q}^{\times 2})^2$.

Let $\lambda_{i,d}:E_d[2]\rightarrow \boldsymbol \mu_2$ be the map $P\mapsto (P,P_{i,d})_{e_2}$, where $(~,~)_{e_2}:E_d[2]\times E_d[2]\rightarrow \boldsymbol \mu_2$ is the Weil pairing. This induces an isomorphism
\[(\lambda_{1,d},\lambda_{2,d}):E_d[2]\stackrel{\sim}{\longrightarrow}\boldsymbol \mu_2 \times \boldsymbol \mu_2.\]
Via this map, we identify $H^1(\mathbb{Q},E_d[2])$ with $H^1(\mathbb{Q},\boldsymbol \mu_2)\oplus H^1(\mathbb{Q},\boldsymbol \mu_2)=(\mathbb{Q}^\times/\mathbb{Q}^{\times 2})^2 $ (cf. \Cref{hilbet symbol example}).
 We similarly identify $H^1(\mathbb{Q}_v,E_d[2])$ with $(\mathbb{Q}_v^\times/\mathbb{Q}_v^{\times 2})^2$ for each place $v$ of $\mathbb{Q}$. In this description, for each place $v$ of $\mathbb{Q}$, the local Tate pairing
\[\left \langle ~,~\right \rangle_v: H^1(\mathbb{Q}_v,E_d[2]) \times H^1(\mathbb{Q}_v,E_d[2])\rightarrow \mathbb{Q}/\mathbb{Z}\]
becomes the pairing $(\mathbb{Q}_v^\times/\mathbb{Q}_v^{\times 2})^2\times (\mathbb{Q}_v^\times/\mathbb{Q}_v^{\times 2})^2\rightarrow \frac{1}{2}\mathbb{Z}/\mathbb{Z}  \cong \boldsymbol \mu_2  $ given by
\begin{equation} \label{explicit tate pairing}
\left((x_1,x_2),(y_1,y_2)\right)\mapsto (x_1,y_2)_v(x_2,y_1)_v,
\end{equation}
where $(~,~)_v$ denotes the quadratic Hilbert symbol. The Kummer map
$\delta_{d,v}:E_d(\mathbb{Q}_v)/2E_d(\mathbb{Q}_v)\hookrightarrow H^1(\mathbb{Q}_v,E_d[2]) $
then becomes the map  
 \begin{equation} \label{elliptic curve kummer map}
(x,y)\longmapsto  \begin{cases}  (x-da_1,x-da_2)~~&~~x\notin \{da_1,da_2\}, \\\left(\alpha\beta,d\alpha\right) & (x,y)=(da_1,0),\\\left(-d\alpha,-\alpha \gamma\right) & (x,y)=(da_2,0).
\end{cases}\end{equation}
See, for example, \cite{MR2514094}*{Proposition X.1.4}.

\subsection{The group $\textup{Sel}_{\tilde{\CCC}_d}(\mathbb{Q},E_d[2])$}

We now define a further Selmer structure, whose associated Selmer group contains $\textup{Sel}_{\CCC_d}(\mathbb{Q},E_d[2])$ as a subgroup, and which admits a cleaner explicit description. 

\begin{defi} \label{sel c tilde defi}
Define the Selmer structure $\widetilde{\CCC}_d$ for $E_d[2]$ (viewed as a $G_\mathbb{Q}$-module) via the local conditions
\[\widetilde{\CCC}(E_d/\mathbb{Q}_v)=\begin{cases}\CCC(E_d/\mathbb{Q}_v)~~&~~v\notin \Sigma, \\H^1(\mathbb{Q}_v,E[2])~~&~~v\in \Sigma. \end{cases}\]
Denote by $\textup{Sel}_{\tilde{\CCC}_d}(\mathbb{Q},E_d[2])$ the associated Selmer group. 
\end{defi}

Note that by construction,  $\textup{Sel}_{\widetilde{\CCC}_d}(\mathbb{Q},E_d[2])$ contains  $\textup{Sel}_{\CCC_d}(\mathbb{Q},E_d[2])$ as a subgroup. In particular, if $\textup{Sel}_{\widetilde{\CCC}_d}(\mathbb{Q},E_d[2])$ is trivial,  then so is $\textup{Sel}_{\CCC_d}(\mathbb{Q},E_d[2])$. The advantage of considering $\textup{Sel}_{\widetilde{\CCC}_d}(\mathbb{Q},E_d[2])$ is that now \Cref{the local conditions lemma} describes all  non-trivial Selmer conditions.  

\begin{notation} \label{preliminary_notat}
Write $N$ for the squarefree product of all (finite) primes $p\in \Sigma$. Further, write $d=ad'd''$, where $d'$ is the product of all primes $p\mid d$ such that both $p\notin \Sigma$ and $p$ splits in $K/\mathbb{Q}$, and $d''$ is the product of all primes $p\mid d$ such that both $p\notin \Sigma$ and $p$ is inert in $K/\mathbb{Q}$.  

For  $d\in \mathbb{Z}$ squarefree, we identify $H^1(\mathbb{Q},E_d[2])$ with  $(\mathbb{Q}^\times/\mathbb{Q}^{\times 2})^2$ as in \Cref{explicit conditions subsection}, and further identify $(\mathbb{Q}^\times/\mathbb{Q}^{\times 2})^2$ with the set of pairs of squarefree integers.  For a prime $p$ and an integer $n$ coprime to $p$, we write 
$\left(\frac{n}{p}\right)$ for the Legendre symbol taking value $1$ if $n$ is a square modulo $p$, and $-1$ else. 
\end{notation}

\begin{proposition} \label{explicit selmer conditions 1}
With the notation and identifications of \Cref{preliminary_notat}, the Selmer group $\textup{Sel}_{\widetilde{\CCC}_d}(\mathbb{Q},E_d[2])$ consists of pairs $(x_1,x_2)$ of squarefree integers such that the following conditions all hold:
\begin{itemize}
\item[(i)] we have $x_i \mid Nd'$ for $i=1,2$,
\item[(ii)] we have $\left(\frac{x_i}{p}\right)=1$ for all $p\mid d''$ and for $i=1,2$,
\item[(iii)] for all $p\mid d'$ we have \[(x_1,d\alpha)_p(x_2,\alpha \beta)_p=1=(x_1,-\alpha \gamma)_p(x_2,-d\alpha)_p.\] 
\end{itemize}
\end{proposition}

\begin{proof}
By \Cref{the local conditions lemma} and the definition of the local groups  $\widetilde{\CCC}(E_d/\mathbb{Q}_v)$, we have $\widetilde{\CCC}(E_d/\mathbb{Q}_p)=0$ for all primes $p$ with $p\notin \Sigma$ such that both $p\mid d$ and $p$ is inert in $K/\mathbb{Q}$, and $\widetilde{\CCC}(E_d/\mathbb{Q}_p)=H^1_{\textup{nr}}(\mathbb{Q}_p,E_d[2])$ for each prime $p$ such that both $p\notin \Sigma$ and $p\nmid d$. These conditions are equivalent to conditions (i) and (ii) in the statement. Since in the definition of $\textup{Sel}_{\widetilde{\CCC}_d}(\mathbb{Q},E_d[2])$ there are no conditions imposed at primes $p\in \Sigma$, in light of \Cref{the local conditions lemma}(ii) it suffices to show that condition (iii) is equivalent to the condition that \[(x_1,x_2)\in \SSS(E_d/\mathbb{Q}_p)=\delta_{d,p}(E_d[2])\] for each prime $p\mid d$ such that both $p\notin \Sigma$ and $p$ splits in $K/\mathbb{Q}$. Since $\SSS(E_d/\mathbb{Q}_p)$ is its own orthogonal complement under the local Tate pairing, $(x_1,x_2)$ is in $\SSS(E_d/\mathbb{Q}_p)$ if and only if it pairs trivially with each element of $\delta_{d,p}(E_d[2])$. Now $P_{d,1}=(da_1,0)$ and $P_{d,2}=(da_2,0)$ is a basis for $E_d[2]$, and by  \Cref{elliptic curve kummer map} we have
\[\delta_{d,p}(P_{d,1})=(\alpha\beta,d\alpha)\in \left(\mathbb{Q}_p^{\times}/\mathbb{Q}_{p}^{\times 2}\right)^2\quad\textup{ and }\quad\phantom{h}~~\delta_{d,p}(P_{d,2})=(-d\alpha,-\alpha\gamma)\in \left(\mathbb{Q}_p^{\times}/\mathbb{Q}_{p}^{\times 2}\right)^2.\]
By \Cref{explicit tate pairing}, $(x_1,x_2)$ pairs trivially with both of these elements under the local Tate pairing at $p$ if and only if 
\[(x_1,d\alpha)_p(x_2,\alpha \beta)_p=1=(x_1,-\alpha \gamma)_p(x_2,-d\alpha)_p.\] 
The result follows. 
\end{proof}

\section{Proof of \texorpdfstring{\Cref{main statistical theorem written in sec 5}}{main theorem}} \label{big analysis section}

Recall that $K=\mathbb{Q}(\sqrt{\theta})/\mathbb{Q}$ is a quadratic extension with Galois group $G$, and $E/\mathbb{Q}$ is an elliptic curve over $\mathbb{Q}$ with $E[2]\subseteq E(\mathbb{Q})$, and given by a Weierstrass equation
\begin{equation}  
E/\mathbb{Q}:y^2=(x-a_1)(x-a_2)(x-a_3)
\end{equation}
for $a_1, a_2, a_3\in \mathbb{Z}$. Recall also that we have defined integers $\alpha=a_1-a_2$, $\beta=a_1-a_3$, and $\gamma=a_2-a_3$, and that the integer $N$ is taken to be the product of all primes in the set $\Sigma$ of \Cref{Notation:Sigma}.

The aim of this section is to prove \Cref{main statistical theorem written in sec 5}.  Specifically, we will show the following, strictly stronger, result. 

\begin{theorem} \label{main statistical theorem}  
We have
\[\#\{d\textup{ squarefree }\colon ~|d|<X,~\textup{Sel}_{\CCC_d}(\mathbb{Q},E_d[2])\neq 0\}\ll X\log(X)^{-0.0394}.\]
In particular 
\[\lim_{X\rightarrow \infty}\frac{\#\{d\textup{ squarefree }\colon ~|d|<X,~\textup{Sel}_{\CCC_d}(\mathbb{Q},E_d[2])=0\}}{\#\{d\textup{ squarefree }\colon ~|d|<X\}}=1.\]
\end{theorem}

\subsection{First reduction}

In order to prove \Cref{main statistical theorem} it suffices to prove the identical result for 
$\textup{Sel}_{\widetilde{\CCC}_d}(\mathbb{Q},E_d[2])$ (cf. \Cref{sel c tilde defi}) in place of $\textup{Sel}_{\CCC_d}(\mathbb{Q},E_d[2])$, since the latter is a subgroup of the former. We begin by defining a further group $S_d$ determined by simpler local conditions. Specifically, we wish to `decouple' the variables $x_1$ and $x_2$ appearing in \Cref{explicit selmer conditions 1}. 
We first introduce some notation. 

\begin{notation}
We introduce the following $3$ sets of primes:
\[\mathcal{P}_0:=\{ p\notin \Sigma,~~p \textup{ split in }K/\mathbb{Q}, \textup{ and }p\textup{ non-split }\mathbb{Q}(\sqrt{\alpha \beta})/\mathbb{Q} \},\]
\[\mathcal{P}_1:=\{ p\notin \Sigma,~~p \textup{ split in }K/\mathbb{Q}, \textup{ and }p\textup{ split in }\mathbb{Q}(\sqrt{\alpha \beta})/\mathbb{Q} \},\]
\[\mathcal{P}_2:=\{ p\notin \Sigma,~~p \textup{ inert in }K/\mathbb{Q}\}.\]
(If $\alpha\beta$ is a square in $\mathbb{Q}$ we take $\mathcal{P}_0:=\emptyset$ and $\mathcal{P}_1$ the collection of primes not in $\Sigma$ which split in $K/\mathbb{Q}$.)
Note that the sets $\Sigma, \mathcal{P}_0, \mathcal{P}_1$ and $\mathcal{P}_2$ give a partition of the set of all primes into $4$ pairwise disjoint subsets. 

For $i=0,1,2$, we  define $\mathcal{F}_i$ to be the set of positive squarefree integers $n$ all of whose prime factors lie in $\mathcal{P}_i$. Note that for $i\neq j$ we have $\mathcal{F}_i\cap \mathcal{F}_j=\{1\}$.  We write $\mathcal{F}_i\cdot \mathcal{F}_j$ for the collection of squarefree integers $n$ which can be written as a product $n=n_i n_j$ for some $n_i\in \mathcal{F}_i$ and $n_j\in \mathcal{F}_j$. Note that such a decomposition is necessarily unique.
\end{notation}

\begin{remark} \label{density remark}
Note that provided $\mathbb{Q}(\sqrt{\alpha\beta})\nsubseteq K$, $\mathcal{P}_0$ and $\mathcal{P}_1$ have Dirichlet density  $1/4$, and $\mathcal{P}_2$ has density $1/2$. If $\mathbb{Q}(\sqrt{\alpha \beta})\subseteq K$ then $\mathcal{P}_0=\emptyset$ and $\mathcal{P}_1$ and $\mathcal{P}_2$ both have Dirichlet density $1/2$.
\end{remark}

\begin{defi} \label{definition of Sd}
For $d$ a squarefree integer, define the subgroup $S_d$  of $\mathbb{Q}^\times/\mathbb{Q}^{\times 2}$ as follows. 
First, write (uniquely) $d=ad_0d_1d_2$ where $a\mid N$, $d_0\in \mathcal{F}_0$, $d_1\in \mathcal{F}_1$, and $d_2\in \mathcal{F}_2$. Now define $S_d$ to be the set of squarefree integers 
\[S_d:=\set{x~\textup{sq. free}~:~\substack{x\mid Nd_0d_1,\\\left(\frac{x}{p}\right)=1\textup{ for all } p\mid d_2,\\(x,d\alpha)_p=1\textup{  for all }p\mid d_1.}}.\]
 We allow $x$ to be either positive or negative.
\end{defi}

\begin{lemma} \label{x-coordinate lemma trivial}
If a pair of squarefree integers $(x_1,x_2)$ is in $\textup{Sel}_{\widetilde{\CCC}_d}(\mathbb{Q},E_d[2])$, then $x_1\in S_d$.
\end{lemma}

\begin{proof}
Immediate from \Cref{explicit selmer conditions 1}, noting that for $p\mid d_1$, since $p$ is split in $\mathbb{Q}(\sqrt{\alpha \beta})/\mathbb{Q}$ by assumption, the condition $(x_1,d\alpha)_p(x_2,\alpha \beta)_p=1$ simply becomes $(x_1,d\alpha)_p=1$.
\end{proof}

We will show the following. As explained below, this is sufficient to prove \Cref{main statistical theorem}.

\begin{theorem} \label{thing we average}
We have
\[\#\{ d\textup{ squarefree }\mid ~|d|<X,~S_d\neq 0\}\ll X \textup{log}(X)^{-0.0394}.\]
\end{theorem}

\begin{proof}[Proof of  \Cref{main statistical theorem} assuming \Cref{thing we average}.]
\Cref{thing we average} combined with \Cref{x-coordinate lemma trivial} shows that the $x_1$-coordinate of any element of $\textup{Sel}_{\widetilde{\CCC}_d}(\mathbb{Q},E_d[2])$ is trivial for $100\%$ of squarefree $d$. By symmetry, the same must then be true of the $x_2$-coordinate since we can relabel $a_1$ and $a_2$ in the equation \Cref{choice of weierstrass equation} for our elliptic curve in order to interchange the roles of $x_1$ and $x_2$. This shows the limit statement of \Cref{main statistical theorem}, and running the same argument but keeping track of error terms proves the general result.
\end{proof}

We now begin preparations for the proof of \Cref{thing we average}.

\subsection{Notation and preparations}

\begin{notation}
Given a positive integer $n$ we write $\omega(n)$ for the number of distinct prime factors of $n$. For $i=0,1,2$ we write $\omega_i(n)$ for the number of distinct prime factors of $n$ which lie in $\mathcal{P}_i$. We denote by $\mu$ the M\"obius function. 
\end{notation}

We will use frequently the following lemma controlling generalised divisor sums. 

\begin{lemma} \label{shiu lemma}
Let $a_0$, $a_1$, and $a_2$ be non-negative real numbers. Then we have 
\[\sum_{\substack{X-Y<n\leq X\\n~\textup{sq. free}}}a_0^{\omega_0(n)}a_1^{\omega_1(n)}a_2^{\omega_2(n)}\ll \begin{cases} Y \log(X)^{\frac{a_0}{4}+\frac{a_1}{4}+\frac{a_2}{2}-1}~~&~~\mathbb{Q}(\sqrt{\alpha\beta})\nsubseteq K,\\ Y\log(X)^{\frac{a_1}{2}+\frac{a_2}{2}-1}~~&~~\mathbb{Q}(\sqrt{\alpha\beta})\subseteq K,\end{cases}\]
uniformly for $2\leq X\textup{exp}(-\sqrt{\log(X)})\leq Y\leq X$. 
\end{lemma}

\begin{proof}
This follows from a (significantly more general)  result of Shiu \cite{MR552470}. Define the multiplicative function $f:\mathbb{Z}\rightarrow \mathbb{R}_{\geq 0}$ by setting, for any $k\geq 1$, $f(p^k)=a_i$ for $p\in \mathcal{P}_i$ ($i=0,1,2$), and taking  $f(p)=1$ for $p\in \Sigma$. We then wish to bound the sum $\sum_{X-Y<n\leq X}f(n)$.  It follows from \Cref{density remark} that we have
\[\sum_{p\leq X}\frac{f(p)}{p}\sim \begin{cases} \left(\frac{a_0}{4}+\frac{a_1}{4}+\frac{a_2}{2}\right)\log \log(X)~~&~~\mathbb{Q}(\sqrt{\alpha \beta})\subsetneq K,\\  \left(\frac{a_1}{2}+\frac{a_2}{2}\right) \log \log(X)~~&~~\mathbb{Q}(\sqrt{\alpha \beta})\subseteq K.\end{cases}\]
The result now follows from \cite{MR552470}*{Theorem 1} (the conditions (i) and (ii) needed for that theorem follow in our setting from well known bounds on the divisor function).
\end{proof}

\subsection{Reduction to computing a weighted average}

In order to prove \Cref{thing we average} we will compute bounds for a certain weighted average of $\#(S_d\setminus \{1\})$. Specifically will we prove:

\begin{proposition} \label{weighted average} 
For any $1<\gamma<7/8+\sqrt{17}/8=1.3903...$, we have
\begin{equation} \label{weighted average proposition}
\sum_{|d|<X,~d~\textup{sq. free}}\gamma^{\omega_2(d)-\omega_0(d)}(\# S_d-1)=o(X).\end{equation}
Moreover, for $\gamma=1/4+\sqrt{17}/4$ the left hand side of \Cref{weighted average proposition} is $\ll X\textup{log}(X)^{-0.0394}.$
\end{proposition}

We begin by showing that this is sufficient to prove \Cref{thing we average}.  

\begin{proof}[Proof of \Cref{thing we average} assuming \Cref{weighted average}]
We first show that the weights are at least $1$ for $100\%$ of squarefree $d$. That is, we claim that 
\[\#\{d~\textup{squarefree}~~\big\vert ~|d|\leq X, ~ \omega_0(d)\geq \omega_2(d)\}\ll X\log(X)^{-0.042}.\]
To see this, fixing any $\lambda>1$ we have
\[\#\{d~\textup{squarefree}~~\big\vert ~|d|\leq X, ~ \omega_0(d)\geq \omega_2(d)\}\leq 2\sum_{1\leq d\leq X}\lambda^{\omega_0(d)-\omega_2(d)}.\]
By \Cref{shiu lemma} the right hand side is $\ll X\log(X)^{\lambda/4+1/(2\lambda)-3/4}.$ Optimising over $\lambda$ we find that when $\lambda=\sqrt{2}$ the exponent is $1/\sqrt{2}-3/4<-0.042$, giving the claim. 

Now fix $1<\gamma<7/8+\sqrt{17}/8$. By the claim we have 
 \begin{eqnarray*}
\#\left\{|d|\leq X~~ \big \vert~~  S_d\neq 0\right\}&\leq &\#\left\{|d|\leq X~~\big \vert~~\omega_0(d)\geq \omega_2(d)\right\}+\\~~&~~&\#\left\{|d|\leq X~~\big \vert ~~  S_d\neq 0,~\omega_2(d)>\omega_0(d)\right\}\\
~~&\ll&X\log(X)^{-0.042}+\sum_{|d|\leq X}\gamma^{\omega_2(d)-\omega_0(d)}(\# S_d-1)
\end{eqnarray*}
 where  above $d$ is implicitly taken squarefree. The result now follows from \Cref{weighted average}.
\end{proof}

\begin{remark}
The reason for the introduction of the weight $\gamma$ is that, in passing from the group $\textup{Sel}_{\CCC_d}(\mathbb{Q},E_d[2])$ to the group $S_d$, we have thrown away the Selmer conditions coming from primes in $\mathcal{P}_0$ in favour of reducing the number of variables involved. This leads to twists having an abnormally large number of prime factors lying in $\mathcal{P}_0$ contributing a disproportionate amount to the average size of $S_d$. The weight $\gamma$ is introduced to compensate for this.  
\end{remark}

\subsection{Strategy of the proof of \Cref{weighted average}}

The proof of \Cref{weighted average} follows closely the argument of \cite{MR2276261}*{\S5}, which has its origins in the work of Heath-Brown \cites{MR1193603,MR1292115}. There, Fouvry--Kl\"uners determine asymptotics for the moments of $4$-ranks of class groups of quadratic fields. Our first step is to express the sum in \Cref{weighted average} as a sum of Jacobi symbols. We do this in \Cref{big sum lemma} below, using ideas from \cite{MR2276261}*{Lemma 16}. The resulting sum, given in  \Cref{the basic sum}, is structurally  similar to the one in \cite{MR2276261}*{Lemma 17}. We then adapt the techniques used by Fouvry--Kl\"uners to bound this sum. There are a couple of points at which the argument we give diverges from that of  Fouvry--Kl\"uners. First, whilst they compute higher moments of the sizes of class groups, we need only compute a (weighted version of) the first moment of the size of $S_d$. Thus the intricate study of `maximal unlinked subsets' undertaken in  \cite{MR2276261}*{\S5.6} can be avoided. On the other had, the variables $D_i$ in \cite{MR2276261}*{\S5} are allowed to vary over all positive squarefree integers, whilst ours are constrained to lie in the thin families $\mathcal{F}_j$. This necessitates changes to the argument in Fouvry--Kl\"uners' first and  fourth families, which correspond to our \Cref{Remaining families} and  \Cref{SW family} respectively. 

\subsection{Expressing the sum in terms of Jacobi symbols} \label{jacobi sums section}
We now begin preparations for the proof of \Cref{weighted average}  by expressing the relevant sum in terms of Jacobi symbols. We first define the following sums which will be ubiquitous in what follows.

\begin{defi}
Let $\lambda$ and $\eta$ be squarefree divisors (either positive or negative) of $N$. For a tuple $(D_i)_{0\leq i\leq 7}$ of coprime positive odd integers, write 
\begin{eqnarray*}
\mathcal{J}_{\eta,\lambda}((D_i)_{0\leq i\leq 7})&:=&\left(\frac{\eta}{D_2}\right)\left(\frac{\lambda}{D_4}\right)\left(\frac{\lambda}{D_6}\right)\left(\frac{D_4}{D_2}\right)
 \left(\frac{D_2}{D_4}\right)\left(\frac{D_6}{D_2}\right)\left(\frac{D_2}{D_6}\right)\\ & &   \times \left(\frac{D_1}{D_2}\right)\left(\frac{D_5}{D_2}\right)\left(\frac{D_7}{D_2}\right) \left(\frac{D_0}{D_4}\right)\left(\frac{D_3}{D_4}\right)\left(\frac{D_0}{D_6}\right)\left(\frac{D_3}{D_6}\right).
 \end{eqnarray*}
 Now for any real number $X>1$, and any positive real $\gamma$, define
 \begin{eqnarray*}
\mathcal{S}_\gamma(\lambda,\eta,X) :=\sum_{\substack{D_0,D_1\in \mathcal{F}_0\\ D_2,D_3,D_4,D_5\in \mathcal{F}_1\\ D_6,D_7\in \mathcal{F}_2\\\prod_iD_i\leq X\\D_i\textup{ coprime}}} \gamma^{-\omega(D_0D_1)}2^{-\omega(D_2D_3D_4D_5)}(2/\gamma)^{-\omega(D_6D_7)}\mathcal{J}_{\eta,\lambda}((D_i)_{0\leq i\leq 7}),
\end{eqnarray*}
with the additional condition that, if $\lambda=1$, then not all of $D_0,D_2$ and $D_3$ are equal to $1$ in the range of summation.
\end{defi}

\begin{lemma} \label{big sum lemma}
For any positive real number $\gamma$ we have
\begin{equation}
\sum_{|d|<X,~d~\textup{sq. free}}\gamma^{\omega_2(d)-\omega_0(d)}(\# S_d-1)=\sum_{a\mid N}\sum_{x_N\mid N}\mathcal{S}_\gamma(x_N,-ax_N\alpha,X/a)
\end{equation}
where the right hand sums run over both positive and negative divisors of $N$.
\end{lemma}

\begin{proof}
Fix $d$ squarefree. As in \Cref{definition of Sd} we write  $d=ad_0d_1d_2$ where $d_i\in \mathcal{F}_i$ for $i=0,1,2$, so that 
\[S_d=\set{x~\textup{sq. free}~:~\substack{x \mid Nd_0d_1,\\\left(\frac{x}{p}\right)=1\textup{ for all } p\mid d_2,\\(x,d\alpha)_p=1\textup{  for all }p\mid d_1.}}.\]
Now fix $x\mid Nd_0d_1$ and note that we have
\begin{equation} \label{jacobi identity}
2^{-\omega(d_2)}\sum_{z_2\mid d_2}\left(\frac{x}{z_2}\right)=2^{-\omega(d_2)}\prod_{p\mid d_2}\left(1+\left(\frac{x}{p}\right)\right)=\begin{cases}1~~&~~\left(\frac{x}{p}\right)=1~~\textup{for all }p\mid d_2\\ 0~~&~~\textup{else,}\end{cases}
\end{equation}
where in the sum above $z_2$ runs over all \textit{positive} divisors of $d_2$.

To deal with the conditions at primes dividing $d_1$, we write $x$ uniquely as $x=x_Nx_0x_1$ where $x_N\mid N$ (and may be negative) $x_0\mid d_0$ and $x_1\mid d_1$. Say $d_0=x_0y_0$ and $d_1=x_1y_1$. Then for $p\mid d_1$, we have (noting that $d_1$ and $\alpha$ are coprime and that all $p\mid d_1$ are odd)
\[(x,d\alpha)_p=\begin{cases} \left(\frac{x}{p}\right)~~&~~p\mid y_1\\(x,-xd\alpha)_p=\left(\frac{-ax_Ny_0y_1d_2\alpha}{p}\right)~~&~~p\mid x_1.\end{cases}\]
Thus, similarly to \Cref{jacobi identity}, we have
\begin{equation} \label{jacobi identity 2}
2^{-\omega(x_1y_1)}\sum_{\substack{w_1\mid x_1\\ z_1\mid y_1}}\left(\frac{x}{z_1}\right)\left(\frac{-ax_Ny_0y_1d_2\alpha}{w_1}\right)=\begin{cases}1~~&~~(x,d\alpha)_p=1~~\textup{for all }p\mid d_1\\ 0~~&~~\textup{else.}
\end{cases}
\end{equation}
We now multiply \Cref{jacobi identity} and \Cref{jacobi identity 2}, write $d_2=z_2z_2'$, $x_1=w_1w_1'$, and $y_1=z_1z_1'$, and sum over all $x=x_Nx_0x_1$ dividing $Nd_0d_1$ to find 
\[\#S_d=\sum_{x_N\mid N}\sum_{\substack{x_0y_0=d_0\\ w_1w_1'z_1z_1'=d_1\\ z_2z_2'=d_2}}2^{-\omega(w_1w_1'z_1z_1'z_2z_2')}\left(\frac{x_Nx_0w_1w_1'}{z_2}\right)\left(\frac{x_Nx_0w_1w_1'}{z_1}\right)\left(\frac{-ax_Ny_0z_1z_1'z_2z_2'\alpha}{w_1}\right)\]
where $x_N$ may be negative but all other variables are positive and coprime. Note that we necessarily have $x_0,y_0\in \mathcal{F}_0$, $w_1,w_1',z_1,z_1'\in \mathcal{F}_1$ and $z_2,z_2'\in \mathcal{F}_2$, so that in particular $\omega_0(d)=\omega(x_0y_0)$ and $\omega_2(d)=\omega(z_2z_2')$. Moreover, the identity element in $S_d$ corresponds to $x_N=x_0=w_1=w_1'=1$, so that restricting the range of summation so that not all of these variables are $1$ counts $\#S_d-1$ instead.  To conclude we sum the resulting expression for  $\#S_d-1$ over all squarefree $d=ad_0d_1d_2$ with $|d|\leq X$, weighted by $\gamma^{\omega_2(d)-\omega_0(d)}$, and relabel variables $(x_0,y_0,w_1,w_1',z_1,z_1',z_2,z_2')=(D_0,D_1,D_2,D_3,D_4,D_5,D_6,D_7)$.
\end{proof}
\begin{remark}
The proof above shows that the reason for excluding the terms where $\lambda=D_0=D_2=D_3=1$ in the definition of $\mathcal{S}_\gamma(\lambda,\eta,X)$ above is to remove the identity element of $S_d$ from the count.	
\end{remark}

Now fix  $1<\gamma<7/8+\sqrt{17}/8$ as in the statement of \Cref{weighted average} . In light of \Cref{big sum lemma} we want to study the sums $\mathcal{S}_\gamma(\lambda,\eta,X)$. 

\begin{defi}
As a book-keeping device, we define the function $\Phi(i,j)$ ($0\leq i \neq j\leq 7$) by setting $\Phi(i,j)=1$ if the Jacobi symbol $\left(\frac{D_i}{D_j}\right)$ appears in the definition of $\mathcal{J}_{\eta,\lambda}((D_i)_{0\leq i\leq 7})$, and $0$ else. 
We say that two indices $i$ and $j$ are \textit{linked} if $\Phi(i,j)+\Phi(j,i)=1$. 
\end{defi}

Note that the sets of linked indices are  
\begin{equation}
\{1,2\}, \{2,5\}, \{2,7\}, \{0,4\}, \{3,4\}, \{0,6\}, \{3,6\}.
\end{equation}

\begin{notation}
To write the sums $\mathcal{S}_\gamma(\lambda,\eta,X)$ in a manageable way, set $\mu_i$ to be $1$ if $\left(\frac{\eta}{D_i}\right)$ appears in $\mathcal{J}_{\eta,\lambda}((D_i)_{0\leq i\leq 7})$ and $0$ else, and set $\nu_i$ to be $1$ if $\left(\frac{\lambda}{D_i}\right)$ appears in $\mathcal{J}_{\eta,\lambda}((D_i)_{0\leq i\leq 7})$ and $0$ else. Further, define 
\[\kappa_i:=\begin{cases} \gamma~~&~~i=0,1\\ 2 ~~&~~i=2,3,4,5~~ \\\frac{2}{\gamma}~~&~~i=6,7.\end{cases}\]
  Finally, we let $\mathscr{D}(X)$ denote the set of tuples of pairwise coprime positive  integers $(D_0,...,D_7)$ such that all of the following hold:
  \begin{itemize}
  \item  we have $D_0,D_1\in \mathcal{F}_0$, $D_2,D_3,D_4,D_5\in \mathcal{F}_1$, and $D_6,D_7\in \mathcal{F}_2$,
  \item we have  $\prod_{i=0}^7D_i\leq X$,
  \item if $\lambda=1$, then $D_0,D_2$ and $D_3$ are not all $1$.
  \end{itemize}
 We thus write
\begin{equation}\label{the basic sum}
\mathcal{S}_\gamma(\lambda,\eta,X)=\sum_{(D_i)\in \mathscr{D}(X) }\prod_{i}\kappa_i^{-\omega(D_i)}\prod_{i}\left(\frac{\eta}{D_i}\right)^{\mu_i}\left(\frac{\lambda}{D_i}\right)^{\nu_i}\prod_{i\neq j}\left(\frac{D_i}{D_j}\right)^{\Phi(i,j)}.
\end{equation}
We also define $n_i$ $(0\leq i\leq 7)$ so that the $D_i$ are required to lie in   $\mathcal{F}_{n_i}$ (e.g. $n_0=n_1=0$).
\end{notation}

\subsection{Bounds on the sums $\mathcal{S}_\gamma(\lambda,\eta,X)$}

\begin{proposition} \label{bounds on the sums}
For any $1<\gamma<7/8+\sqrt{17}/8$, and for any (positive or negative) divisors $\lambda$ and $\eta$ of $N$, we have $\mathcal{S}_\gamma(\lambda,\eta,X)=o(X)$.
Moreover, when $\gamma=1/4+\sqrt{17}/4$ we have
\[\mathcal{S}_\gamma(\lambda,\eta,X)\ll X\log(X)^{-0.0394}.\]
\end{proposition}

It's immediate from \Cref{big sum lemma} that \Cref{bounds on the sums} implies \Cref{weighted average} and so, via \Cref{thing we average}, we obtain \Cref{main statistical theorem}. The rest of the section is occupied with the proof of  \Cref{bounds on the sums}.

\subsubsection{The contribution from $D_0, D_2, D_3=1$ and $\lambda=\theta$} \label{lambda eq theta case}
Recall that $K=\mathbb{Q}(\sqrt{\theta})$ for some squarefree integer $\theta$ (necessarily dividing $N$). We first  show that the contribution to $\mathcal{S}_\gamma(\theta,\eta,X)$ coming from $D_0=D_2=D_3=1$ is negligible, since leaving this in would prevent a uniform argument at a later point.  Note that when $D_0=D_2=D_3=1$  all Jacobi symbols appearing in \Cref{the basic sum} are equal to $1$ except those that involve $\lambda=\theta$. Moreover, since elements of $\mathcal{F}_2$ are products of primes inert in $K$, any $n\in \mathcal{F}_2$ has $\left(\frac{\theta}{n}\right)=\mu(n)$. On the other hand, we similarly have $\left(\frac{\theta}{n}\right)=1$ for all $n\in \mathcal{F}_1$. Consequently, the contribution to $\mathcal{S}_\gamma(\theta,\eta,X)$ from tuples with $D_0=D_2=D_3=1$ is given by
\begin{equation}\label{mobius sum}
\sum_{\substack{(D_i)\in \mathscr{D}(X)\\ D_0,D_2,D_3=1}}\mu(D_6)\prod_{i\neq 0,2,3}\kappa_i^{-\omega(D_i)}=\sum_{\substack{r\in \mathcal{F}_0\cdot \mathcal{F}_1\\ r\leq X}}\gamma^{-\omega_0(r)}\sum_{\substack{n\in \mathcal{F}_2 \\ n\leq X/r}}\gamma^{\omega(n)}\sum_{m\mid n}\mu(m).
\end{equation}
In the above, to pass from the left hand side to the right hand side we have set $r=D_1D_4D_5$ and $n=D_6D_7$, noting that e.g. given $r\in \mathcal{F}_0\cdot \mathcal{F}_1$ there are $2^{\omega_1(r)}$ ways or writing $r$ as a product $D_1D_4D_5$ where $D_1\in \mathcal{F}_0$ and $D_4,D_5\in \mathcal{F}_1$, and that this multiplicity cancels the contribution of $\kappa_4^{-\omega(D_4)}  \kappa_5^{-\omega(D_5)}$.  Now since  $\sum_{m\mid n}\mu(m)$ is equal to $0$ if $n>1$,  and $1$ if $n=1$, we find  
\begin{equation} \label{contribution from the particular bad elements}
~~\phantom{hii}|\textup{RHS of }\Cref{mobius sum}|=\sum_{\substack{r\in \mathcal{F}_0\cdot \mathcal{F}_1 \\ r\leq X}}\gamma^{-\omega_0(r)}\ll X\log(X)^{-1/2}
\end{equation} 
where for the bound we are using \Cref{shiu lemma}.

\subsubsection{Number of prime factors of the variables} \label{prime factors of the variables}
We now show that the contribution coming from $D_i$ with a large number of prime factors is negligible. This will be important in \Cref{SW family}. Set $\Omega=4e\cdot (\log\log(X)+B_0)$ with $B_0$ as in \cite{MR2276261}*{Lemma 11}, and let $\Sigma_1$ be the contribution to $\mathcal{S}_\gamma(\lambda,\eta,X)$ from the tuples $(D_i)\in \mathscr{D}(X)$ satisfying
\begin{equation}
	\omega(D_i)\geq \Omega~~\textup{ for some }0\leq i \leq 7.
\end{equation}
Writing $n=\prod_i D_i$ we have 
\begin{align*}
	|\Sigma_1|&\ll\sum_{\substack{n\leq X\\ \omega(n)\geq \Omega}}\frac{2^{\omega_0(n)}4^{\omega_1(n)}2^{\omega_2(n)}}{\gamma^{\omega_0(n)}2^{\omega_1(n)}(2/\gamma)^{\omega_2(n)}}\mu^2(n)
	\\& \ll \sum_{\substack{n\leq X\\ \omega(n)\geq \Omega}}\mu^2(n) 2^{\omega(n)}.
\end{align*}

Applying the Cauchy--Schwarz inequality and arguing using \cite{MR2280878}*{Lemma A} as in \cite{MR2276261}*{\S5.3} (paragraph above Equation (30)) we find $\Sigma_1\ll X\log(X)^{-1}$.

\subsubsection{Ranges of the variables}

We now divide the ranges of summation into intervals, and treat these intervals separately. Specifically, we set
\begin{equation}
\Delta:=1+\frac{1}{\log(X)^2}
\end{equation}
and divide the ranges of the variables into intervals $[\Delta^n,\Delta^{n+1}]$ for $n=0,1,2,...$, noting that $1$ is the only integer in the $n=0$ interval.  For $i=0,...,7$ we let $A_i$ denote a number of the form $\Delta^n$ with $1\leq \Delta^n\leq X$, let $\textbf{A} =(A_i)_{0\leq i \leq 7}$, and define
\begin{equation} \label{the sum over small intervals}
\mathcal{S}_\gamma(\lambda,\eta,X,\textbf{A})=\sum_{\substack{(D_i)\in \mathscr{D}'(X)\\ A_i\leq D_i\leq \Delta A_i}}\prod_{i}\kappa_i^{-\omega(D_i)}\prod_{i}\left(\frac{\eta}{D_i}\right)^{\mu_i}\left(\frac{\lambda}{D_i}\right)^{\nu_i}\prod_{i\neq j}\left(\frac{D_i}{D_j}\right)^{\Phi(i,j)},
\end{equation}
 where, in light of \Cref{prime factors of the variables} and \Cref{contribution from the particular bad elements}, we define $\mathscr{D}'(X)$ to be the subset of $\mathscr{D}(X)$ consisting of tuples $(D_i)_i$ such that   $\omega(D_i)\leq\Omega$ for each $i$, and such that, if $\lambda=\theta$, then not all of $D_0$, $D_2$ and $D_3$ are equal to $1$.   Since for $\alpha$ small positive we have $\log(1+\alpha)\approx \alpha$, for $X$ large  $\log(X)/\log(\Delta)\approx \log(X)^{3}$, so there are order $\log(X)^{24}$ expressions \Cref {the sum over small intervals}  as $\textbf{A}$ varies.

Following \cite{MR2276261}*{\S5.4} we split the collection of all $\textbf{A}$ into families and treat each in turn. 

\subsubsection{First family: $\prod_i A_i$ large.}

In order to exploit oscillations of the Jacobi symbols it will be necessary to allow the variables $D_i$ to range (essentially) freely in the interval $[A_i,\Delta A_i]$. To this end, we first deal with the case where the product of the $A_i$ is large, where the condition $\Pi_iD_i\leq X$ is relevant. Specifically, the first family of the $\textbf{A}$ is defined by the condition
\begin{equation} \label{first family}
\prod_{0\leq i \leq 7}A_i\geq \Delta^{-8}X.
\end{equation}

The argument here is essentially identical to that occurring between Equations (33) and (34) of \cite{MR2276261}: we have 
\begin{eqnarray*}
\sum_{\textbf{A}\textup{ satisfies }\cref{first family}}|\mathcal{S}_\gamma(\lambda,\eta,X,\textbf{A})|&\leq &\sum_{\textbf{A}\textup{ satisfies }\cref{first family}} \sum_{\substack{(D_i)\in \mathscr{D}'(X)\\ A_i\leq D_i\leq \Delta A_i }} \prod_{i}\kappa_i^{-\omega(D_i)} \\ & \leq &\sum_{\Delta^{-8}X\leq n \leq X}2^{\omega(n)} \\&\ll &(1-\Delta^{-8})X\log(X)  \\ &\ll &X \log(X)^{-1}
\end{eqnarray*}
where for the last inequality we are using that
\[1-\Delta^{-8}=1-(1+\log(X)^{-2})^{-8}=1-(1-8\log(X)^{-2}+O(\log(X)^{-4}))\ll \log(X)^{-2}.\]

Note that if $\textbf{A}$ does not satisfy \Cref{first family} then the condition $\prod_i D_i\leq X$ is made automatic by the restrictions on the intervals   the $D_i$ lie in, and may henceforth be dropped. 

\subsubsection{Second family: two large factors corresponding to linked indices}

We introduce the parameter $X^\dagger:=\log(X)^{78},$
and consider the $\textbf{A}$ such that
\begin{equation} \label{second family}
\prod_{0\leq k \leq 7}A_k\leq \Delta^{-8}X,\textup{ and there exist linked indices }i\neq j \textup{ with }A_i,A_j \geq X^\dagger.
\end{equation} 
Here the argument is almost identical to that given between Equations (40) and (42) in \cite{MR2276261}, ultimately relying on a result of Heath-Brown exploiting double oscillations of characters \cite{MR1347489}*{Corollary 4}. For such $\textbf{A}$, since $i$ and $j$ are linked we have (swapping $i$ and $j$ if necessary)
\begin{equation*} \label{big double oscillation sum}
 |\mathcal{S}_\gamma(\lambda,\eta,X,\textbf{A})|\ll \sum_{\substack{ A_k\leq D_k\leq \Delta A_k\\ k\neq i,j}}\prod_{k\neq i,j}\kappa_k^{-\omega(D_k)}\big\vert \sum_{\substack{1\leq D_i \leq \Delta A_i\\ 1\leq D_j \leq \Delta A_j}} f(D_i;(D_k)_{k\neq i,j})g(D_j;(D_k)_{k\neq i,j})\left(\frac{D_i}{D_j}\right) \big\vert,
\end{equation*}
where in the inner  sum $D_i$ and $D_j$ are odd coprime integers with no further constraints,  
\begin{equation*}
f(D_i;(D_k)_{k\neq i,j})=\mathbbm{1}_{\substack{D_i\in \mathcal{F}_{n_i},\\D_i\geq A_i, \\ \omega(D_i)\leq \Omega }}\cdot \kappa_i^{-\omega(D_i)}\mu^2\Big(D_i\prod_{k\neq i,j}D_k\Big)\left(\frac{\eta}{D_i}\right)^{\mu_i}\left(\frac{\lambda}{D_i}\right)^{\nu_i}\prod_{k\neq i,j}\left(\frac{D_i}{D_k}\right)^{\Phi(i,k)}\left(\frac{D_k}{D_i}\right)^{\Phi(k,i)}
\end{equation*}
and  $g(D_j;(D_k)_{k\neq i,j})$ is defined in the same way but with $i$ and $j$ switched. The coefficients $f(D_i;(D_k)_{k\neq i,j})$ and $g(D_j;(D_k)_{k\neq i,j})$ are complex numbers with absolute value $<1$, so applying 
 \cite{MR2276261}*{Lemma 15} (with $\epsilon=1/6$) to the inner sum above, and summing over the remaining variables, gives

\begin{equation}
|\mathcal{S}_\gamma(\lambda,\eta,X,\textbf{A})|\ll \Delta^2A_i A_j (X^\dagger)^{-1/3}\cdot \prod_{k\neq i,j}\Delta A_k   \leq  X(X^\dagger)^{-1/3}.
\end{equation}
 Summing  over each of the $\ll\log(X)^{24}$ possibilities for $\textbf{A}$   we find 
\begin{equation}
\sum_{\textbf{A}\textup{ satisfies }\cref{second family}}|\mathcal{S}_\gamma(\lambda,\eta,X,\textbf{A})|\ll X\log(X)^{-1}.
\end{equation}

\subsubsection{Third family: one large and one small factor corresponding to linked indices} \label{SW family}

We introduce a further parameter $X^{\ddag}=\textup{exp}(\log(X)^{\epsilon})$ for   fixed  $\epsilon>0$ (to be chosen later). Note that for  $X$ sufficiently large we have $X^{\ddag}>X^\dagger$. 
The family of $\textbf{A}$ we now consider is given by
\begin{equation} \label{third family}
\textup{Neither }\cref{first family}\textup{ nor }\cref{second family}\textup{ hold, and }\exists ~ i\neq j \textup{ linked with } 1<A_j<X^\dagger\textup{ and }A_i\geq X^\ddag.
\end{equation}

This section of the argument corresponds to the treatment of Fouvry--Kl\"{u}ners fourth family \cite{MR2276261}*{Equations (43) to (47)}, and we similarly obtain cancellation from the Siegel--Walfisz theorem. However, the conditions that the $D_i$ lie in the thin families $\mathcal{F}_{n_i}$ necessitate some changes and the resulting argument is modelled on \cite{MR2726105}*{\S7.5}. 

Fix such an $\textbf{A}$. In the definition of $\mathcal{S}_\gamma(\lambda,\eta,X,\textbf{A})$ we  group all terms involving $D_i$.  Since $\eta$ and $\lambda$ divide $N$, for fixed $(D_k)_{k\neq i}$ there is a Dirichlet character $\chi_{i,(D_k)_{k\neq i}}$ modulo $4N$ with
\begin{equation} \label{second oscillation cancellation sum}
\left(\frac{\eta}{D_i}\right)^{\mu_i}\left(\frac{\lambda}{D_i}\right)^{\nu_j}\prod_{k\neq i}\left(\frac{D_i}{D_k}\right)^{\Phi(i,k)}\left(\frac{D_k}{D_i}\right)^{\Phi(k,i)}=\chi_{i,(D_k)_{k\neq i}}(D_i)\prod_{k\neq i}\left(\frac{D_i}{D_k}\right)^{\Phi(i,k)+\Phi(k,i)}
\end{equation}
where in the above we are using quadratic reciprocity for Jacobi symbols. From the definition of linked indices, writing  $d:=d((D_k)_{k\neq i})=\prod_{k \textup{ linked to } i}D_k$ (which is at least $3$ by assumption), we   have  
\begin{equation} \label{preparing SW}
|\mathcal{S}_\gamma(\lambda,\eta,X,\textbf{A})|\leq \sum_{\substack{A_k\leq D_k\leq \Delta A_k\\ k\neq i}}\prod_{k\neq i}\kappa_k^{-\omega(D_k)}\left \vert \sum_{A_i\leq D_i \leq \Delta A_i} \kappa_i^{-\omega(D_i)}\chi_{i,(D_k)_{k\neq i}}(D_i)\left(\frac{D_i}{d}\right)\right\vert 
\end{equation}
where in the inner sum  $D_i$ is in $\mathcal{F}_{n_i}$ and is coprime to the $D_k$ in the outer sum, and $\omega(D_i)  \leq \Omega$.
Now $d$ is odd and coprime to $N$  so
\[D_i \mapsto \chi_{i,(D_k)_{k\neq i}}(D_i)\left(\frac{D_i}{d}\right) \]
is a primitive Dirichlet character modulo $q$ for some $q$  divisible by $d$, and dividing $4Nd$. In particular,  $3\leq q \ll (\Delta X^\dagger)^7$ since   \Cref{second family} does not hold.

 Replacing the inner sum in \Cref{preparing SW} with its maximum possible value we have
\begin{equation} \label{the main term of sw}
|\mathcal{S}_\gamma(\lambda,\eta,X,\textbf{A})|\ll \frac{X}{\Delta A_i} \cdot  \max_{a,\chi,q}\big \vert \sum_{\substack{A_i \leq D_i \leq \Delta A_i\\ (a,D_i)=1\\  D_i\in \mathcal{F}_{n_i},~ \omega(D_i)\leq \Omega}}\kappa_i^{- \omega(D_i)} \chi(D_i)\big \vert,
\end{equation}
where the maximum is taken over all $1\leq a \leq X$, all   $3\leq q \ll (\Delta X^\dagger)^7$  which contain at least one prime factor coprime to $N$ , and all primitive Dirichlet characters $\chi$ modulo $q$. Here the condition $(a,D_i)=1$ takes care of the coprimality of $D_i$ with the remaining $D_k$ . We now partition the inner sum according to the number $1\leq l\leq \Omega$ of prime factors of $D_i$, write $D_i=np$ where $p$ is the largest prime factor of $D_i$, and denote by $P^+(n)$  the largest prime factor of the remaining integer $n$, giving
\begin{align} \label{annoying sw sum}
\max_{a,\chi,q}\big \vert \sum_{\substack{A_i \leq D_i \leq \Delta A_i\\ (a,D_i)=1\\ D_i\in \mathcal{F}_{n_i},~\omega(D_i)\leq \Omega}}\kappa_i^{- \omega(D_i)} \chi(D_i)\big \vert \leq \sum_{1\leq l \leq \Omega}\sum_{\substack{n\\\omega(n)=l-1}}\max_{a,\chi,q}\big \vert \sum_{\substack{\max(P^+(n),A_i/n)<p<\Delta A_i/n\\ (a,p)=1\\p\in \mathcal{P}_{n_i}}} \chi(p)\big \vert,
\end{align}
where we allow $n$ to range over arbitrary positive integers  with $l-1$ factors. To treat the innermost sum, first note that we can drop the condition $(a,p)=1$  at the expense of adding  
\[\vert\sum_{p\mid a}\mathbbm{1}_{\mathcal{P}_{n_i}}(p)\chi(p)\vert \leq \omega(a)\ll \log(X)\]
to its value.
Next, since $K/\mathbb{Q}$ and $\mathbb{Q}(\sqrt{\alpha \beta})/\mathbb{Q}$ ramify only at primes dividing $N$, a prime $p$ is in $\mathcal{P}_{n_i}$  if and only if $p~(\textup{mod }4N)$ lies in a certain subset of $(\mathbb{Z}/4N\mathbb{Z})^\times$. In particular we may express the indicator function $\mathbbm{1}_{\mathcal{P}_i}$ as a finite sum $\sum_{s}a_s\chi_s$ where each $\chi_s$ is a Dirichlet character modulo $4N$, and the $a_s$ are real numbers.
Since the modulus $q$ of any $\chi$ appearing in \Cref{annoying sw sum} contains at least one prime not dividing $N$ (coming from $D_j$), each $\chi_s\chi$ is a primitive Dirichlet character modulo $q'$ for some $3\leq q' \ll (\Delta X^\dagger)^7$ also. By the triangle inequality and \cite{MR2276261}*{Lemma 13} (a consequence of the Siegel--Walfisz theorem) we conclude that  for all constants $A>0$   we have

\begin{align} \label{sw}
\max_{a,\chi,q}\big \vert \sum_{\substack{
\max(P^+(n),A_i/n)<p<\Delta A_i/n\\(a,p)=1
\\p\in \mathcal{P}_{n_i}}} \chi(p)\big \vert&\ll&\max_{a,\chi,q}\big \vert \sum_{
\max(P^+(n),A_i/n)<p<\Delta A_i/n
}  \chi(p)\big \vert +\log(X)\\\ &\ll_A& (X^\dagger)^4\cdot \frac{\Delta A_i}{n} \cdot \log(A_i/n)^{-A}+\log(X). \nonumber
\end{align}

Now   $n$ has at most $\Omega$ prime factors, so the   sum on the left of \Cref{sw}  is non-empty only if $n\leq \Delta A_i^{1-1/\Omega}$, in which case
\[\log(A_i/n)^{-A}\ll \log(A_i^{1/\Omega})^{-A}\ll \left(\frac{1}{\Omega}\log(X)^\epsilon\right)^{-A}\ll \log(X)^{-\epsilon A}.\]
We now insert this into \cref{sw}, and insert the result into \cref{annoying sw sum} and finally \cref{the main term of sw}, to find 
\begin{eqnarray*}|\mathcal{S}_\gamma(\lambda,\eta,X,\textbf{A})| &\ll_A& \frac{X}{\Delta A_i} \cdot \sum_{1\leq n\leq  \Delta A_i^{1-1/\Omega}}\left[ (X^\dagger)^4\cdot \frac{\Delta A_i}{n} \cdot  \log(X)^{-\epsilon A}+\log(X) \right]\\ ~~& \ll_A & X\log(X)^{1-\epsilon A}(X^\dagger)^4+\frac{X\log(X)^1}{(X^\ddag)^{1/\Omega}}.
\end{eqnarray*}
Summing over the $\ll \log(X)^{24}$ possibilities for $\textbf{A}$ and recalling that $\Omega \ll \log\log(X)$, we find 
\begin{equation*}
\sum_{\textbf{A}\textup{ satisfies }\cref{third family}}|\mathcal{S}_\gamma(\lambda,\eta,X,\textbf{A})|\ll X\log(X)^{-1} 
\end{equation*}
  provided $A$ is chosen large enough (compared to $\epsilon)$.

\subsubsection{Remaining families}  \label{Remaining families}
We now consider those $\textbf{A}$ such that 
\begin{equation} \label{main term families}
\textup{None of }\cref{first family}, \cref{second family}, \textup{ or }\cref{third family} \textup{ hold.}
\end{equation}
 Here the argument deviates significantly from that in  \cite{MR2276261}. Fix such an $\textbf{A}$, and define 
\[\mathcal{I}_\textbf{A}:=\{0\leq i \leq 7~~\mid A_i\geq X^\ddag\}.\]
Recalling that $X^\ddag>X^\dagger$ (for sufficiently large $X$), it follows from the conditions on $\textbf{A}$ that
\begin{itemize}
\item $\mathcal{I}_{\textbf{A}}$ is unlinked, 
\item if $j\notin \mathcal{I}_{\textbf{A}}$  is linked to an element of $\mathcal{I}_{\textbf{A}}$ then $A_j=1$ (so in particular, if $D_j$ is such that $A_j\leq D_j \leq \Delta A_j$, then $D_j=1$). 
\end{itemize}

 We begin by discarding as many options for $\mathcal{I}_{\textbf{A}}$ as we can simply using the trivial bound
\begin{equation}
|\mathcal{S}_\gamma(\lambda,\eta,X,\textbf{A})|\leq \sum_{\substack{(D_i)\in \mathscr{D}'(X)\\ A_i\leq D_i\leq \Delta A_i}}\prod_{i}\kappa_i^{-\omega(D_i)}.
\end{equation}
Specifically, let $I$ be any (possibly empty) set of unlinked indices, and let $i_0=|I\cap \{0,1\}|$, $i_1=|\{2,3,4,5\}\cap I|$, and $i_2=|I\cap \{6,7\}|$. Then
\begin{equation} \label{preparing for shiu}
\sum_{\substack{\textbf{A} \textup{ satisfies }\cref{main term families}\\ \mathcal{I}_{\textbf{A}}=I}}|\mathcal{S}_\gamma(\lambda,\eta,X,\textbf{A})|\leq  \sum_{n\leq (\Delta X^\ddag)^8}2^{\omega(n)}\sum_{m\leq X/n}\frac{i_0^{\omega_0(m)}}{\gamma^{\omega_0(m)}}\cdot \frac{i_1^{\omega_1(m)}}{2^{\omega_1(m)}}\cdot \frac{i_2^{\omega_2(m)}}{ (2/\gamma)^{\omega_2(m)}}.
\end{equation}
Here in the above sum, if $i_j=0$ then we interpret $i_j^{\omega_j(m)}$ as being equal to $1$ when $m$ has no prime factors in $\mathcal{P}_j$. The right hand side is derived from the left by setting $n=\prod_{i\notin I}D_i$ and $m=\prod_{i\in I}D_i$. To treat the sum on the right hand side of \Cref{preparing for shiu} we apply \Cref{shiu lemma}. Here the argument diverges according to whether $\mathbb{Q}(\sqrt{\alpha\beta})\subseteq K$ or not. Since the former, somewhat degenerate, case is easier we make the following assumption, consigning the case $\mathbb{Q}(\sqrt{\alpha\beta})\subseteq K$ to \Cref{degenerate cases remark}. 
\begin{assumption} \label{non degeneracy assumption}
Assume henceforth that  $\mathbb{Q}(\sqrt{\alpha\beta})\nsubseteq K$.
\end{assumption}
 Applying \Cref{shiu lemma} to the right hand side of \Cref{preparing for shiu} we obtain
\begin{eqnarray*}
\sum_{\substack{\textbf{A} \textup{ satisfies }\cref{main term families}\\ \mathcal{I}_{\textbf{A}}=I}}\mathcal{S}_\gamma(\lambda,\eta,X,\textbf{A})&\ll &\sum_{n\leq (\Delta X^\ddag)^8}2^{\omega(n)} \frac{X}{n}\log(X/n)^{i_0/(4\gamma)+i_1/8+\gamma i_2/4-1}\\~~&\ll& X\log(X)^{i_0/(4\gamma)+i_1/8+\gamma i_2/4-1}\sum_{n\leq (\Delta X^\ddag)^8}\frac{2^{\omega(n)}}{n}\\
~~&\ll & X\log(X)^{i_0/(4\gamma)+i_1/8+\gamma i_2/4-1+2\epsilon}
\end{eqnarray*}
with the last $\ll$ following from the bound $\sum_{n\leq Y}\frac{2^{\omega(n)}}{n}\ll \log(Y)^2$ (to prove this e.g. square the bound $\sum_{n\leq Y}\frac{1}{n}\ll \log(Y)$). We now study the exponent $i_0/(4\gamma)+i_1/8+\gamma i_2/4-1+2\epsilon$ as we vary over unlinked sets $I$. Note that $I$ is contained in one of the maximal unlinked sets of indices
\[\mathcal{I}_1:=\{2,4,6\},~\mathcal{I}_2:=\{0,2,3\},~\mathcal{I}_3:=\{0,1,3,5,7\},~\mathcal{I}_4:=\{1,4,5,6,7\}.\]
 We then have (recall that we've fixed $1<\gamma<7/8+\sqrt{17}/8$):

\begin{itemize}
\item  $I\subseteq \mathcal{I}_1$. Here $i_0=0$, $i_1\leq 2$, $i_2\leq 1$ so that 
\[i_0/(4\gamma)+i_1/8+\gamma i_2/4-1+2\epsilon\leq -1/4+2\epsilon.\]
\item  $I\subseteq \mathcal{I}_2$. Here $i_0\leq 1$, $i_1\leq 2$ and $i_2=0$ so that 
\[i_0/(4\gamma)+i_1/8+\gamma i_2/4-1+2\epsilon\leq -1/2+2\epsilon.\] 
\item  $I\subseteq \mathcal{I}_3$. Here $i_0\leq 2$, $i_1\leq 2$ and $i_2\leq 1$. Then 
\[i_0/(4\gamma)+i_1/8+\gamma i_2/4-1+2\epsilon\leq1/(2\gamma)+\gamma/4-3/4+2\epsilon=\frac{(\gamma-1)(\gamma-2)}{4\gamma}+2\epsilon.\]
Note that as $1<\gamma<2$ this is strictly negative for sufficiently small $\epsilon$.
\item $I\subsetneq \mathcal{I}_4$. Since $I$ is properly contained in $\mathcal{I}_4$ we have $i_0\leq 1$, $i_1\leq 2$, $i_2\leq 2$, and at least one of these inequalities is strict. This leads to $3$ cases. First assume that $i_0=0$. Then 
\[i_0/(4\gamma)+i_1/8+\gamma i_2/4-1+2\epsilon\leq \gamma/2-3/4+2\epsilon.\]
This is strictly negative for sufficiently small $\epsilon>0$ since $\gamma <3/2$. Next, assume that $i_1\leq 1$. Then 
\[i_0/(4\gamma)+i_1/8+\gamma i_2/4-1+2\epsilon\leq 1/(4\gamma)+\gamma/2-7/8+2\epsilon=\frac{4\gamma^2-7\gamma+2}{8\gamma}+2\epsilon.\]
The numerator has roots at $\gamma=7/8\pm \sqrt{17}/8\approx 0.36,1.39$. This is strictly negative for sufficiently small $\epsilon>0$ since $\gamma<7/8+\sqrt{17}/8$ (which is why we have chosen this upper bound on $\gamma$). The final case is when $i_2\leq 1$ where we have
\[i_0/(4\gamma)+i_1/8+\gamma i_2/4-1+2\epsilon\leq 1/(4\gamma)+\gamma/4-3/4+2\epsilon=\frac{\gamma^2-3\gamma+1}{4\gamma}+2\epsilon.\]
In the range considered, the function $\frac{\gamma^2-3\gamma+1}{4\gamma}$ is always less that its value at e.g. $2$, where it is equal to $-1/8$. 
\end{itemize}
In conclusion, for all unlinked sets $I\neq \{1,4,5,6,7\}$, choosing $\epsilon$ sufficiently small, we have 
\begin{equation}
\sum_{\substack{\textbf{A} \textup{ satisfies }\cref{main term families}\\ \mathcal{I}_{\textbf{A}}=I}}|\mathcal{S}_\gamma(\lambda,\eta,X,\textbf{A})|\ll X\log(X)^{-r_\gamma}
\end{equation}
for some $r_\gamma>0$, provided that $1<\gamma<7/8+\sqrt{17}/8$.

\begin{remark} \label{optimal gamma remark}
Optimising the exponent $r_\gamma$ over $1<\gamma<7/8+\sqrt{17}/8$, we find that the best uniform upper  bound for $i_0/(4\gamma)+i_1/8+\gamma i_2/4-1+2\epsilon$ as we range over all  unlinked sets $I\neq \{1,4,5,6,7\}$ is obtained when  $(\gamma-1)(\gamma-2)/4\gamma=(4\gamma^2-7\gamma+2)/8\gamma$, which yields $\gamma=1/4+\sqrt{17}/4$. At this choice of  $\gamma$ we have \[i_0/(4\gamma)+i_1/8+\gamma i_2/4-1+2\epsilon=\frac{1}{16}(3\sqrt{17}-13)+2\epsilon=2\epsilon-0.0394...\]
\end{remark}

\subsubsection{Completing the argument} \label{Completion of the argument}
Finally, it remains to consider $\textbf{A}$ satisfying \cref{main term families} such that $\mathcal{I}_\textbf{A}=\{1,4,5,6,7\}$. Since $\mathcal{I}_{\textbf{A}}$ is a maximal unlinked subset, the assumptions on $\textbf{A}$ force $A_0=A_2=A_3=1$ so that also $D_0=D_2=D_3=1$. Note that the definition of $\mathscr{D}'(X)$ then excludes   $\lambda=1$ or  $\lambda=\theta$.   Putting $D_0=D_2=D_3=1$ into  the definition of $\mathcal{S}_\gamma(\lambda,\eta,X,\textbf{A})$ we find
\begin{equation} \label{hopefully final case}
\mathcal{S}_\gamma(\lambda,\eta,X,\textbf{A})=\sum_{\substack{(D_i)\in \mathscr{D}'(X)\\ A_i\leq D_i\leq \Delta A}}\left(\frac{\lambda}{D_4}\right)\left(\frac{\lambda}{D_6}\right)\prod_{i}\kappa_i^{-\omega(D_i)},\end{equation}
where $A_4,A_6\geq X^{\ddag}$ by assumption. We  get cancellation in this sum via the Siegel--Walfisz theorem as in \Cref{SW family}, although unlike the previous case we must be careful of potential interaction between the conditions defining the sets $\mathcal{F}_{i}$ and the Dirichlet characters appearing. Specifically, arguing as in \Cref{SW family}, we find 
\[|\mathcal{S}_\gamma(\lambda,\eta,X,\textbf{A})|\ll \frac{X\log(X)}{\Delta A_6}\max_{1\leq a \leq X} \big \vert \sum_{\substack{A_6 \leq D_6 \leq \Delta A_6\\ (a,D_6)=1}}\kappa_6^{- \omega(D_6)} \left(\frac{\lambda}{D_6}\right)\big \vert \]
and that the inner sum satisfies
\begin{equation} \label{annoying sw sum 2}
\big \vert \sum_{\substack{A_6 \leq D_6 \leq \Delta A_6\\ (a,D_6)=1}}\kappa_6^{- \omega(D_6)} \left(\frac{\lambda}{D_6}\right)\big \vert  \leq \sum_{1\leq l \leq \Omega}\sum_{\substack{n\\\omega(n)=l-1}}\big \vert \sum_{\substack{\max(P^+(n),A_6/n)<p<\Delta A_6/n\\ (a,p)=1\\ p\in \mathcal{P}_{3}}} \left(\frac{\lambda}{p}\right)\big \vert.
\end{equation}
As before we may remove the condition $(a,p)=1$ at the expense of an acceptable error term. To treat the condition that $p\in \mathcal{P}_2$, recall that $\mathcal{P}_2$ is the set of primes coprime to $N$  which are inert in $K/\mathbb{Q}$. In particular,  the indicator function $\mathbbm{1}_{\mathcal{P}_2}(p)$ is given by $\frac{1}{2}(1-\left(\frac{\theta}{p}\right))$. Inserting this into the sum, we may apply \cite{MR2276261}*{Lemma 13}   as in \Cref{SW family} since  $\lambda \neq 1,\theta$ means that both $D\mapsto \left(\frac{\lambda  }{D}\right)$ and $D\mapsto \left(\frac{\lambda \theta}{D}\right)$ are non-principal. Continuing to argue as in \Cref{SW family} yields

\begin{equation}
\sum_{\substack{\textbf{A} \textup{ satisfies }\cref{main term families}\\ \mathcal{I}_{\textbf{A}}=\{1,4,5,6,7\}}}|\mathcal{S}_\gamma(\lambda,\eta,X,\textbf{A})|\ll X\log(X)^{-1},
\end{equation}
which completes the proof of \Cref{bounds on the sums}.

\begin{remark} \label{degenerate cases remark}
Suppose instead (of \Cref{non degeneracy assumption}) we have $\mathbb{Q}(\sqrt{\alpha\beta})\subseteq K$. This time applying \Cref{shiu lemma} to the right hand side of  \Cref{preparing for shiu}  gives the bound  
\[\sum_{\substack{\textbf{A} \textup{ satisfies }\cref{main term families}\\ \mathcal{I}_{\textbf{A}}=I}}\mathcal{S}_\gamma(\lambda,\eta,X,\textbf{A})\ll X\log(X)^{i_1/4+  i_2\gamma/4-1+2\epsilon}.\]
Splitting into cases according to which maximal unlinked subset $I$ is contained in, one finds that in the first $3$ cases, namely $I\subseteq \mathcal{I}_i$ for $i=1,2,3$, the exponent satisfies
\[i_1/4+  i_2\gamma/4-1+2\epsilon \leq -1/8+2\epsilon \]
provided $\gamma<3/2$.  In the final case where $I\subseteq \mathcal{I}_4$, we note that  $\mathbb{Q}(\sqrt{\alpha\beta})\subseteq K$ forces $\mathcal{F}_0=\emptyset$, so that $D_1$ (and also $D_0$) is necessarily equal to $1$. Thus $I\subseteq \{4,5,6,7\}$. Now provided $I\neq \{4,5,6,7\}$,  the exponent is strictly negative (for sufficiently small $\epsilon$) for $\gamma<3/2$, and is e.g. equal to $(\sqrt{17}-5)/8=-0.1096...$ if one takes $\gamma=1/4+\sqrt{17}/4$ as in \Cref{optimal gamma remark}. We are thus left to deal with the case $I=\{4,5,6,7\}$. This forces $D_2=D_3=1$ (since they are both linked to elements of $I$), in addition to $D_0=D_1=1$. One may conclude  as in \Cref{Completion of the argument}.
\end{remark}

\section{Prime twists of the congruent number curve}
\label{sec:an_example_congruent_primes_over_qq_-3}

In this section we prove \Cref{main intro theorem congruent}. That is, we provide an example of a thin subfamily of quadratic twists for which the statistical behaviour of the 2-Selmer group differs from that of the family of all twists.  In particular, there is a non-trivial Galois action in a positive proportion of cases so that, by \Cref{what happens when selc vanishes}, $\sel{\CCC_d}(\QQ, E_d[2])$ is non-trivial for a positive proportion of $d$ in our thin subfamily.

We restrict our quadratic field $K=\QQ(\sqrt{\theta})$ to be an imaginary quadratic number field which has class number 1 and in which 2 is inert (so $-\theta\in\set{3, 11, 19, 43, 67, 163}$).  Write $\mathcal{O}_{K}$ for the ring of integers of $K$, and note that the only prime which ramifies in $K$ is $-\theta$. We take
\[E:y^2=x^3-x=x(x-1)(x+1),\]
to be the congruent number curve.  This has good reduction away from $2$. Taking $p\nmid 2\theta$ to be a rational prime, we will explicitly describe the group $\sel{}^2(E_p/K)$ as a $G=\gal(K/\QQ)$-module. 

For a place $v$ of $K$, we will identify the local Kummer images $\SSS_v(E_p/K)$ of \Cref{subsec:Kummer} with their image under the 2-descent map \eqref{elliptic curve kummer map} (in our case, $a_1=0, a_2=1, a_3=-1$), so that
\[\SSS_v(E_p/K)\subseteq K_v^\times/K_v^{\times 2}\times K_v^\times/K_v^{\times 2}.\]
We  view the Selmer group  $\sel{}^2(E_p/K)$ as a subgroup of $K^\times/K^{\times 2}$   similarly, noting that this identification respects the $G$-action.  
 
For a vector space $V$ and $v_1,\dots,v_n\in V$ we write $\gp{v_1,v_2,...,v_n}$ for the subspace generated by $v_1,\dots,v_n$.

\subsection{2-Descent}
Our primary goal is to characterise the groups $\sel{}^2(E_p/K)$ for $p$ prime, which we do via 2-Descent.  We first begin by identifying the local Kummer images at each prime. 
\begin{lemma}\label{lem:p twisting local conditions}
	Let $p\nmid 2\theta$ be a prime, and let $v$ be a place of $K$.  Then the local Kummer image at $v$ for $E_p$ is given by:
	\begin{enumerate}[label=(\roman*)]
		\item If $v\mid \infty$ then
		\[\SSS_v(E_p/K)=0.\]
		\item If $v\nmid 2p$ then
		\[\SSS_v(E_p/K)=\gp{(1,u),(u,1)}\]
		where $u$ is any nonsquare unit in $K_v$.

		\item If $v\mid p$, then
		\[\SSS_v(E_p/K)=\gp{(-1,-p), (p,2)}.\]
		\item If $v=2$ and $\zeta\in K_2$ is a primitive third root of unity, then
		\[\SSS_2(E_p/K) = \gp{T_1,T_2,T_3,T_4}\]
		where
		\begin{align*}
		T_1&:=\left(-1,-p\right), &T_3&:=\left(\zeta+3, \zeta+3(1+p)\right),\\
		T_2&:=\left(1,2\right), &T_4&:=\left(1, 4\zeta+5\right).
		\end{align*}
	\end{enumerate}
\end{lemma}
\begin{proof}
	Since $K$ is imaginary, if $v\mid \infty$ the group $H^1(K_v, E[2])$ is trivial and so (i) holds.  Lemma \ref{the local conditions lemma} then provides (ii) as $p\not\in\Sigma$. In order to prove (iii), it is enough to note that by Lemma \ref{lem:No4Torsion}, since $\dim\SSS_v(E_p/K)=2$, $\SSS_v(E_p/K)=\delta_v(E_p[2])$.

	For $v=2$, note firstly that $\dim((K_2^\times/K_2^{\times 2})^2)=8$, so, as in \Cref{selmer is max iso example}, since $\SSS_p(E_p/K)$ is maximal isotropic with respect to the local Tate pairing we have $\dim\SSS_2(E_p/K)=4$.

Let $x_3=-(\zeta+3)/3$ and $x_4=-(3\zeta+2)/3$.  It is elementary to compute that 
\[x_3^3-p^2x_3\equiv -3\zeta\mod 8\quad\quad\quad x_4^3-p^2x_4\equiv \zeta^2\mod 8.\]
Since $-3,\zeta$ and $\zeta$ are all square in $K_2$, by Hensel's lemma each $x_i^3-p^2x_i$ is then also a square in $K_2$. In particular, there are $y_3,y_4$ in $K_2$ such that $P_i=(x_i,y_i)$ lies in $E_p(K_2)$ for $i=3,4$.  We then have $\delta_2(P_3)=T_3$ since $-3$ is square in $K_2$ and moreover
\begin{align*}
	\delta_2(P_4)&=\left(3\zeta+2, 3\zeta+2+3p\right).
\end{align*}
Moreover, the space generated by the $\delta_2(P)$ for $P\in E_p[2]$ is $\gp{(p,2), (-1,-p)}$. Since $K_2/\mathbb{Q}_2$ is unramified of degree $2$, $p$ is congruent to $\pm 1$ modulo $K_2^{\times 2}$, so this space is spanned by $T_1$ and $T_2$. One then checks that 
\[T_1 \cdot T_2\cdot T_4=\delta_2(P_4)\]
inside $(K_2^\times/K_2^{\times 2})^2$, so that $T_4$ is in  $\SSS_2(E_p/K)$. Since $T_1, T_2, T_3$ and $T_4$ are readily checked to be linearly independent, the result follows. 
\end{proof}

In the case that $p$ is split in $K/\QQ$, we will need to understand the image of the primes over $p$ in the localisation at $2$, for which we will use the following result. As in 
\Cref{lem:p twisting local conditions}, $p\nmid 2\theta$ is a prime, and we denote by $\zeta$ a fixed primitive $3$rd root of unity in $K_2$. For $x$ in $K$ we denote its conjugate under the action of $G$ as $\bar{x}$. 

\begin{lemma}\label{lem:EpsilonMod8}
Suppose that $p$ splits in $K/\QQ$, and write $p=\epsilon\bar{\epsilon}$ for some  $\epsilon\in \mathcal{O}_K$.  Then in $K_2^\times$ we have  
\[\epsilon\equiv \pm (\zeta+2-p)\pmod{K_2^{\times 2}}.\]
(Since $-1$ is not a square in $K_2$, precisely one of these two possibilities occurs.)
\end{lemma}
\begin{proof}
The ring of integers of $K_2$ is $\mathbb{Z}_2[\zeta]$ and by Hensel's lemma, an element of $\mathbb{Z}_2[\zeta]^\times$ is a square if and only if it is a square modulo $8$. Now using the fact that both $5$ and $\zeta=\zeta^4$ are squares in $K_2$, we find that any element of $\mathbb{Z}_2[\zeta]^\times/\mathbb{Z}_2[\zeta]^{\times 2}$ can be  written uniquely in the form 
$ a\pm \zeta$  for some $a\in \{\pm 1,\pm 5\}$ (in this representation, the trivial class is $-1-\zeta=\zeta^2$). Now writing $\epsilon \pmod{K_2^{\times 2}}$  in this form we find  that, in $K_2^\times/K_2^{\times 2}$, we have
\[p=N_{K_2/\mathbb{Q}_2}(\epsilon)=(a\pm \zeta)(a\pm \zeta^2)=2 \mp a.\]
Thus $a\equiv \pm (2-p)~(\textup{mod }8)$ and the result follows.  
\end{proof}
We are now ready to describe the Selmer groups. In the statement, all isomorphisms are as $\mathbb{F}_2[G]$-modules.

\begin{proposition}\label{prop:FullCharacterisationForPrimeTwists}
	Let $p$ be an odd prime not dividing $\theta$. Then
	\begin{enumerate}[label=(\roman*)]
		\item If $p$ is inert in $K/\QQ$ we have
		\[\sel{}^2(E_p/K)\cong \F_2^4.\]
		\item If $p$ is split in $K/\QQ$ and $\epsilon\in \mathcal{O}_K$ has norm $p$, we have
		\[\sel{}^2(E_p/K)\cong
		\begin{cases}
			\F_2^2\oplus\F_2[G]&p\equiv 5, 7\pmod 8,\\
			 
			\F_2^2&p\equiv 3\pmod 8,\\
			\F_2^2\oplus\F_2[G]^2&p\equiv 1\pmod 8\textnormal{ and }\bar{\epsilon}\in K_{\epsilon}^{\times 2},\\
			\F_2^4&p\equiv 1\pmod 8\textnormal{ and }\bar{\epsilon}\not\in K_{\epsilon}^{\times 2}.
		\end{cases}\]
	\end{enumerate}
	
\end{proposition}

\begin{proof}
Let $p\neq 2$ be inert in $K/\QQ$. Since $E_p$ has good reduction outside $2$ and $p$, the $2$-Selmer elements are units outside $2,p$.  As $K$ has class number 1  we thus want to find all $a_i,b_i\in\set{0,1}$ for which
\begin{equation}\label{eq:GenericInertPrimeTwistElement}
((-1)^{a_1}2^{a_2}p^{a_3}, (-1)^{b_1}2^{b_2}p^{b_3})
\end{equation}
lies in both of the local groups $\SSS_p(E_p/K)$ and $\SSS_2(E_p/K)$ described in \Cref{lem:p twisting local conditions}.  As $K_p/\QQ_p$ is unramified of degree 2, both $-1$ and $2$ are squares in $K_p$.  Thus all elements of the form \eqref{eq:GenericInertPrimeTwistElement} lie in $\SSS_p(E_p/K)$. We now apply the Selmer conditions at $2$. Since $p$ is odd we have $p\equiv \pm 1\pmod{K_2^{\times 2}}$. Consequently,  a global element of the form \eqref{eq:GenericInertPrimeTwistElement} which lies in $\sel{}^2(E_p/K)$ necessarily maps to the subspace of $\SSS_2(E_p/K)$ generated by $T_1=(-1,-p)$ and $T_2=(1,2)$. Restricting to elements of the form \eqref{eq:GenericInertPrimeTwistElement} which do map to this space gives
	\[\sel{}^2(E_p/K)=\gp{(p,2), (-1, -p), (1, (-1)^\delta p), ((-1)^\delta p, 1)}\cong \F_2^4\]
where $\delta=1$ if $p\not\in K_2^{\times 2}$ and $\delta=0$ otherwise.

Now  suppose   $p$ splits in $K/\QQ$, and fix $\epsilon\in K^\times$ such that $\epsilon \bar{\epsilon}=p$.  As above, the 2-Selmer elements are unramified outside $\set{2,\epsilon,\bar{\epsilon}}$, so we want to find all $a_i,b_i\in \set{0,1}$ for which
\begin{equation}\label{eq:GenericSplitPrimeTwistElement}
((-1)^{a_1}2^{a_2}\epsilon^{a_3}\bar{\epsilon}^{a_4}, (-1)^{b_1}2^{b_2}\epsilon^{b_3}\bar{\epsilon}^{b_4})
\end{equation}
lies in each of the groups $\SSS_\epsilon(E_p/K)$, $\SSS_{\bar{\epsilon}}(E_p/K)$ and $\SSS_2(E_p/K)$ described in \Cref{lem:p twisting local conditions}.  This is an elementary computation, which we do by treating each possibility for $p\pmod 8$ separately.  We repeat the local Kummer images from Lemma \ref{lem:p twisting local conditions}:
\begin{align*}
	\SSS_2(E_p/K)&=\gp{\left(-1,-p\right), \left(1,2\right), \left(\zeta+3, \zeta+3(1+p)\right), \left(1, 4\zeta+5\right)}
	\\\SSS_\epsilon(E_p/K)&=\gp{(-1,-\epsilon\bar{\epsilon}), (\epsilon\bar{\epsilon},2)},
	\\\SSS_{\bar{\epsilon}}(E_p)&=\gp{(-1,-\epsilon\bar{\epsilon}), (\epsilon\bar{\epsilon},2)}.
\end{align*}
We now break into cases. 

\underline{$\mathbf{p\equiv -1\pmod 8:}$} Here $-1$ is nonsquare in $K_\epsilon$. Replacing $\epsilon$ with $-\epsilon$ if necessary,  we  assume  $\bar{\epsilon}\in K_{\epsilon}^{\times2}$. Note also that $2$ is a square in $K_\epsilon$.  By symmetry, this gives $2,\epsilon\in K_{\bar{\epsilon}}^{\times2}$. The elements of the form \eqref{eq:GenericSplitPrimeTwistElement} which lie in $\SSS_\epsilon(E_p/K)$ are then those of the shape
\[\left((-1)^{a_1}2^{a_2}\epsilon^{a_3}\bar{\epsilon}^{a_4}, (-\epsilon)^{a_1}2^{b_2}\bar{\epsilon}^{b_4}\right).\]
Reducing further to those that satisfy the conditions of $\SSS_{\bar{\epsilon}}(E_p)$ we are left with elements of the shape
\begin{equation} \label{eq:candidate_-1_mod_8}
\left((-1)^{a_1}2^{a_2}\epsilon^{a_3}\bar{\epsilon}^{a_4}, (-\epsilon\bar{\epsilon})^{a_1}2^{b_2}\right).
\end{equation}
Finally, as $p\equiv -1 ~(\textup{mod } 8)$  we have  
\[\SSS_2(E_p/K)=\gp{\left(-1,1\right), \left(1,2\right), \left(\zeta+3, 1\right), \left(1, 4\zeta+5\right)}. \]
Since the first coordinate of each of these basis vectors has valuation $0$, we must have $a_2=0$. Further, Lemma \ref{lem:EpsilonMod8} gives $\epsilon\equiv\pm(\zeta+2-p)\equiv \pm(\zeta+3)$ in $K_2^\times/K_2^{\times 2}$, and since $\epsilon\bar{\epsilon}=p\equiv -1\pmod{K_2^{\times 2}}$ we have $\bar{\epsilon}\equiv \mp (\zeta+3)$.  It follows that each of the elements 
\[(\epsilon, 1), (\bar{\epsilon}, 1), (1, 2), (-1, -\epsilon\bar{\epsilon}) \]
are in $\sel{}^2(E_p/K)$. Since each element of the form \eqref{eq:candidate_-1_mod_8} with $a_2=0$ can be written as a linear combination of these Selmer elements, we have
\begin{align*}
\sel{}^2(E_p/K)&=\gp{(1, 2), (-1, -\epsilon\bar{\epsilon}), (\epsilon, 1), (\epsilon\bar{\epsilon}, 1)}
\\&\cong \F_2^2\oplus\F_2[G].
\end{align*}

\underline{$\mathbf{p\equiv 3\pmod 8:}$} Again, $-1$ is nonsquare in $K_\epsilon$ so we assume $\bar{\epsilon}\in K_{\epsilon}^{\times 2}$.  Additionally, $2$ is nonsquare in $K_\epsilon$, hence $-2$ is a square.  With $\epsilon$ and $\bar{\epsilon}$ swapped this all remains true.
 
The elements of the form \eqref{eq:GenericSplitPrimeTwistElement} which lie in $\SSS_\epsilon(E_p/K)$ are thus those of the shape
\[\left((-2)^{a_2}\bar{\epsilon}^{a_4}(-1)^{b_3}\epsilon^{a_3}, (-2)^{b_2}\bar{\epsilon}^{b_4}(-1)^{a_3}(-\epsilon)^{b_3}\right).\]
Reducing further to those that satisfy the conditions of $\SSS_{\bar{\epsilon}}(E_p)$ we are left with
\begin{equation}\label{eq:p=3mod8Twist}
\left((-2)^{a_2}(-1)^{b_3}(\epsilon\bar{\epsilon})^{a_3}, (-2)^{b_2}(-\epsilon\bar{\epsilon})^{b_3}(-1)^{a_3}\right).
\end{equation}

Finally, we apply the conditions at $2$. By Lemma \ref{lem:EpsilonMod8} we have $\epsilon\equiv\pm(\zeta-1)\pmod{K_2^{\times 2}}$.  As $p\equiv 3~(\textup{mod }8)$ we have
\[\SSS_2(E_p/K)=\gp{\left(-1,1\right), \left(1,2\right), \left(\zeta+3, \zeta+4\right), \left(1, 4\zeta+5\right)}.\]
Since the first coordinate of each basis element is a unit, we must have $a_2=0$. Considering the second coordinate, and noting that $\epsilon\bar{\epsilon}\equiv -1\pmod{K_2^{\times 2}}$, we find $a_3=b_2$. This leaves  a $2$-dimensional space of candidate Selmer elements. However, since the elements $(p,2)$ and $(-1,-p)$ (which correspond to the $2$-torsion points) lie in the Selmer group,  we have
\begin{align*}
\sel{}^2(E_p/K)&=\gp{(-1,-p), (p, 2)}
\\&\cong \F_2^2.
\end{align*}

\underline{$\mathbf{p\equiv 5\pmod 8:}$} Here $-1$ is square in both $K_\epsilon$ and $K_{\bar{\epsilon}}$, and $2$ is a nonsquare unit in both $K_\epsilon$ and $K_{\bar{\epsilon}}$.  We now split into two cases according to whether $\bar{\epsilon}$ is in $(K_{\epsilon}^\times)^2$.  To capture this, we fix 
\[\delta=\begin{cases}
	1&\bar{\epsilon}\not\in K_{\epsilon}^{\times 2}
	\\0&\text{else.}
\end{cases}\]
Note that if $\bar{\epsilon}\not\in K_{\epsilon}^{\times 2}$ then we necessarily have $2\equiv\bar{\epsilon}\pmod{K_{\epsilon}^{\times 2}}$.  Acting by $\gal(K/\QQ)$, we  see that $\bar{\epsilon}$ is in $(K_{\epsilon}^\times)^2$ if and only if $\epsilon$ is in $(K_{\bar{\epsilon}}^\times)^2$.

The elements of the form \eqref{eq:GenericSplitPrimeTwistElement} which lie in $\SSS_\epsilon(E_p/K)$ are thus those of the shape
\[\left((-1)^{a_1}(2^\delta\bar{\epsilon})^{a_4}(2^\delta\epsilon)^{a_3},(-1)^{b_1}(2^\delta\bar{\epsilon})^{b_4} 2^{a_3}(2^\delta\epsilon)^{b_3}\right).\]
	Reducing further to those that lie in $\SSS_{\bar{\epsilon}}(E_p)$ forces $a_3=a_4$, leaving those of the shape
	\begin{equation}\label{eq:p=5mod8Twistnonsqeps}
	\left((-1)^{a_1}(\epsilon\bar{\epsilon})^{a_3},(-1)^{b_1}(2^\delta\bar{\epsilon})^{b_4} 2^{a_3}(2^\delta\epsilon)^{b_3}\right).
	\end{equation}
Finally, we apply the conditions at $2$. By Lemma \ref{lem:EpsilonMod8} we have $\epsilon\equiv \pm (\zeta-3)\equiv \mp (4\zeta+5)\pmod{K_2^{\times 2}}$.  Moreover, as $p\equiv 5~(\textup{mod }8)$ we have
\[\SSS_2(E_p/K)=\gp{\left(-1,-1\right), \left(1,2\right), \left(\zeta+3, \zeta+2\right), \left(1, 4\zeta+5\right)}.\]
Since $\epsilon\bar{\epsilon}=p\equiv 1\pmod{K_2^{\times 2}}$, we have $\bar{\epsilon}\equiv \epsilon\equiv \mp (4\zeta+5)$.   Thus the   elements
\[(-1,-1), (\epsilon\bar{\epsilon}, 2),(1,\mp  2^\delta\epsilon),(1,\mp  2^\delta\bar{\epsilon})\]
all lie in $\in \sel{}^2(E/K)$, and are visibly linearly independent. Noting that $(1,-1)$ is not in $\SSS_2(E_p/K)$, we conclude that
\begin{align*}
\sel{}^2(E_p/K)&=\gp{(-1,-1), (\epsilon\bar{\epsilon}, 2),(1,\mp  2^\delta\epsilon),(1,\mp  2^\delta\bar{\epsilon})}
\\&\cong\F_2^2\oplus\F_2[G].
\end{align*}

\underline{$\mathbf{p\equiv 1\pmod 8:}$} Here both $-1$ and $2$ are squares in both $K_\epsilon$ and $K_{\bar{\epsilon}}$.    As before, set
\[\delta=\begin{cases}
	1&\bar{\epsilon}\not\in K_{\epsilon}^{\times 2}
	\\0&\text{else.}
\end{cases}\]
The elements of the form \eqref{eq:GenericSplitPrimeTwistElement} which lie in $\SSS_\epsilon(E_p/K)$ are those of the shape
\begin{equation}\label{eq:p=1mod8Twisteps}
	\left((-1)^{a_1}2^{a_2}(\epsilon\bar{\epsilon}^\delta)^{c_1}\bar{\epsilon}^{(1-\delta)c_2},(-1)^{b_1}2^{b_2}(\epsilon\bar{\epsilon}^\delta)^{d_1}\bar{\epsilon}^{(1-\delta)d_2} \right),
\end{equation}
for some $c_1,c_2,d_1,d_2$ in $\{0,1\}$.
For either value of $\delta$ these elements all   lie in $\SSS_{\bar{\epsilon}}(E_p)$.

Finally, we apply the conditions at $2$.  By Lemma \ref{lem:EpsilonMod8} we have 
\[\epsilon\equiv \pm ( -\zeta-1)\equiv \pm \zeta^2\equiv \pm 1 \pmod{K_2^{\times 2}},\] 
and as $\epsilon \bar{\epsilon}=p$ with have $\epsilon \equiv  \bar{\epsilon}\pmod{K_2^{\times 2}}$. Moreover, with $p\equiv 1~(\textup{mod }8)$ we have
\[\SSS_2(E_p/K)=\gp{\left(-1,-1\right), \left(1,2\right), \left(\zeta+3, \zeta+6\right), \left(1, 4\zeta+5\right)}.\]
As the first coordinate of each of these basis elements has trivial valuation, we have $a_2=0$.

Suppose that $\delta=1$. Then we see that an element of the form\Cref{eq:p=1mod8Twisteps} is in the Selmer group if and only if, in addition to $a_2=0$, we have   $a_1=b_1$. Thus we find
\begin{align*}
 	\sel{}^2(E_p/K)&=\gp{(-1,-1), (1,2), (1,\epsilon\bar{\epsilon}), (\epsilon\bar{\epsilon},1)}
 	\\&\cong \F_2^4.
\end{align*} 

Now suppose that $\delta=0$. Setting $a_2=0$ in \eqref{eq:p=1mod8Twisteps}  leaves a $7$-dimensional space of candidate Selmer elements. Further, one readily checks that $(-1,1)$, which has the form \eqref{eq:p=1mod8Twisteps} for $a_1=1$ and all other variables $0$, is not in $\SSS_2(E_p/K)$. Thus $\sel{}^2(E_p/K)$ is at most $6$ dimensional.  However, using the fact that $\epsilon\equiv \bar{\epsilon}\equiv  \pm 1\pmod{K_2^{\times 2}}$, one readily checks that the $6$ linearly independent elements 
\[\{(-1,-1), (1,2), (1,\pm \epsilon), (\pm \epsilon,1), (1, \pm\bar{\epsilon}), (\pm\bar{\epsilon},1)\},\] 
each of which are of the form \eqref{eq:p=1mod8Twisteps}, map to $\SSS_2(E_p/K)$ after localising at $2$. Thus,
\begin{align*}
 	\sel{}^2(E_p/K)&=\gp{(-1,-1), (1,2), (1,\pm \epsilon), (\pm \epsilon,1), (1, \pm\bar{\epsilon}), (\pm\bar{\epsilon},1)}
 	\\&\cong \F_2^2\oplus \F_2[G]^2.
\end{align*}
This completes the proof.
\end{proof}

\begin{remark}
	The proof of part (i) shows that the conditions at inert primes impose no restrictions. Using this observation, one sees similarly that if $d$ is odd and divisible only by inert primes, then
	\[\sel{}^2(E_d/K)\cong \F_2^{2+2\omega(d)}.\]
	This gives a concrete instance of the growth of $\sel{}^2(E_d/K)$ seen also in e.g.\Cref{Genus Theory as 2-Torsion Frobenii}.
\end{remark}

\subsection{Statistics}
Here we use R\'{e}dei symbols alongside the Chebotarev density theorem to determine the statistical behaviour of $\sel{}^2(E_p/K)$ from Proposition \ref{prop:FullCharacterisationForPrimeTwists}.  We refer the reader to \cite{stevenhagen2018redei} for definitions concerning  R\'{e}dei symbols.
\begin{lemma}\label{lem:R\'{e}deiDeterminesSquareSplitting}
Let   $p\equiv 1\pmod 8$ be a prime which splits in $K/\mathbb{Q}$, and let $\epsilon\in \mathcal{O}_K$ have norm $p$.   Then $\bar{\epsilon}\in (K_{\epsilon}^\times)^2$ if and only if the R\'edei symbol $[\theta,-\theta,p]$ is trivial.
\end{lemma}
\begin{proof}
 Note that $-1$ is a square in $K_\epsilon$ since $p\equiv 1\pmod 8$. In particular, the statement is unchanged upon replacing $\epsilon$ with $-\epsilon$. By \Cref{lem:EpsilonMod8} we may thus assume that we have
 \[\bar{\epsilon}\equiv -(\zeta+1)=\zeta^2\equiv 1\pmod{K_2^{\times 2}}.\]
	Now  consider the diagram of fields
	\[\begin{tikzpicture}
		\node at (0,3) (F) {$F=\QQ(\sqrt{\theta}, \sqrt{\epsilon}, \sqrt{\bar{\epsilon}})$};
		\node at (-2, 2) (KrtEps) {$K(\sqrt{\bar{\epsilon}})$};
		\node at (1, 2) (E) {$L=\QQ(\sqrt{\theta}, \sqrt{p})$};
		\node at (-1, 1) (K) {$K=\QQ(\sqrt{\theta})$} ;
		\node at (2, 1) (Kp) {$K'=\QQ(\sqrt{\theta p})$};
		\node at (0, 0) (Q) {$\QQ$};

		\draw (Q) -- (Kp);
		\draw (Q) -- (K);
		\draw (K) -- (E);
		\draw (K) -- (KrtEps);
		\draw (Kp) -- (E);
		\draw (KrtEps) -- (F);
		\draw (E) -- (F);
		\draw (Q) -- (Kp);
	\end{tikzpicture}\]
Since $\epsilon$ ramifies in $L/K$, we see that $\bar{\epsilon}\in (K_{\epsilon}^\times)^2$ if and only if the unique prime of $L$ lying over $\epsilon$ splits in $F/L$. Let $\mathfrak{p}$ denote the unique prime of $K'$ lying over $p$. Since $p$ splits in $K/\mathbb{Q}$, we see that $\mathfrak{p}$ splits in $L/K'$.   Further, $\bar{\epsilon}$ ramifies in $L/K$ and hence has even valuation (either 0 or 2) at any prime $\mathfrak{p}'\mid \mathfrak{p}$ of $L$. In particular, the extension $F=L(\sqrt{\bar{\epsilon}})/L$ is unramified at such $\mathfrak{p}'$.  Thus $F/K'$ is unramified at $\mathfrak{p}$.  We now conclude that $\bar{\epsilon}\in (K_{\epsilon}^\times)^2$ if and only if the Artin symbol $\art{F/K'}{\mathfrak{p}}$ is trivial. Before relating this to a R\' {e}dei symbol, it will be useful to prove the following two claims. 

\textbf{Claim 1: The field $F/K'$ is everywhere unramified.}  That $F'/K'$ is unramified at primes not dividing $2p\theta$ is clear, and we have already shown that the unique prime of $K'$ dividing $p$ is unramified in $F'/K'$.  For primes over $2$ note that $K$ and $K'$ are unramified at 2, and so $L/\QQ$ is unramified at $2$ also.  Further, having chosen $\bar{\epsilon}$ to be a square in $K_2$, the extension  $K(\sqrt{\bar{\epsilon}})/K$ is split at 2. Thus, as the compositum of $\K(\sqrt{\bar{\epsilon}})$ and $L$,  the full extension $F/\mathbb{Q}$ is unramified at $2$.  Now note that  $\ell=-\theta$ is an odd prime. Since $p$ has trivial $l$-adic valuation, the extension $F=K'(\sqrt{p},\sqrt{\epsilon})/K'$ is unramified at (the unique prime of $K'$ over) $l$. This proves the claim. 

\textbf{Claim 2: For each prime  $q$, the Hilbert symbols $(p,\theta)_q$ and $(p,p)_q$ are trivial.}
By assumption, $p$ is a norm from $K=\mathbb{Q}(\sqrt{\theta})$, so that $(p,\theta)_q$ is trivial for all $q$.  Next, for each prime $q$ we have $(p,p)_q=(p,-1)_q$. That this latter symbol is trivial for $q\neq 2,p$ is immediate, whilst for $q=2,p$ it is trivial since $p\equiv 1\pmod 8$. This proves the claim.
 
Returning to the proof, by Claim 2 the  R\'edei symbol $[\theta,p,p]$ exists (see \cite{stevenhagen2018redei}*{Definition 7.8}). Writing $\epsilon=x+y\sqrt{\theta}$ for $x,y$ in $\mathbb{Q}$, we have $x^2-\theta y^2=p$ by assumption. The field $F$ is then given by adjoining to $L$ the element 
\[\sqrt{\epsilon}=\sqrt{x+y\sqrt{\theta}}.\]
Further,  by Claim 1 the extension $F/K'$ is minimally ramified in the sense of \cite{stevenhagen2018redei}*{Definition 7.6}. Thus we may take  $a=\theta$, $b=p$ and  $F_{a,b}=F$ in \cite{stevenhagen2018redei}*{Definition 7.8}, giving $[\theta,p,p]=\art{F/K'}{\mathfrak{p}}$.   Consequently, we see that $\bar{\epsilon}\in K_{\epsilon}^{\times 2}$ if and only if the R\'edei symbol $[\theta,p,p]$ is trivial.

By \cite{stevenhagen2018redei}*{Proposition 7.10} the R\'edei symbol $[p,\theta,-\theta p]$ exists and is trivial (to see that $\theta p$ is a second kind decomposition, use \cite{stevenhagen2018redei}*{Prop 4.2 (4)} and our computations of Hilbert symbols above).  Now, using the trilinearity and reciprocity of R\'edei symbols \cite{stevenhagen2018redei}*{Theorem 1.1} we have
\begin{align*}
	[\theta,p,p]&=[p,\theta,p]+[p,\theta,-\theta p]
	\\&=[p,\theta,-\theta]
	\\&=[\theta,-\theta,p]
\end{align*}
as required.
\end{proof}

This allows us to give a complete statistical description of the $\mathbb{F}_2[G]$-module $\sel{}^2(E_p/K)$. First we introduce some notation. 

\begin{notation}
For $p$ a prime, we define non-negative integers $e_1(E_p/K)$ and $e_2(E_p/K)$ such that we have a $G$-module isomorphism
\[\textup{Sel}^2(E_p/K)\cong \F_2^{  e_1(E_p/K)}\oplus\F_2[G]^{  e_2(E_p/K)}.\]
\end{notation}

\begin{theorem} \label{thm:density_for_CNC}
	The  density of primes $p$ for which $e_1(E_p/K)=e_1$ and $e_2(E_p/K)=e_2$ is as follows:
	\[\lim_{X\to\infty}\frac{\#\set{p\leq X\textnormal{ prime}~:~e_1(E_p/K)=e_1\textnormal{ and }e_2(E_p/K)=e_2}}{\#\set{p\leq X\textnormal{ prime}}}=
	\begin{cases}
		9/16&\textnormal{if }(e_1,e_2)=(4,0),
		\\1/16&\textnormal{if }(e_1,e_2)=(2,2),
		\\1/4&\textnormal{if }(e_1,e_2)=(2,1),
		\\1/8&\textnormal{if }(e_1,e_2)=(2,0).
	\end{cases}\]
\end{theorem}
\begin{proof}
	As a consequence of  Lemma \ref{lem:R\'{e}deiDeterminesSquareSplitting}, and the Chebotarev density theorem applied to Proposition \ref{prop:FullCharacterisationForPrimeTwists},   it suffices to show that $[\theta,-\theta,p]$ is trivial for precisely half  of the primes   $p \equiv 1\pmod8$ which 
	split in $K/\QQ$ (with respect to the natural density).

Fix a prime $p\nmid 2\theta$. In the notation of \cite{stevenhagen2018redei}*{Definitions 7.6, 7.8}, let $F_{\theta,-\theta}$ be minimally ramified over $\mathbb{Q}(\sqrt{\theta},\sqrt{-1})$, so that by definition the R\'{e}dei symbol $[\theta,-\theta,p]$ is equal to the Artin symbol 
\begin{equation} \label{eq:displayed_artin_symbol}
 \art{F_{\theta,-\theta}/\QQ(\sqrt{-1})}{\mathfrak{p}},
 \end{equation}
where $\mathfrak{p}$ is any ideal of $\mathbb{Q}(\sqrt{-1})$ of norm $p$.  The field $F_{\theta,-\theta}$ is a cyclic degree 4 extension of $\QQ(\sqrt{-1})$ fitting into the diagram below. It is dihedral of degree $8$ over $\mathbb{Q}$ and contains $\mathbb{Q}(\sqrt{\theta},\sqrt{-1})$ as a subfield.
	
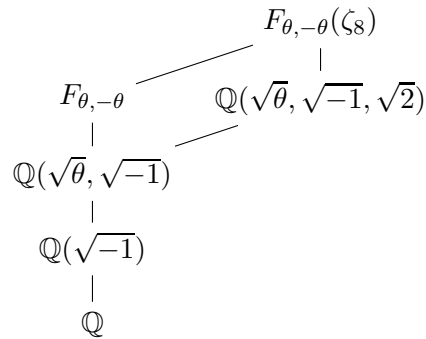
\[\begin{tikzpicture}
		\node at (3, 3) (E8) {$\QQ(\sqrt{\theta}, \sqrt{-1},\sqrt{2})$};
		\node at (3, 4) (F8) {$ F_{\theta,-\theta}(\zeta_8)$};

		\node at (0,3) (F) {$F_{\theta,-\theta}$};
		\node at (0, 2) (E) {$\QQ(\sqrt{\theta}, \sqrt{-1})$};
		\node at (0, 1) (QQ-1) {$\QQ(\sqrt{-1})$} ;
		\node at (0, 0) (Q) {$\QQ$};

		\draw (Q) -- (QQ-1);
		\draw (QQ-1) -- (E);
		\draw (E) -- (F);
		\draw (F) -- (F8);
		\draw (E) -- (E8);
		\draw (E8) -- (F8);
 	\end{tikzpicture}\]
 	The field $F_{\theta,-\theta}(\zeta_8)/\mathbb{Q}$ is Galois of degree $16$.  Now $p$ both splits in $K/\mathbb{Q}$ and is congruent to $1$ modulo $8$ if and only if it splits completely in $\mathbb{Q}(\sqrt{\theta},\sqrt{-1},\sqrt{2})=\QQ(\sqrt{\theta},\zeta_8)$. On the other hand, the Artin symbol \eqref{eq:displayed_artin_symbol} is trivial if and only if $p$ splits completely in $F_{\theta,-\theta}$. 
 	
 	Consequently, we wish to compute the density of primes which split completely in $F_{\theta,-\theta}(\zeta_8)$, amongst those that split completely in $\mathbb{Q}(\sqrt{\theta},\zeta_8)$.  By the Chebotarev density theorem, this is equal to $1/2$.
\end{proof}

\bibliographystyle{plain}

\bibliography{references}

\end{document}